\documentclass[a4paper,10pt]{article}

\usepackage{amsmath}
\usepackage{amscd}
\usepackage{amsfonts}
\usepackage{stmaryrd}
\usepackage{mathrsfs}
\usepackage{amsthm}
\usepackage{ dsfont }
\usepackage{latexsym, amssymb,amsthm,amsmath,lscape,epsfig,setspace,tikz,slashed}
\usepackage[retainorgcmds]{IEEEtrantools}
\usetikzlibrary{decorations.markings}
 \usetikzlibrary{arrows}
\usepackage[pdftex]{hyperref}
\usepackage[margin=10pt,font=small,labelfont=bf]{caption}
\usepackage{a4wide}
\usepackage[all,cmtip]{xy}
\usepackage{graphicx}
\usepackage{stackrel}

\newcommand{\mfg}{\mathfrak{g}}

\newcommand{\R}{\mathbb{R}}

\newcommand{\J}{\mathcal{J}}

\newcommand{\T}{\mathbb{T}}

\newcommand{\mcD}{\mathcal{D}}
\newcommand{\C}{\mathbb{C}}

\newcommand{\mcL}{\mathcal{L}}
\newcommand{\mcH}{\mathcal{H}}

\newcommand{\mcO}{\mathcal{O}}

\newcommand{\lb}{\llbracket}
\newcommand{\rb}{\rrbracket}

\newcommand{\del}{\partial}
\newcommand{\delbar}{\bar{\partial}}

\newcommand{\tL}{\widetilde{L}}

\newcommand{\ve}{\varepsilon}
\newcommand{\tve}{{\widetilde{\varepsilon}}}
\newcommand{\oM}{\overline{M}}
\newcommand{\mcF}{\mathcal{F}}
\newcommand{\mcC}{\mathcal{C}}

\numberwithin{equation}{section}

\theoremstyle{plain}
\newtheorem{thm}{Theorem}[section]
\newtheorem{prop}[thm]{Proposition}

\theoremstyle{definition}
\newtheorem{lem}[thm]{Lemma}
\newtheorem{defn}[thm]{Definition}
\newtheorem{defn/thm}[thm]{Definition/Theorem}
\newtheorem{ex}[thm]{Example}
\newtheorem{cor}[thm]{Corollary}

\theoremstyle{remark}
\newtheorem{rem}[thm]{Remark}

\begin{document}

\title{Hodge decompositions for Lie algebroids on manifolds with boundary}

\author{
J.L.van der Leer Dur\'an\thanks{{\tt joeyvdld@math.toronto.edu}}\\
University of Toronto\\
}

\date{\vspace{-6ex}}

\maketitle

\abstract{\noindent We investigate when the Chevalley-Eilenberg differential of a complex Lie algebroid on a manifold with boundary admits a Hodge decomposition. We introduce the concepts of Cauchy-Riemann structures, elliptic and non-elliptic boundary points and Levi-forms, which we use to define the notion of $q$-convexity. We show that the Chevalley-Eilenberg complex of an elliptic, $q$-convex Lie algebroid admits a Hodge decomposition in degree $q$. This generalizes the well-known Hodge decompositions for the exterior derivative on real manifolds and the delbar-operator on $q$-convex complex manifolds. 
We establish the results in a more general setting, where the differential does not necessarily square to zero and moreover varies in a family, including an analysis of the behaviour on the deformation parameter. As application we give a proof of a classical holomorphic tubular neighbourhood theorem (which implies the Newlander-Nirenberg theorem) based on the Moser trick, and we provide a finite-dimensionality result for certain holomorphic Poisson cohomology groups.

\vskip12pt

\tableofcontents

\addcontentsline{toc}{section}{Introduction}

\section*{Introduction}

On a compact, complex manifold without boundary the $\delbar$-operator can be studied in the framework of elliptic operator theory, leading to Hodge decompositions on the level of Hilbert spaces (Weil \cite{MR0111056}). If the manifold has a boundary however, one is faced with a non-elliptic boundary value problem that is considerably more intricate. Already in the case of a bounded domain in $\C^n$ with smooth boundary, the image of the $\delbar$-operator is not necessarily closed (c.f.\ \cite{MR0372241}), illustrating that one needs to impose suitable conditions on the complex manifold in order for it to admit Hodge decompositions. The study of this problem, also referred to as the \lq\lq$\delbar$-Neumann problem\rq\rq\, was eventually solved by Kohn \cite{MR0208200} (see \cite{MR0461588} for an historical overview). As it turns out, the analytical properties of the $\delbar$-operator are closely related to the behaviour of a local invariant of the boundary called the Levi-form (\cite{Levi1910}). A sufficient criterion for having a Hodge decomposition in a specific Dobleault-degree is then given by a pointwise eigenvalue property of the Levi-form. In \cite{MR0461588}, Folland and Kohn give a detailed explanation of this fact and provide several applications, including a proof of the Newlander-Nirenberg theorem (\cite{MR0088770}) and of Grauert's solution of the Levi-problem (\cite{MR0098847}).
\newline
\newline
In this paper we extend the fundamental work of Kohn to the more general setting of complex Lie algebroids. First introduced by Pradines \cite{MR0216409}, Lie algebroids play an important role in contemporary differential geometry. They generalize both Lie algebras and tangent bundles, form the infinitesimal objects associated to Lie groupoids, and provide a unified framework for several geometric structures including symplectic, Poisson, complex, and generalized complex structures. 
All the structural data of a Lie algebroid is contained in its Chevalley-Eilenberg differential, a differential on the exterior algebra of the dual of the Lie algebroid. 
To develop Hodge theory for this differential we will consider Lie algebroids whose differential defines an elliptic complex, a condition that is necessary also when there is no boundary. Such an elliptic Lie algebroid comes with a Cauchy-Riemann structure, as is the case in complex geometry. However, for general Lie algebroids a new phenomenon appears: there turn out to be two types of boundary points, which we call elliptic and non-elliptic, depending on a certain transversality property of the algebroid at the boundary. The elliptic boundary points in our terminology turn out to be exactly the points where the associated Neumann boundary value problem is elliptic in the sense of Lopatinsky \cite{Lop} (see also \cite{MR0161012}). For example, for the exterior derivative all boundary points are elliptic and for the $\delbar$-operator all boundary points are non-elliptic, while in general there will be a mixture of elliptic and non-elliptic boundary points. At non-elliptic points we define the Levi form, a conformal class of Hermitian bilinear forms on the Cauchy-Riemann structure, and use it to introduce the notion of $q$-convexity. Our main theorem states that for a $q$-convex elliptic Lie algebroid, there is a Hodge decomposition for the Chevalley-Eilenberg differential in degree $q$.   
\newline
\newline
Our main application of these Hodge decompositions will appear in \cite{NT} and concerns neighbourhood theorems in generalized complex geometry, a type of geometry introduced by Hitchin \cite{MR2013140} and Gualtieri \cite{MR2811595}. For this application a more general set-up than the one described above is required. We will state the main theorem for pre-Lie algebroids, meaning that we no longer require the differential to square to zero, incorporating e.g.\ almost complex structures. In this context it still makes sense to study the Laplacian of the Chevalley-Eilenberg operator, which we require to be elliptic. 
Furthermore, we will consider families of pre-Lie algebroid structures and analyze how the associated family of Neumann operators (usually called Green's operators on manifolds without boundary) behave with respect to the deformation parameter. The resulting \lq\lq Leibniz-type\rq\rq\ estimates in this context are crucial for the applications in \cite{NT}, where they form the input of an algorithm that originates from the work of Nash \cite{MR0075639} and Moser \cite{MR0199523}. 
We also prove a $C^1$-statement for the Neumann operators varying in a family, allowing us to give a simple proof (based on the Moser trick) of a classical neighbourhood theorem in complex geometry, that contains the Newlander-Nirenberg theorem \cite{MR0088770} as a special case. To the current knowledge of the author, this method of proof is not yet available in the existing literature.  Finally, we discuss in some detail holomorphic Poisson structures. In this setting there generally are elliptic and non-elliptic boundary points, and the Levi-form depends both on the geometry of the underlying complex manifold as well as on the behaviour of the Poisson structure. We use our main theorem to give a finite-dimensionality result for certain Poisson cohomology groups, which is an example of a result that does not follow from the Hodge theory of the underlying complex manifold.

\
\newline
\noindent \textbf{Organization:} Section 1 gives a short introduction to Lie algebroids. We define Cauchy-Riemann structures, elliptic and non-elliptic boundary points and the Levi-form, which we use to introduce the notion of $q$-convexity. We state our main result, concerning Hodge decompositions for the Chevalley-Eilenberg complex, and present as applications a proof of a holomorphic tubular neighbourhood theorem using the Moser trick, and a finite-dimensionality result for Poisson cohomology.

Section 2 provides the functional analytical set-up of linear first-order partial differential operators on manifolds with boundary. We recall the notion of elliptic regularity and its relation to Hodge decompositions, and we state a general theorem that characterizes differential operators satisfying elliptic regularity. We also study some aspects of Hodge decompositions varying in families. 

Section 3 is devoted to the proof of the main result stated in Section 1. We establish the proof by showing that the Lie algebroid differential satisfies the conditions of the more general theorem of Section 2 in degree $q$ if the Lie algebroid is $q$-convex.    
\newline
\newline
\noindent \textbf{Acknowledgments:} The author is thankful to Micheal Bailey, Gil Cavalcanti and Marco Gualtieri for useful conversations. This research is supported by NSERC Discovery Grant 355576.

\section{Lie algebroids}\label{10:41:57}


\subsection{Definitions and examples}

Let $M$ be a compact manifold with boundary $\del M$ and interior $\mathring{M}$. We write $C^\infty(E)$ for the space of smooth sections of a vector bundle $E$ over $M$. 
All vector bundles will be assumed to be complex, with $E_\C$ denoting the complexification of a real bundle $E$.  
\begin{defn}
A \textsl{pre-Lie algebroid} on $M$ is a vector bundle $L$ equipped with a derivation 
\begin{align*}
d_L:C^\infty(\Lambda^\bullet L^\ast)\rightarrow C^\infty(\Lambda^{\bullet+1} L^\ast).
\end{align*} 
If $d_L^2=0$ then $L$ is called a \textsl{Lie algebroid}.
\end{defn}
\noindent Since $d_L$ is a derivation, it is determined by the two maps $d_L:C^\infty(M)\rightarrow C^\infty(L^\ast)$ and $d_L:C^\infty(L^\ast)\rightarrow C^\infty(\Lambda^{2} L^\ast)$. The first induces, after dualizing, an \textsl{anchor map} $\rho:L\rightarrow TM_\C$, while the second induces a bracket $[\cdot, \cdot ]:C^\infty(L)\times C^\infty(L)\rightarrow C^\infty(L)$ defined by 
\begin{align*}
\alpha([u,v]) :=-d_L\alpha (u,v)+\rho(u)\cdot \alpha(v)-\rho(v)\cdot \alpha(u),
\end{align*}
for $u,v\in C^\infty(L)$, $\alpha\in C^\infty(L^\ast)$. The relation between $[\cdot,\cdot]$ and $\rho$ is given by the Leibniz rule 
\begin{align*}
[u,fv]=f[u,v]+(\rho(u)\cdot f)v,
\end{align*}
for $u,v\in C^\infty(L)$, $f\in C^\infty(M)$. Conversely, a bracket $[\cdot,\cdot]$ and an anchor $\rho$ satisfying the above Leibniz rule together define a derivation $d_L$ by
\begin{align*}
d_L\alpha(u_0,\ldots, u_{k}):=&\sum_i (-1)^i \rho(u_i)\cdot \alpha(u_0,\ldots,\widehat{u_i},\ldots, u_{k})\\&+\sum_{i<j} (-1)^{i+j}\alpha([ u_i,u_j],u_0,\ldots ,\widehat{u_i},\ldots,\widehat{u_j},\ldots,u_{k}), 
\end{align*}
where $u_i\in C^\infty(L)$, $\alpha\in C^\infty(\Lambda^kL^\ast)$. Hence, specifying $d_L$ is equivalent to giving the data $([\cdot,\cdot],\rho)$, and one can verify directly that $d_L^2=0$ if and only if $[\cdot,\cdot]$ satisfies the Jacobi identity and $\rho:C^\infty(L)\rightarrow C^\infty(TM_\C)$ preserves brackets\footnote{In fact, one can show that the Jacobi identity for $[\cdot,\cdot]$ itself already implies that $\rho$ is bracket preserving.}. 
\begin{ex}
The tangent bundle $L=TM_\C$ with the usual Lie bracket and identity anchor is a Lie algebroid. The associated differential is the usual exterior derivative.
\end{ex}
\begin{ex}\label{08:32:47}
Let $L=T^{0,1}M\subset TM_\C$ be the $(-i)$-eigenbundle of an almost complex structure $I$ on $M$. Then $d_L=\delbar$ 
endows $L$ with the structure of a pre-Lie algebroid. The anchor is given by the inclusion and the bracket by the $(0,1)$-projection of the ordinary Lie bracket. Here $d_L^2=0$ if and only if $I$ is integrable.  
\end{ex}
\noindent More examples arise from Dirac geometry. Consider the bundle $\T M_\C:=TM_\C\oplus T^\ast M_\C$, endowed with the split-signature metric (called the \textsl{natural pairing}) 
\begin{align*}
\langle X+\xi,Y+\eta  \rangle:=\frac{1}{2}(\xi(Y)+\eta(X)), 
\end{align*}
and a bracket on its space of sections called the \textsl{Courant bracket}, given by
\begin{align*}
\lb X+\xi,Y+\eta\rb:=[X,Y]+\mcL_X\eta-\iota_Yd\xi-\iota_Y\iota_XH,
\end{align*}
for $X,Y\in C^\infty(TM_\C)$ and $\xi,\eta\in C^\infty(T^\ast M_\C)$. Here $H$ is a real closed $3$-form on $M$ that is fixed throughout. The Courant bracket satisfies the Jacobi identity but is not skew-symmetric.
\begin{defn}
An \textsl{almost Dirac structure} on $M$ is a subbundle $L\subset \T M_\C$ which is Lagrangian with respect to $\langle \cdot,\cdot \rangle$. If $L$ is \textsl{integrable} (or \textsl{involutive}), i.e.\ if $\lb C^\infty(L),C^\infty(L)\rb\subset C^\infty(L)$, then $L$ is called a \textsl{Dirac structure}.  
\end{defn}
\noindent An almost Dirac structure $L$ is equipped with an anchor map $\rho : L\rightarrow TM_\C$ induced by the projection of $\T M_\C$ to $TM_\C$. If $\tL$ is another almost Dirac structure which is complementary to $L$ in the sense that $\T M_\C=L\oplus \tL$, we obtain a bracket on $C^\infty(L)$ given by
\begin{align*}
[u,v]:=\lb u,v\rb^L, 
\end{align*}
where $(\cdot)^L$ denotes the component in $L$ with respect to the decomposition $\T M_\C=L\oplus \tL$. Note that $[\cdot,\cdot]$ is skew-symmetric because $L$ is Lagrangian. If $L$ is integrable then taking the $L$-component is redundant, the bracket is independent of $\tL$ and satisfies the Jacobi identity which implies $d_L^2=0$. In general though, $d_L^2=0$ does not imply that $L$ is integrable.

\begin{ex}\label{09:55:15}
Let $\omega\in \Omega^2(M;\C)$ be a complex two-form and consider the almost Dirac structure $L=\text{graph}(\omega)=\{X+\omega(X)| \ X\in TM_\C\}\subset \T M_\C$. We have 
\begin{align}\label{09:17:19}
\lb X+\omega(X),Y+\omega(Y)\rb=[X,Y]+\omega([X,Y])+\iota_Y\iota_X(d\omega-H),
\end{align} 
so $L$ is integrable precisely when $d\omega=H$. One possible complement for $L$ is given by $\tL=T^\ast M_\C$, for which the associated bracket on $L$ coincides with the ordinary Lie bracket on $TM_\C$ via the isomorphism $\rho:L\rightarrow TM_\C$. It follows that $d_L=d$, the usual exterior derivative on $C^\infty(\Lambda^\bullet L^\ast)\cong \Omega^\bullet(M;\C)$, and $d_L^2=0$ regardless of whether $L$ is integrable or not.
\end{ex}

\begin{ex}\label{10:58:57} Consider $L=\text{graph}(\pi)=\{\xi+\pi(\xi)|\ \xi \in T^\ast M_\C\}$ for $\pi\in C^\infty(\Lambda^2TM_\C)$. For $f,g\in C^\infty(M)$ write $\{f,g\}:=\pi(df,dg)$, so that $L$ is involutive if and only if 
\begin{align*}
\{f,\{g,h\}\}+\{h,\{f,g\}\}+\{g,\{h,f\}\}=H(X_f,X_g,X_h), 
\end{align*}
where $X_f:=\pi(df)$. One possible complement is $\tL=TM_\C$, for which the resulting derivation $d_L$ on $C^\infty(\Lambda^\bullet L^\ast)\cong C^\infty(\Lambda^\bullet TM_\C)$ is determined by $d_Lf=[\pi,f]$ and $d_LX=[\pi,X]+\Lambda^2\pi(\iota_XH)$, where $[\cdot,\cdot]$ denotes the Schouten-Nijenhuis bracket on multi-vector fields. In this case $d_L^2=0$ is equivalent to involutivity of $L$.
\end{ex}

\begin{rem}\label{08:21:13}
If $L$ is almost Dirac with almost Dirac complement $\tL$, then $\rho$ is bracket preserving if and only if the \textsl{Nijenhuis tensor}, defined by $N_L(u,v)=\lb u,v\rb^{\tL}$ for $u,v\in L$, has the property that $N_L(u,v)\in \tL\cap T^\ast M_\C$. Hence, if $\tL\cap T^\ast M_\C=0$ as in the previous example, then $d_L^2=0$ implies that $N_L=0$ so that $L$ is integrable.
\end{rem}


A special class of Dirac structures is given by generalized complex structures. 
\begin{defn}
An \textsl{almost generalized complex structure} on $M$ is a complex structure $\J$ on the bundle $\T M$ that is orthogonal with respect to the natural pairing $\langle\cdot,\cdot\rangle$.
\end{defn}
An almost generalized complex structure $\J$ is equivalently described by the decomposition $\T M_\C=L\oplus \overline{L}$, where $L$ is the $(+i)$-eigenbundle of $\J$, and orthogonality of $\J$ amounts to $L$ being Lagrangian. Hence, almost generalized complex structures are equivalent to almost Dirac structures satisfying $L\cap \overline{L}=0$, and we will often call such an $L$ an almost generalized complex structure. In this setting there is a canonical complement $\tL:=\overline{L}$ and hence a canonical choice for $d_L$. If $L$ is integrable we call $\J$ (or $L$) a \textsl{generalized complex structure}. 
\begin{ex}\label{09:51:12}
Consider $L=\text{graph}(B+i\omega)$ as in Ex.\ref{09:55:15}, where $B$ and $\omega$ are real. Then $L\cap \overline{L}=0$ if and only if $\omega$ is non-degenerate, and $\J$ is given by
\begin{align*}
\J=\begin{pmatrix} 1 & 0 \\ B & 1 \end{pmatrix}\begin{pmatrix} 0 & \omega^{-1} \\ -\omega & 0 \end{pmatrix}\begin{pmatrix} 1 & 0 \\ -B & 1 \end{pmatrix}
\end{align*}
with respect to the decomposition $\T M=TM\oplus T^\ast M$. 
\end{ex}
 For the following two examples we take $H=0$.
\begin{ex}\label{08:30:48}
Let $I$ be an almost complex structure on $M$. 
Then 
\begin{align*}
\J=\begin{pmatrix} -I & 0 \\ 0 & I^\ast \end{pmatrix}
\end{align*}
defines an almost generalized complex structure with $L:=T^{0,1}M\oplus T^{\ast 1,0}M$. 
\end{ex}
\begin{ex}\label{11:08:18}
Let $I$ be an almost complex structure on $M$ and $Q \in C^\infty(\Lambda^2 TM)$ a bivector satisfying $QI^\ast=IQ$. Then we can define an almost generalized complex structure by
\begin{align*}
\J=\begin{pmatrix} -I & Q \\ 0 & I^\ast \end{pmatrix}.
\end{align*}
If we define $\sigma:=-\frac{1}{4}(IQ+iQ)\in C^\infty(\Lambda^{2}T^{1,0}M)$, then 
\begin{align*}
L:=T^{0,1}M\oplus \text{graph}(\sigma)=\{X+\sigma(\xi)+\xi| \ X\in T^{0,1}M, \xi\in T^{\ast 1,0}M\}.
\end{align*} 
This structure is integrable if and only if $(I,\sigma)$ is a holomorphic Poisson structure.
\end{ex}
It is also possible to incorporate into $d_L$ coefficients in another vector bundle $V$.
\begin{defn}
An \textsl{$L$-connection} on a vector bundle $V$ is an operator $\nabla : C^\infty(V)\rightarrow C^\infty(L^\ast\otimes V)$ satisfying $\nabla(fs)=d_Lf\otimes s+f\nabla s$ for all $s\in C^\infty(V)$ and $f\in C^\infty(M)$.
\end{defn}
\noindent An $L$-connection induces an operator $d_L^\nabla:C^\infty(\Lambda^\bullet L^\ast \otimes V)\rightarrow C^\infty(\Lambda^{\bullet+1} L^\ast \otimes V)$ satisfying
\begin{align*}
d_L^\nabla(\alpha\otimes v)=d_L\alpha\otimes v+(-1)^q\alpha\wedge \nabla v 
\end{align*}  
 for $\alpha\in C^\infty(\Lambda^{q} L^\ast)$ and $v\in C^\infty(V)$. As for $d_L$, we do not assume that $(d_L^\nabla)^2=0$, i.e.\ the $L$-connection does not have to be flat.
\begin{ex}
Let $L=T^{0,1}M$ be associated to an almost complex structure on $M$. Then the usual $\delbar$-operator
defines $L$-connections on $\Lambda^pT^{\ast 1,0}M$ and $\Lambda^pT^{1,0}M$ for every $p\geq 0$.
\end{ex}

\begin{ex}\label{09:30:04}
Let $L$ be a pre-Lie algebroid and $K\subset L$ a subbundle satisfying $\rho(K)=0$, $[K,K]=0$ and $[L,K]\subset K$. Then $L/K$ naturally inherits the structure of a pre-Lie algebroid, and the vector bundle $K$ comes equipped with an $L/K$-connection 
\begin{align*}
\nabla_u v:=[s(u),v],
\end{align*}
where $s:L/K\rightarrow L$ is a section. Since $K$ is abelian, the connection $\nabla$ is independent of $s$. This induces connections on the exterior powers of $K^\ast$, and the resulting derivations 
\begin{align*}
d_{L/K}^\nabla:C^\infty(\Lambda^q (L/K)^\ast\otimes \Lambda^p K^\ast)\rightarrow C^\infty(\Lambda^{q+1} (L/K)^\ast\otimes \Lambda^p K^\ast)
\end{align*} 
coincide with $d_L$ on $\Lambda^\bullet L^\ast=\bigoplus \Lambda^\bullet L/K^\ast\otimes \Lambda^\bullet K^\ast$. An example of this is given by $L=T^{0,1}M\oplus T^{\ast 1,0}M$ and $K=T^{\ast 1,0}M$ for an almost complex structure on $M$. The relevance of this construction will be explained in Rem.\ref{09:32:20}.
\end{ex} 
An important invariant of a Lie algebroid $L$ are its cohomology groups
$$H^\bullet(L):=\frac{\text{Ker}(d_L)}{\text{Im}(d_L)}.$$
For instance, $H^0(L)$ 
corresponds to the functions on $M$ that are constant along the image of the anchor. For generalized complex structures, $H^1(L)$ can be identified with infinitesimal symmetries of $L$ modulo those that are Hamiltonian, while $H^2(L)$ gives the set of infinitesimal generalized complex deformations of $L$, modulo those obtained from $L$ by symmetries of $\T M$.  

\subsection{Cauchy-Riemann structures and the Levi form}

Our aim is to develop Hodge theory for the operator $d_L$ of a pre-Lie algebroid $L$ (Thm.\ref{08:54:49} below). 
On a compact manifold without boundary this can be done if $d_L$ defines an \textsl{elliptic complex}, i.e.\ $d_L^2=0$ and for each real nonzero $\xi\in T^\ast M$ the symbol sequence 
\begin{align}\label{12:11:39}
\begin{CD}
\cdots \Lambda^{q-1}L^\ast     @>\sigma(d_L,\xi)>>   \Lambda^qL^\ast @>\sigma(d_L,\xi)>> \Lambda^{q+1}L^\ast \cdots
\end{CD}
\end{align}
is exact, where $\sigma(d_L,\xi)=\rho^\ast\xi\wedge (\cdot)$ denotes the principal symbol of the first-order operator $d_L$. Note that this last condition makes sense even when $d_L^2\neq 0$. 
\begin{defn}\label{08:56:59}
A pre-Lie algebroid $L$ is called \textsl{elliptic} if (\ref{12:11:39}) is exact, or equivalently if 
\begin{align}\label{10:38:35}
\rho(L)+\overline{\rho(L)}=TM_\C.
\end{align}
\end{defn}
\begin{ex} All the examples of the previous section, except for Ex.\ref{10:58:57} in general, are elliptic. Note that almost generalized complex structures are always elliptic. 
\end{ex}
 Ellipticity is the first natural requirement to impose on a pre-Lie algebroid for admitting Hodge decompositions, but if $\del M\neq \emptyset$ we also need information about $L$ along the boundary. 

Let $L$ be an elliptic pre-Lie algebroid. Since $T\del M_\C=\overline{T\del M_\C}$, both $\rho(L)$ and $\overline{\rho(L)}$ are transverse to $T\del M_\C$ and so the pre-image $\rho^{-1}(T\del M_\C)$ defines a co-rank $1$ subbundle of $L|_{\del M}$. 
\begin{defn}
The \textsl{Cauchy-Riemann structure} of $L$ is the bundle $\mathcal{C}\mathcal{R}:=\rho^{-1}(T\del M_\C)$. 
\end{defn}
\begin{rem} If $\rho:L\rightarrow TM_\C$ preserves brackets (e.g.\ if $L$ is a Lie-algebroid) then $\mathcal{C}\mathcal{R}$ is itself a pre-Lie algebroid on $\del M$ called the \textsl{restriction} of $L$ to the submanifold $\del M$.
\end{rem}
\begin{ex}
If $L=T^{0,1}M$ for an almost complex structure $I$, then $\mathcal{C}\mathcal{R}=T^{0,1}M\cap T\del M_\C$. 
Note that usually $\overline{\mathcal{C}\mathcal{R}}$ is called the Cauchy-Riemann structure of $(M,I)$. 
\end{ex}
 As mentioned in the introduction, the exterior derivative defines an elliptic boundary value problem while the $\delbar$-operator does not. Below we will point out what the crucial difference is between these two operators, and show that a general Lie algebroid mixes these two different types of behaviour. 
We will need the following lemma.
\begin{lem}\label{12:31:25}
For $x\in \del M$ define 
\begin{align*}
\mu_{x}:=\frac{T_x\del M_\C}{\rho(\mathcal{C}\mathcal{R}_x)+\overline{\rho(\mathcal{C}\mathcal{R}_x)}}.
\end{align*}
Then $\dim_\C(\mu_{x})\leq 1$, and $\mu_x=0$ if and only if $\rho(L_x)\cap\overline{\rho(L_x)}\nsubseteq T_x\del M_\C$.
\end{lem}
\begin{proof}
Let $\nu\in T_x M$ be outward pointing. Since $\rho(L_x)$ is transverse to $T_x\del M_\C$, we can write $L_x=\C u\oplus \mathcal{C}\mathcal{R}_x$ where $\rho(u)-\nu\in T_x\del M_\C$. It follows that $\mu_{x}$ is spanned by the element $\rho(u)-\overline{\rho(u)}$, which is zero in $\mu_{x}$ precisely when 
$\rho(L_x)\cap\overline{\rho(L_x)}$ is transverse to $T_x\del M_\C$.
\end{proof}
\begin{defn}
A point $x\in \del M$ is called \textsl{elliptic} if $\mu_{x}=0$ and \textsl{non-elliptic} otherwise. 
\end{defn}
The reason for this terminology is that elliptic boundary points are precisely those points where $d_L$ defines an elliptic boundary value problem in the sense of Lopatinsky \cite{Lop} (c.f.\ \cite{MR0161012}). They are also precisely the points where $\mathcal{C}\mathcal{R}$ itself is elliptic over $\del M$, in the sense that $\rho(\mathcal{C}\mathcal{R}_x)+\overline{\rho(\mathcal{C}\mathcal{R})_x}=T_xM_\C$. Note that the set of elliptic points is open in $\del M$.  
\begin{ex}
If $L=TM_\C$ (corresponding to the exterior derivative) then $\rho(L)\cap \overline{\rho(L)}=TM_\C$ and so every boundary point is elliptic. Hodge theory in this setting was established by Conner \cite{MR0078467}. If $L=T^{0,1}M$ (corresponding to the $\delbar$-operator) then $\rho(L)\cap \overline{\rho(L)}=0$ and every boundary point is non-elliptic.
\end{ex}
\begin{ex}\label{11:21:28}
An almost generalized complex structure $\J$ induces a bi-vector $\pi_\J$ on $M$ by $\pi_\J(\xi):=\pi_{TM}(\J\xi)$, where $\pi_{TM}:\T M \rightarrow TM$ is the projection. One can verify that 
$$\rho(L)\cap \overline{\rho(L)}=\pi_{\J}(T^\ast M_\C),$$ so $x\in \del M$ is non-elliptic if $\pi_\J(T^\ast M)\subset T\del M$, or equivalently if 
$\pi_\J(N^\ast_{\del M})=0$ ($N^\ast_{\del M}$ denotes the conormal bundle of $\del M$). 
If $\J$ is complex (Ex.\ref{08:30:48}) then $\pi_\J=0$ and every point is non-elliptic, while if $\J$ is symplectic (Ex.\ref{09:51:12}) then $\pi_\J$ is invertible and every point is elliptic. 
\end{ex}
In order to develop Hodge theory for $d_L$ we need extra information about $L$ at the non-elliptic boundary points. Before stating the relevant definition we need a lemma.  
\begin{lem}\label{12:21:10}
Let $u,v\in \mathcal{C}\mathcal{R}_x$ and let $\tilde{u},\tilde{v}\in C^\infty(\mathcal{C}\mathcal{R})$ be local extensions of $u$ and $v$. Then $[\rho(\tilde{u}),\overline{\rho(\tilde{v})}]|_x\in \mu_{x}$ is independent of the choice of $\tilde{u}$ and $\tilde{v}$.
\end{lem}
\begin{proof}
Since $\rho(\tilde{u})$ and $\overline{\rho(\tilde{v})}$ are tangent to $\del M$, their Lie bracket is well-defined in $T_x\del M_\C$ and projects to an element of $\mu_{x}$. By definition of $\mu_{x}$ as a quotient we see that $[\rho(\tilde{u}),\overline{\rho(\tilde{v})}]|_x\in \mu_{x}$ is $C^\infty(M)$-linear in both $\tilde{u}$ and $\tilde{v}$, from which the lemma immediately follows.
\end{proof}
\begin{defn}
Let $L$ be an elliptic pre-Lie algebroid and $x\in \del M$ non-elliptic. The \textsl{Levi-form} of $L$ at $x$ is the Hermitian bilinear form on $\mathcal{C}\mathcal{R}_x$ with values in $\mu_{x}$ given by
\begin{align}\label{12:20:40}
\mathcal{L}_{x}(u,v):=-i[\rho(u),\overline{\rho(v)}].
\end{align}  
\end{defn}
\begin{rem} Here we are implicitly choosing local extensions of $u$ and $v$ as in Lem.\ref{12:21:10} to compute the right hand side of (\ref{12:20:40}). 
\end{rem}
If $\nu\in T_xM$ is outward pointing, we can choose $u\in L_x$ so that $\rho(u)-\nu\in T_x\del M_\C$, implying that $\mu_x$ is spanned by the real element $i(\overline{\rho(u)}-\rho(u))$. The latter depends only on the choice of $\nu$, while varying $\nu$ only changes it by a positive multiple inside $\mu_x$. These identifications $\mu_x\cong \C$, differing from each other only by positive rescalings, allow us to identify $\mcL_x$ with a conformal class of Hermitian bilinear forms on $\mcC\mathcal{R}_x$, and hence we can talk about the number of positive and negative eigenvalues of $\mcL_x$.


\begin{rem}\label{12:51:01} The following gives a explicit realization of the Levi form. Let $x\in \del M$ and $\nu$ an outward pointing vector field near $x$. Choose a frame $w_1,\ldots,w_l$ for $L$ around $x$ such that $\mcC\mathcal{R}$ is spanned by $w_1|_{\del M},\ldots,w_{l-1}|_{\del M}$ and such that $\rho(w_l)|_{\del M}-\nu\in T\del M_\C$. 
If we set $v_i:=\rho(w_i)$, then there exist (not necessarily unique) smooth functions $a^k_{ij}, b^k_{ij}$ such that 
\begin{align}\label{15:39:25}
[\overline{v_j},v_i]=\sum_{k=1}^l( a^k_{ij}v_k+b^k_{ij}\overline{v_k}).
\end{align} 
Since $\rho(L)$ is transverse to $T\del M_\C$ we have $\overline{v_l}\in \rho(L)+T\del M_\C$, hence if $x$ is elliptic (so that $T_x\del M_\C=\rho(\mcC\mathcal{R}_x)+\overline{\rho(\mcC\mathcal{R}_x)}$) we can arrange for $b^l_{ij}|_{\del M}=0$ around $x$ for all $i$ and $j$. If $x$ is non-elliptic, then $[\overline{v_j},v_i]|_{\del M}\in T\del M_\C$ for $i,j<l$ implies that $a^l_{ij}=-b^l_{ij}$ on $\del M$, and hence 
\begin{align*}
[\overline{v_j},v_i]=\sum_{k<l} (a^k_{ij}v_k+ b^k_{ij}\overline{v_k})-ib^l_{ij}(i(\overline{v_l}-v_l))
\end{align*}
on $\del M$. It follows from the definition (see (\ref{12:20:40})) that $\mcL_{x}(w_i,w_j)=b^l_{ij}(x)\in \C$ for all $i,j<l$.
\end{rem}
\begin{defn}\label{08:56:00}
An elliptic pre-Lie algebroid $L$ is called \textsl{$q$-convex}, for $0\leq q \leq \text{rank}_\C(L)$ an integer, if for every non-elliptic point $x\in \del M$ the Levi-form $\mcL_{x}$ on $\mcC\mathcal{R}_x$ has either at least $(\text{rank}_\C(L)-q)$ positive eigenvalues or at least $(q+1)$ negative eigenvalues. 
\end{defn}
\begin{rem}
Note that for $L$ to be $q$-convex, some points may be elliptic while other points may have enough positive or negative eigenvalues for the Levi form. Note also that $L$ is always $q$-convex in top degree $q=\text{rank}_\C(L)$. The terminology $q$-convex is not standard, and perhaps slightly misleading as negative eigenvalues correspond to concave boundaries. In \cite{MR0461588} (in the context of complex structures) it is referred to as \lq\lq condition $Z(q)$\rq\rq.
\end{rem}
\begin{ex}
If $\rho(L)\cap \overline{\rho(L)}$ is everywhere transverse to $T\del M_\C$ then $L$ is $q$-convex for all $q$ because all points are elliptic. This is the case e.g.\ for generalized complex structures $\J$ for which the line field $\pi_\J(N^\ast_{\del M})\subset T\del M$ is nowhere zero. 
\end{ex} 

We discuss the cases of complex structures and holomorphic Poisson structures in detail.

\subsubsection*{Complex structures}

Let $I$ be an almost complex structure on $M$ and $L=T^{0,1}M$. Because $L\cap \overline{L}=0$, every point in $\del M$ is non-elliptic. Choose a frame $e_1,Ie_1,\ldots,e_n,Ie_n$ for $TM$ such that $-Ie_n|_{\del M}=:\nu$ points outwards, while the others are all tangent to $\del M$. Define 
$$w_j:=ie_j-Ie_j \hspace{10mm}  1\leq j \leq n,$$ 
so that $\{w_j\}_{ j<n}$ span $\mcC\mathcal{R}=T^{0,1}M\cap T\del M_\C$ and $w_n-\nu=ie_n\in T\del M_\C$. In particular, 
$i(\overline{w_n}-w_n)=2e_n$ 
defines a positive generator for $\mu$ (see Lem.\ref{12:31:25}). 
\begin{lem}\label{09:24:01}
Let $r$ be a smooth function around a point $x\in\del M$ satisfying $r|_{\mathring{M}}<0$, $r|_{\del M}=0$ and $d_xr\neq 0$. Then after a suitable rescaling of $\nu$ such that $dr(\nu)=1$, we have
\begin{align*}
\hspace{30mm} \mcL_{x}(u,v)=\delbar r([u,\overline{v}]) =d\delbar r(\overline{v},u) \hspace{15mm} \forall u,v\in \mcC\mathcal{R}_{x}.
\end{align*} 
\end{lem}
\begin{proof}
The one-form $ i\delbar r$ annihilates $\mcC\mathcal{R}$ and $\overline{\mcC\mathcal{R}}$ while $i\delbar r(2e_n)=dr(\nu)=1$. Consequently, 
\begin{equation*}
\mcL_{x}(u,v)=i\delbar r(-i[u,\overline{v}])=\delbar r([u,\overline{v}])=-d\delbar r(u,\overline{v}) \hspace{15mm} \forall u,v\in \mcC\mathcal{R}_{x}.
\end{equation*} 
In the last step we used that $\delbar r$ annihilates $\mcC\mathcal{R}$ and $\overline{\mcC\mathcal{R}}$ in a neighbourhood of $x$ in $\del M$. 
\end{proof}
\noindent If $M$ is a submanifold of a complex manifold $M'$ without boundary, then around any point in $\del M$ we can find holomorphic coordinates $(z_1,\ldots,z_n)$ in $M'$. Then $u=\sum_i u^i \del_{\overline{z}^i}\in T^{0,1}M$ lies in $ \mcC\mathcal{R}$ if and only if $\sum_i u^i  \del_{\bar{z}_i} r=0$, and 
\begin{align*}
\mcL(u,v)=\sum_{i,j} \frac{\del^2r}{\del z^i\del \overline{z}^j} \overline{v^i} u^j.
\end{align*} 
\begin{ex}
If the above Hessian of $r$ is positive definite on $T^{0,1}M_\C$ then its restriction to $\mcC\mathcal{R}$ is too and hence $(M,I)$ is $q$-convex for all $q\geq 1$. For instance, $r:=|z|^2-1$ on the unit ball in $\C^n$ has this property, which is therefore $q$-convex for all $1\leq q\leq n$. If we remove a smaller ball from its interior then the boundary has two components, on which $\mcL$ is positive and negative definite. This annular region is therefore $q$-convex for all $1\leq q\leq n-2$. 
\end{ex}   
\begin{ex}\label{10:11:15} 
Let $Y$ be an $(n-1)$-dimensional compact complex manifold without boundary and $p:X\rightarrow Y$ a holomorphic line bundle. Let $h$ be a Hermitian metric on $X$, let $M$ be the unit disc-bundle in $X$ and 
$L=T^{0,1}M$. We have $\mcC\mathcal{R}\cong p^\ast T^{0,1}Y$, and if $R_h$ denotes the curvature associated to $h$ then a direct calculation shows that 
 \begin{align*}
\hspace{15mm} \mcL(u,v)=R_h(u,\overline{v})\hspace{15mm} \forall u,v\in T^{0,1}Y.
\end{align*}
Consequently, $M$ is $q$-convex if and only if $-iR_h$ has, at each point $y\in Y$, either at least $n-q$ positive eigenvalues or at least $q+1$ negative eigenvalues\footnote{By definition, the eigenvalues of a real $(1,1)$-form $\tau$ are the eigenvalues of the Hermitian matrix $\tau_{ij}$ with respect to a decomposition $\tau=i\sum_{i,j}\tau_{ij}dz^i\wedge d\bar{z}^j$ in local coordinates.}. 
\end{ex}
Using the fact that $c_1(X)$ is represented by $\tfrac{i}{2\pi}R_h$, one can sometimes translate the previous example into a statement about $c_1(X)$.
\begin{lem}\cite{MR0137127}\label{11:16:39}
Let $Y$ satisfy the $\del\delbar$-lemma\footnote{Specifically, we need that every real $(1,1)$-form which is $d$-exact is also $\del\delbar$-exact.}. Then there exists a $q$-convex disc neighbourhood 
of $Y$ in $X$  if and only if $c_1(X)$ has a real representative $\tau\in \Omega^{1,1}(Y)$ which at each point has either at least $n-q$ negative eigenvalues or at least $q+1$ positive eigenvalues.
\end{lem}
\begin{ex}\label{11:18:37}
If $Y$ is compact and $X$ is negative, meaning that $c_1(X)$ admits a representative whose eigenvalues are all negative, then $Y$ is K\"ahler and the $\del\delbar$-lemma holds. Consequently, there exists a disc-neighbourhood which is $q$-convex for all $q\geq 1$. The line bundles $\mcO_{\mathbb{P}^m}(-n)$ over $\mathbb{P}^m$ for $n>0$ satisfy this for example.
\end{ex}

\begin{rem}\label{09:32:20}
Instead of $L=T^{0,1}M$ we can also consider $L':=T^{0,1}M\oplus T^{\ast 1,0}M$. Then $\mcC\mathcal{R}_{L'}=\mcC\mathcal{R}_{L}\oplus T^{\ast 1,0}M|_{\del M}$ and $\mcL_{L'}=\mcL_{L}\oplus 0$. So, if e.g.\ $\mcL_{L}$ is positive definite and $\text{dim}_\C(M)=n$, then $L$ is $q$-convex for all $0< q\leq n$, while $L'$ is $q$-convex for all $n<q\leq 2n$. 
In this situation Ex.\ref{09:30:04} can be used to provide more information on the Hodge theory of $L'$ than Thm.\ref{08:54:49} does when applied directly to $L'$.
\end{rem}

\subsubsection*{Holomorphic Poisson structures}

Let $(I,\sigma)$ be a holomorphic Poisson structure with corresponding Dirac structure (Ex.\ref{11:08:18})
\begin{align*}
L=T^{0,1}M\oplus \text{graph}(\sigma)
\cong T^{0,1}M\oplus T^{\ast 1,0}M.
\end{align*} 
A point $x\in \del M$ is non-elliptic precisely when $\sigma(N^\ast_{\del M,x})=0$ (c.f.\ Ex.\ref{11:21:28}), in which case
\begin{align*}
\mcC\mathcal{R}_{x}&=(T_x\del M_\C\cap T_x^{0,1}M)\oplus \text{graph}(\sigma) \cong  (T_x\del M_\C\cap T_x^{0,1}M)\oplus T^{\ast 1,0}M.
\end{align*} 
Let $r$ a function as in Lem.\ref{09:24:01} and denote by $X_r:=\sigma(dr)\in T\del M_\C\cap T^{1,0}M $ its Hamiltonian vector field. Since $X_r$ vanishes at $x$, it induces a linear map 
$$[X_r,\cdot]:T_x\del M_\C\cap T^{1,0}_xM\rightarrow T_x\del M_\C\cap T^{1,0}_xM.$$  
\begin{lem}\label{12:24:05}
For all $u,v\in T_x\del M_\C\cap T_x^{0,1}M$ and $\alpha,\beta \in T^{\ast 1,0}_xM$ we have
\begin{align*}
\mcL_{x}(u,v)=&\del\delbar r(\overline{v},u), \hspace{10mm}
\mcL_{x}(\alpha,u)=&\alpha([X_r,\overline{u}]), \hspace{10mm}  
\mcL_{x}(\alpha,\beta)=&\del\delbar r({\sigma(\alpha)},\overline{\sigma(\beta)}).
\end{align*}
\end{lem}
\begin{proof} Choose extensions of $u,v$ and $\alpha,\beta$ in $T\del M_\C\cap T^{0,1}M$ and $T^{\ast 1,0}M$, respectively, again denoted by $u,v,\alpha,\beta$. Then $u$ and $v$ are sections of $\mcC\mathcal{R}$ but $\alpha+\sigma(\alpha)$ and $\beta+\sigma(\beta)$ in general are not, because $\sigma(\alpha)$ and $\sigma(\beta)$ are not necessarily everywhere tangent to $\del M$. Choose $w\in T^{0,1}M$ satisfying $dr(w)=1$, so that $\sigma(\alpha)+\sigma(dr,\alpha)w$ is tangent to $\del M$ and hence
\begin{align*}
\alpha+\sigma(\alpha)+\sigma(dr,\alpha)w=\alpha+\sigma(\alpha)+\alpha(X_r)w
\end{align*}
is an extension of $\alpha+\sigma(\alpha)$ in $ \mcC\mathcal{R}$ (and similarly for $\beta$), which we can use to compute $\mcL_{x}$. The expression for $\mcL_x(u,v)$ follows immediately from Lem.\ref{09:24:01}, while
\begin{align*}
\mcL_{x}(\alpha,u)=\delbar r\big([\sigma(\alpha)+\alpha(X_r)w,\overline{u}]\big)=\alpha([X_r,\overline{u}]).
\end{align*} 
Finally, using that $X_r$ vanishes at $x$, we obtain
\begin{align*}
\mcL_x(\alpha,\beta)=&\delbar r\big([\sigma(\alpha)+\alpha(X_r)w,\overline{\sigma(\beta)+\beta(X_r)w}]	\big)
=\delbar r([\sigma(\alpha), \overline{\sigma(\beta)}]) + \alpha([X_r,\overline{\sigma(\beta)}])\\
=&-\del\delbar r (\sigma(\alpha), \overline{\sigma(\beta)})+ \overline{\beta([X_r,\overline{\sigma(\alpha)}])}+ \alpha([X_r,\overline{\sigma(\beta)}])
=\del\delbar r (\sigma(\alpha), \overline{\sigma(\beta)}),			
\end{align*}
where in the last step we used that $\alpha([X_r,\overline{\sigma(\beta)}])=\del\delbar r(\sigma(\alpha),\overline{\sigma(\beta)})$. 
\end{proof}

\begin{ex}\label{09:10:02}
Let $M$ be the unit ball in $\C^{2k+2}$ with holomorphic Poisson structure 
$$\sigma=x\frac{\del}{\del x}\wedge\frac{\del}{\del y}+\sum_{i=1}^k \frac{\del}{\del {z^i}}\wedge \frac{\del}{\del {w^i}}.$$
Consider the function $r=|x|^2+|y|^2+|z|^2+|w|^2-1$, with Hamiltonian vector field 
\begin{align*}
X_r
=&|x|^2\del_y-\overline{y}x\del_x+\overline{z}^i\del_{w^i}-\overline{w}^i\del_{z^i}.
\end{align*}
The non-elliptic points on $\del M$ are those where $X_r$ vanishes, i.e.\ $\{x=z=w=0\}\subset M$. At such a point $y\in M$ the Cauchy-Riemann structure is given by 
$$\mcC\mathcal{R}_{y}=\langle\del_{\overline{x}},\del_{\overline{z}^i},\del_{\overline{w}^i},dx,dy,dz^i,dw^i\rangle$$ 
and
 $[X_r,\del_x]=\overline{y}\del_x$, $[X_r,\del_{z^i}]=[X_r,\del_{w^i}]=0$. It follows from Lem.\ref{12:24:05} together with a quick calculation that $\mcL_y$ has $4k+1$ positive, $1$ negative and $1$ zero eigenvalues for every non-elliptic point $y\in M$. In particular, $L=T^{0,1}M\oplus \text{graph}(\sigma)$ is $q$-convex for $q=0$ and $3\leq q\leq 4k+4$. 
\end{ex}

\subsection{Deformations}\label{16:33:30}

In applications one often requires Hodge theory for Lie algebroids that vary in a family. We will describe a precise setting here and introduce some notation that will be useful in this context. Let $L$ and $H$ be fixed vector bundles on $M$. 
\begin{defn}
A \textsl{family of pre-Lie algebroid structures on} $L$ \textsl{(parametrized by} $H$\textsl{)} is a differential operator 
\begin{align*}
d_L:C^\infty(\Lambda^\bullet L^\ast )\times C^\infty (H)\rightarrow C^\infty(\Lambda^{\bullet+1} L^\ast )
\end{align*} 
such that for each fixed $\ve\in C^\infty(H)$ the induced operator $d_{L_\ve}:=d_L(\cdot, \ve)$ is a derivation.   
\end{defn}
\begin{rem} i) We will also call a family a \textsl{deformation}, regarding $d_{L_0}$ as an initial pre-Lie algebroid that is being deformed. In particular, $\ve$ is usually only considered in an open neighbourhood of zero in $C^\infty(H)$ (with respect to the $C^\infty$-topology). When there is no risk of confusion, we will denote the pair $(L,d_{L_\ve})$ by $L_\ve$ and abbreviate $L=L_0$.

ii) Because of the derivation property, $d_L$ is of order at most $1$ in the $L^\ast$-variable. It can however have any finite order in the deformation variable, although in the examples that we have in mind this will typically be $1$ as well. 

iii) We can include coefficients from another vector bundle $V$ by considering an operator
\begin{align*}
d_L^\nabla:C^\infty(\Lambda^\bullet L^\ast \otimes V)\times C^\infty (H)\rightarrow C^\infty(\Lambda^{\bullet+1} L^\ast \otimes V),
\end{align*}
which for each fixed $\ve$ defines an $L_\ve$-connection $d^\nabla_{L_\ve}$ on $V$. 
\end{rem}
\begin{ex}
Let $L\subset \T M_\C$ be an almost Dirac structure and let $\tL$ be an almost Dirac complement. Small deformations of $L$ can be described by elements $\ve\in \Lambda^2\tL\cong \Lambda^2L^\ast$ via 
\begin{align*}
L(\ve):=\{u+\iota_u\ve| \ u\in L\}\subset L\oplus \tL=\T M_\C.
\end{align*}
Since $\tL$ remains complementary to $L(\ve)$, we can identify $\tL\cong L(\ve)^\ast$. This induces derivations $d_{L_\ve}$ on $C^\infty(\Lambda^\bullet\tL)\cong C^\infty(\Lambda^\bullet L^\ast)$, so we obtain a family of pre-Lie algebroid structures on $L$ parametrized by $H:=\Lambda^2L^\ast$. If $\tL$ is integrable we have
$$d_{L_\ve}=d_L+\lb \ve,\cdot\rb.$$
\end{ex}
If a pre-Lie algebroid $L$ is elliptic and $q$-convex, then so is any small deformation of it because ellipticity and $q$-convexity are open conditions\footnote{This is actually not immediately obvious, but can be seen e.g.\ using Rem.\ref{12:51:01}.}. The reason to study families in this context is that we need to understand certain variational aspects of the family of Hodge decompositions. For that purpose we will use the following notation, borrowed from \cite{MR3128977}. For $\ve\in C^\infty(H)$ and $\varphi\in C^\infty(E)$, where $H$ and $E$ are vector bundles on $M$, we will write 
\begin{align}\label{10:19:58}
\mcL(|\ve|_k;||\varphi||_l):=Poly(|\ve|_{\lfloor k/2 \rfloor+1}) \cdot ||\varphi||_l+Poly(|\ve|_{\lfloor k/2 \rfloor+1})\cdot |\ve|_k\cdot || \varphi||_{\lfloor l/2 \rfloor+1}.
\end{align}
Here $|\cdot|_k$ and $||\cdot||_l$ denote the $C^k$-norm and $l$-th Sobolev norm on $H$ and $E$, respectively, $\lfloor x\rfloor$ denotes the biggest integer bounded by $x\in \R$, and $Poly$ denotes a polynomial with non-negative coefficients which depend only on $k$, $l$ and $|\ve|_0$. We also allow the degree of these polynomials to depend on $k$ and $l$. The main point is that $\mcL(|\ve|_k;||\varphi||_l)$ is linear in $\varphi$ and each monomial contains precisely one high derivative (of order $k$ or $l$), the rest are all lower derivatives (of order $k/2$ or $l/2$). The use of the symbol $\mcL$, standing for \lq\lq Leibniz\rq\rq, is unfortunate because we also use it for the Levi form, but since these two will never be used in the same context there should be no risk of confusion. The main example where this notation appears is when applying derivatives to a product. For example, if $d_{L_\ve}$ is a family of pre-Lie algebroids on $L$ of order $a$ in the deformation parameter $\ve$, we have (for every $k\in \mathbb{Z}_{\geq 0}$)
\begin{align*}
||d_{L_\ve}\varphi ||_k\leq \mcL(|\ve|_{k+a}; ||\varphi||_{k+1}).
\end{align*}
\begin{rem}
To bound a given quantity by $\mcL(|\ve|_k;||\varphi||_l)$, the main things to verify are that there are no derivatives higher than $k$ on $\ve$ or higher than $l$ on $\varphi$, and that we never apply at the same time more than $\lfloor k/2\rfloor +1$ derivatives on $\ve$ or more than $\lfloor l/2\rfloor +1$ derivatives on $\varphi$. 
\end{rem} 
\begin{rem}
Sometimes we use $\mcL(|\ve|^2_k;||\varphi||^2_l)$, which is defined by (\ref{10:19:58}) but with all norms on the right-hand side squared. One readily verifies that the bound $A\leq \mcL(|\ve|_k;||\varphi||_l) $ is equivalent to the bound $A^2\leq \mcL(|\ve|^2_k;||\varphi||^2_l)$ for any $A\in \R$. 
\end{rem}
\begin{rem}\label{16:50:34}
For $k_1,k_2,l_1,l_2\in \mathbb{Z}_{\geq 0}$ we have
\begin{align}\label{10:50:14}
\mcL(|\ve|_{k_1};||\varphi||_{l_1})\cdot \mcL(|\ve|_{k_2};||\varphi||_{l_2})\leq \mcL(|\ve|^2_{k};||\varphi||^2_{l}),
\end{align}
where $k:=\max(k_1,k_2)$ and $l:=\max(l_1,l_2)$. 
Moreover, if we are given bounds of the form $||\varphi||_{l_1}\leq \mcL( |\ve|_{k_2};||\psi||_{l_2})$ and $||\varphi||_{\lfloor l_1/2 \rfloor+1 }\leq \mcL( |\ve|_{\lfloor k_2/2\rfloor +1};||\psi||_{\lfloor l_2/2\rfloor +1})$, then 
\begin{align}\label{10:50:23}
\mcL(|\ve|_{k_1};||\varphi||_{l_1})\leq \mcL(|\ve|_{k};||\psi||_{l_2}),
\end{align}
where again $k=\max(k_1,k_2)$. We will use (\ref{10:50:14}) and (\ref{10:50:23}) often implicitly in calculations. 
\end{rem}

\noindent

\subsection{The main result}

Let $L$, $V$ and $H$ be fixed vector bundles on $M$, let $d_{L_\ve}$ be a family of pre-Lie algebroid structures on $L$ and $d_{L_\ve}^\nabla$ a family of $L_\ve$-connections on $V$, both parametrized by $H$. We will abbreviate $d_{L_\ve}^\nabla$ by $d_{L_\ve}$ when there is no risk of confusion. Fix a Riemannian metric on $M$ and Hermitian metrics on $L$ and $V$. Denote by $(d_{L_\ve})^\ast_f$ the formal adjoint of $d_{L_\ve}$ and by 
$$\Delta_{\ve}:=d_{L_\ve}  (d_{L_\ve})^\ast_f+(d_{L_\ve})^\ast_f d_{L_\ve}$$ the Laplacian, which is a second-order differential operator on $\Lambda^qL^\ast\otimes V$ for every $q\geq 0$. 

\begin{thm}\label{08:54:49}
For every $q\geq 0$ there exists a self-adjoint unbounded operator $\square_\ve$ on $L^2(\Lambda^qL^\ast\otimes V)$ that agrees with $\Delta_\ve$ on $\text{Dom}(\square_\ve)\cap C^\infty(\Lambda^q L^\ast\otimes V)$. 
If $L_0$ is elliptic and $q$-convex for some $0\leq q\leq \text{rank}_\C(L)$, then there exists a neighbourhood $B$ of $0\in C^\infty(H)$ such that the following holds for all $\ve\in B$:
\begin{itemize}
\item[1)\ ] \lq\lq Elliptic regularity\rq\rq: If $\varphi\in \text{Dom}(\square_\ve)\cap L^2(\Lambda^q L^\ast\otimes V)$ has the property that $\square_\ve\varphi$ is smooth, then so is $\varphi$. Furthermore, there exists an integer $a\in \mathbb{Z}_{\geq 0}$ 
such that
\begin{align}\label{21:18:00}
||\varphi||_{k+1}\leq \mcL(|\ve|_{k+a};||\square_\ve\varphi||_k)+\mcL(|\ve|_{k+a};||\varphi||) 
\end{align}
for all $\varphi\in \text{Dom}(\square_\ve)\cap C^\infty(\Lambda^q L^\ast\otimes V)$ and all $k\in \mathbb{Z}_{\geq 0}$.
\item[2)\ ] \lq\lq Hodge decomposition\rq\rq: There is a closed, orthogonal decomposition 
\begin{align}\label{21:21:28}
L^2(\Lambda^qL^\ast\otimes V)=\text{Ker}(\square_\ve)\oplus \text{Im}(\square_\ve),
\end{align}
with $\text{Ker}(\square_\ve)\subset C^\infty(\Lambda^q L^\ast\otimes V)$ finite-dimensional. 
\end{itemize}
The decomposition (\ref{21:21:28}) gives rise to the Neumann operator 
$N_\ve$ on $L^2(\Lambda^qL^\ast\otimes V)$, which by definition is zero on $\mcH_\ve:=\text{Ker}(\square_\ve)$ and satisfies $N_\ve \square_\ve\varphi=(1-\pi_\ve)\varphi$ for $\varphi\in \text{Dom}(\square_\ve)$, where $\pi_\ve$ is the projection to $\mcH_\ve$. 
\begin{itemize}
\item[3)\ ] For each fixed $\ve\in B$ the operator $N_\ve$ is bounded, self-adjoint and induces bounded operators $N_\ve:L^2_k(\Lambda^qL^\ast\otimes V)\rightarrow L^2_{k+1}(\Lambda^qL^\ast\otimes V)$. For every $\varphi\in L^2(\Lambda^qL^\ast\otimes V)$ we have $\varphi=\square_\ve N_\ve\varphi+\pi_\ve\varphi$, and if $\varphi\in \text{Dom}(\square_\ve)$ then also $\varphi= N_\ve\square_\ve\varphi+\pi_\ve\varphi$.
\item[4)\ ] If $\mcH_0=0$ then, after possibly shrinking $B$, we have $\mcH_\ve=0$ for all $\ve\in B$ and 
\begin{align}\label{12:14:33}
||N_\ve \varphi||_{k+1}\leq \mcL(|\ve|_{k+a};||\varphi||_k).
\end{align}  
\item[5)] If $d_{L_\ve}^2=0$, then (\ref{21:21:28}) decomposes further into  
\begin{align}\label{21:21:49}
L^2(\Lambda^qL^\ast\otimes V)=\text{Ker}(d_{L_\ve})\cap \text{Ker}(d^\ast_{L_\ve})\oplus \text{Im}(d_{L_\ve})\oplus \text{Im}(d_{L_\ve}^\ast)
\end{align}
with $\text{Ker}(d_{L_\ve})\cap \text{Ker}(d^\ast_{L_\ve})\cong H^q({L_\ve})$, where $d_{L_\ve}^\ast$ denotes the Hilbert-space adjoint of $d_{L_\ve}$. Moreover, if $\mcH_\ve=0$ for all $\ve\in B$ and if $\varphi_\ve\in C^\infty(\Lambda^qL^\ast\otimes V)$ is a smooth family of $d_{L_\ve}$-closed elements, then the family $\psi_\ve:= d_{L_\ve}^\ast N_\ve \varphi_\ve$ depends in a $C^1$-manner on $\ve\in B$ and satisfies $d_{L_\ve}\psi_\ve=\varphi_\ve$.
\end{itemize}
\end{thm}
\begin{rem}
The exact value of $a$ in (\ref{21:18:00}) is not important to us, what matters is that it is independent of $k$ ($a$ will, among others, depend on the order of $d_{L_\ve}$ in $\ve$). Hence, we will often be implicit about the precise value of $a$ and even alter its value from one calculation to the next as long as it remains clear that it can, in the end, be chosen independent from $k$. 
\end{rem}
\begin{rem}
Since $N_\ve$ preserves the subspace $C^\infty(\Lambda^qL^\ast\otimes V)$, the decomposition stated in 3) implies the familiar decomposition $\varphi=\Delta_\ve N_\ve\varphi +\pi_\ve\varphi$ for all $\varphi\in C^\infty(\Lambda^qL^\ast\otimes V)$. 
\end{rem}
\begin{rem}
Since $\square_\ve$ coincides with the second order differential-operator $\Delta_\ve$ on $\text{Dom}(\square_\ve)\cap C^\infty(\Lambda^q L^\ast\otimes V)$, one would expect $||\varphi||_{k+2}$ on the left hand side in (\ref{21:18:00}). This is indeed the case if $\del M=\emptyset$ or if all points on the boundary are elliptic, but Folland \cite{MR0309156} showed (in the case of the $\delbar$-operator on the unit disc) that (\ref{21:18:00}) can not be improved in general.   
\end{rem}
\begin{rem}
For an explicit example of a complex manifold where the image of the $\delbar$-operator is not closed in the $L^2$-topology, see Remark 3 on page 75 in \cite{MR0461588}. 
\end{rem}

\subsection{Applications}

Our main application of Thm.\ref{08:54:49} concerns neighbourhood theorems in generalized complex geometry. It is in that setting that we really need the Leibniz-type estimate (\ref{12:14:33}), the latter being the main source of technicalities in the proof of Thm.\ref{08:54:49}. These Leibniz-type estimates are the input for a Nash-Moser style algorithm to linearize the geometric structure around a given submanifold, however since the details are quite involved they are discussed separately in \cite{NT}. In this section we discuss two direct applications of Thm.\ref{08:54:49}: a proof of a holomorphic version of the tubular neighbourhood theorem based on the Moser trick, and a finite-dimensionality result for holomorphic Poisson cohomology. 

\subsubsection*{Holomorphic tubular neighbourhoods}

Let $(X,I)$ be a complex manifold and $Y\subset X$ a compact, complex submanifold. Denote by $N_Y$ the normal bundle of $Y$ in $X$, which is a holomorphic vector bundle and so its total space is a complex manifold denoted by $(NY,I_{lin})$. 
We say that $Y$ admits a \textsl{holomorphic tubular neighbourhood} in $(X,I)$ if there exists a biholomorphic map $\varphi$ from a neighbourhood $U$ of $Y$ in $(NY,I_{lin})$ to a neighbourhood $\varphi(U)$ of $Y$ in $(X,I)$, satisfying $\varphi|_Y=\text{Id}$ and $d\varphi|_Y=\text{Id}:N_Y\rightarrow N_Y$. It is unknown to the author whether the following result occurs in the literature exactly as stated, although several variations of it have appeared in (among others) \cite{MR0137127} and \cite{MR0206980}. The main point here is the method of proof, based on the Moser trick, which the author believes has not appeared anywhere before in this context.   
\begin{prop}\label{09:19:15}
Let $Y\subset X$ be a compact complex submanifold with the property that there exists a $1$-convex neighbourhood $U$ of $Y$ in $(NY,I_{lin})$, satisfying $H^1(U,T^{1,0}U)=0$. Then $Y$ admits a holomorphic tubular neighbourhood in $(X,I)$.
\end{prop}
\begin{proof}
Using a smooth tubular neighbourhood we may assume that $X=NY$ as smooth manifolds and that $I$ is a complex structure on $NY$ for which $Y\subset NY$ is a complex submanifold. 
Let $m_t:NY\rightarrow NY$ be multiplication by $t$ on the fibers. For small $t$ this induces a map from $U$ to itself and we consider the family of complex structures $m_t^\ast I:=I_t$ on $U$. As $t$ goes to zero, $I_t$ converges smoothly to the complex structure $I_0:=I_{lin}$ on $NY$. For each $t$ the time-derivative $\dot{I}_t$ defines a $\delbar_{I_t}$-closed element of $C^\infty(T^{\ast 0,1}_{I_t}U\otimes T^{1,0}_{I_t}U)$. Using the projections $T^{1,0}_{I_t}U\rightarrow T^{1,0}_{I_0}U=T^{1,0}U$, which are isomorphisms for small $t$, we can transfer everything to the fixed bundle $T^{1,0}U$. 
Since $H^1(U,T^{1,0}U)=0$, 
the time-dependent real vector field $Z_t$ on $U$, defined by
$$(Z_t)^{1,0}_{I_t}=\tfrac{i}2\delbar_t^\ast N_t\dot{I}_t,$$
depends in a $C^1$-manner on $t$ and satisfies $\mcL_{Z_t}I_t=-\dot{I}_t$ by part 5) of Thm.\ref{08:54:49}. The flow $\varphi_{t,0}$ of $Z_t$ is then $C^1$ and satisfies $\varphi_{t,0}^\ast I_t=I_0=I_{lin}$. Hence, the map $\psi_t:=\varphi_{t,0}^{-1}\circ m_{1/t}$ defines a $C^1$-embedding from a neighbourhood of $Y$ in $(NY,I)$ to a neighbourhood of $Y$ in $(NY,I_{lin})$ with complex-linear derivative.
Note that $\psi_t$ does not necessarily fix $Y$, but by choosing $t$ small enough we can arrange for $p\circ \psi_t :Y\rightarrow Y$ to be an isomorphism, where $p:NY\rightarrow Y$ is the projection. Then $\psi_t(Y)\subset NY$ is the image of a holomorphic section, which we may use to translate $\psi_t$ so that it fixes $Y$. Composing by a holomorphic automorphism of $(NY,I_{lin})$ if necessary we can also arrange for $d\psi_t|_Y=Id$. The inverse of $\psi_t$ defines the desired map $\varphi$. To see that $\varphi$ is smooth we first apply the above when $Y$ is a point, yielding a $C^1$-embedding $\varphi:U\rightarrow \C^n$ whose components satisfy $\delbar \varphi_j=0$. Standard elliptic theory then implies that the $\varphi_j$ are smooth when restricted to a smaller neighbourhood of the point $Y$. Knowing now that every point has a holomorphic coordinate chart around it, the same argument shows that any $C^1$-map between complex manifolds with complex-linear derivative is smooth.   
\end{proof}
\begin{rem}
Prop.\ref{09:19:15} has a version for generalized complex structures as well, the only problem being that the last step (upgrading from $C^1$ to $C^\infty$) does not work in that generality.  
\end{rem}
\begin{ex}
When $Y$ is a point, Prop.\ref{09:19:15} reduces to the Newlander-Nirenberg theorem \cite{MR0088770}, stating that every complex manifold (without boundary) admits a holomorphic atlas. 
\end{ex}
\begin{ex}
Let $X$ be a complex surface and suppose that $Y=\C\mathbb{P}^1$ with normal bundle $N_Y=\mathcal{O}(-1)$. Lem.\ref{11:16:39} implies that $Y$ has a $1$-convex neighbourhood $U$ in $N_Y$, and an explicit calculation on \v{C}ech-cohomology shows that\footnote{See Ex.\ref{10:38:53} for why \v{C}ech-cohomology agrees with $\delbar$-cohomology in this context. We abbreviate $TU=T^{1,0}U$.} $H^1(U,TU)=0$. 
\end{ex}
\begin{rem}
If $Y$ has codimension $1$ (which can be arranged by blowing up $Y$ in $X$), there are concrete ways of verifying whether $Y$ admits a $1$-convex neighbourhood $U$ in $NY$ (c.f.\ Lem.\ref{11:16:39}). In this case the condition $H^1(U,TU)=0$ can be studied via the exact sequences
$$ H^1(U,\mathfrak{m}^{i+1}_YTU)\rightarrow H^1(U,\mathfrak{m}^{i}_YTU)\rightarrow H^1(Y,(N_Y^\ast )^i\otimes TU|_Y),$$ 
where $\mathfrak{m}_Y$ is the ideal sheaf of holomorphic functions vanishing on $Y$. Together with a vanishing theorem for the groups $H^1(U,\mathfrak{m}^{i}_YTU)$ for large $i$ (see \cite{MR0206980}), it follows that $H^1(U,TU)=0$ whenever $H^1(Y,(N_Y^\ast )^i\otimes TU|_Y)=0$ for all $i\geq 0$. This will be the case for instance if $H^1(Y,(N_Y^\ast )^i\otimes TY)=H^1(Y,(N_Y^\ast )^{i-1})=0$ for all $i\geq 0$, because of the exact sequence
$$0\rightarrow TY\rightarrow TU|_Y \rightarrow N_Y\rightarrow 0.$$  
The condition $H^1(U,TU)=0$ is stronger than actually necessary, because the family $I_t$ in the proof of Prop.\ref{09:19:15} fixes both the complex structure on $Y$ and the holomorphic structure on $N_Y$. This allows one to relax the above conditions to the conditions $H^1(Y,(N_Y^\ast )^i\otimes TY)=H^1(Y,(N_Y^\ast )^{i})=0 $ for all $i\geq 1$. We refer to \cite{MR0137127} and \cite{MR0206980} for the precise statements and proofs.
\end{rem}

\subsubsection*{Holomorphic Poisson cohomology}

A holomorphic Poisson structure $\sigma$ induces, using the Schouten-Nijenhuis bracket, a differential $[\sigma,\cdot]$ on the space of holomorphic multivector fields $H^0(\Lambda^\bullet T^{1,0}M)$, whose cohomology we denote by $H^k(\sigma)$. 
\begin{prop}
Let $(M,I,\sigma)$ be a compact holomorphic Poisson manifold. If the Lie algebroid $L=T^{0,1}M\oplus \text{graph}(\sigma)$ is $q$-convex and if $H^r(M,\Lambda^{q-r-1} T^{1,0}M)=0$ for all $r>0$, then $H^q(\sigma)$ is finite dimensional.  
\end{prop}
\begin{proof}
The differential of $L$ is given by $d_L=\delbar+[\sigma,\cdot]$, a sum of two anti-commuting differentials. In the associated spectral sequence converging to $H^\bullet(L)$, the stated conditions imply that $H^q(\sigma)$ gets identified with a subspace of $H^q(L)$, which is finite dimensional.  
\end{proof}

\begin{ex}\label{10:38:53}
H\"ormander \cite{MR0179443} proved that if $(M,I)$ is $q$- and $q+1$-convex, then for any holomorphic vector bundle $E$ on $M$ we have $H^q(M,E)\cong H^q(\mathring{M},E|_{\mathring{M}})$. In particular, if $(M,I)$ has positive definite Levi-form and if $(\mathring{M},I)$ is Stein, then $H^r(M,\Lambda^{s} T^{1,0}M)=0$ for all $r>0$ and $s\geq 0$. For example, the unit ball in $\C^n$ has this property, showing that in Ex.\ref{09:10:02} the groups $H^q(\sigma)$ are finite-dimensional for $q=0$ and $3\leq q\leq 2k+2$.   
\end{ex}

\section{Elliptic regularity}\label{10:23:50}

To separate the analysis from the geometry we use this section to discuss Hodge decompositions for first order operators in general, returning to the setting of Lie algebroids in Sect.\ref{14:11:23}.

\subsection{First order differential operators}\label{17:07:25}

Let $M$ be a compact manifold with boundary $\del M$ and interior $\mathring{M}=M\backslash \del M$. For a complex vector bundle $E$ on $M$ and an open subset $U\subset M$ we write 
\begin{align*}
C^\infty(U, E), \hspace{10mm} C^\infty_{c}(U, E)
\end{align*}
for the smooth sections of $E$ on $U$ and those with support contained in $U$, respectively. We emphasize that if $U\cap \del M\neq \emptyset$, sections in $C^\infty_{c}(U, E)$ are not required to vanish on $\del M$. We will abbreviate $C^\infty(E):=C^\infty(M,E)$. Fix a Riemannian metric $g$ on $M$ and a fiberwise Hermitian metric on $E$, inducing an $L^2$-inner product $(\cdot, \cdot)$ on $C^\infty(E)$. For $k\in \mathbb{Z}_{\geq 0}$ we denote by $||\cdot ||_k$ the $k$-th Sobolev norm on $C^\infty(E)$, write $L^2_k(E)$ for the associated Hilbert space and abbreviate $L^2(E):=L^2_0(E)$. We will also use the $C^k$-norms $|\cdot|_k$, endowing $C^\infty(E)$ with a Fr\'echet-topology. Rellich's lemma and Soblev's theorem hold in this setting, giving compact inclusions $L^2_{k+1}(E)\hookrightarrow L^2_k(E)$, and estimates $|\cdot|_k\leq C||\cdot ||_{k+a}$ for $a>\frac{1}{2}\text{dim}_\R(M)$, respectively.

\begin{rem}
Throughout we will use the letter $C$ to denote a positive constant which is allowed to vary from one estimate to the next. Such constants often depend on background data such as the metrics on $M$ and $E$, as well as on the degrees of the Sobolev- or $C^k$-norms involved. When there is no risk of confusion we will omit these dependences. 
\end{rem}
  Let $P:C^\infty(E)\rightarrow C^\infty(E)$ be a first-order partial differential operator. We will view $P$ as an \textsl{unbounded operator} on $L^2(E)$, i.e.\ as a linear map which is only defined on the dense subspace $C^\infty(E)\subset L^2(E)$. 
 \begin{ex}
The relevant example to us is $E=\bigoplus_q \Lambda^qL^\ast\otimes V$ and $P=d^\nabla_L$, for $L$ a pre-Lie algebroid and $V$ a vector bundle endowed with an $L$-connection (see Sect.\ref{10:41:57}). 
 \end{ex} 
Using integration by parts one can construct a \textsl{formal adjoint} $P^\ast_f$ of $P$, which is another first-order differential operator on $E$ characterized by  
\begin{align*}
\hspace{25mm} (P\varphi_1,\varphi_2)=(\varphi_1,P^\ast_f \varphi_2), \hspace{10mm} \forall \varphi_1,\varphi_2\in C^\infty_c(\mathring{M},E). 
\end{align*}
Using the formal adjoint one can show that the closure of the graph of $P$ in $L^2(E)\times L^2(E)$ is again the graph of an unbounded operator, called the \textsl{closure} of $P$.
Abusing notation slightly, we continue to denote this closure by $P$ and denote by $\text{Dom}(P)\subset L^2(E)$ its domain of definition. The \textsl{Hilbert space adjoint} $P^\ast$ of $P$ is defined as follows. Its domain is   
\begin{align*}
\text{Dom}(P^\ast):=\{ \varphi\in L^2(E)|\ \psi\mapsto (P\psi, \varphi) \ \text{is continuous on}\ \text{Dom}(P)\}. 
\end{align*}
Since $\text{Dom}(P)\subset L^2(E)$ is dense we have, for every $\varphi\in \text{Dom}(P^\ast)$, a unique element $P^\ast \varphi\in L^2(E)$ characterized by $(P\psi, \varphi)=(\psi,P^\ast \varphi)$ for all $\psi\in \text{Dom}(P)$, which defines $P^\ast$ on its domain. The precise relations between $P,P_f^\ast$ and $P^\ast$ are summarized in the following two lemmas, whose proofs can be found in \cite{MR0461588} (Prop.(1.3.2) and the discussion leading to it).
\begin{lem}  \label{15:10:32}
Let $\nu$ be the outward unit normal to $\del M$ and set $\nu^\flat:=g(\nu)\in N^\ast_{\del M}$. Then
\begin{align}\label{23:07:17}
(P\varphi,\psi)=(\varphi,P^\ast_f \psi)+\int_{\del M} \langle \sigma(P,\nu^\flat)\varphi,\psi\rangle=(\varphi,P^\ast_f \psi)-\int_{\del M} \langle \varphi,\sigma(P^\ast_f,\nu^\flat)\psi\rangle 
\end{align} 
for all $\varphi,\psi \in C^\infty(E)$. Here $\sigma$ denotes the principal symbol. 
\end{lem}
\begin{rem}\label{11:43:09}
If $r\in C^\infty(M)$ satisfies $r|_{\mathring{M}}<0$, $r|_{\del M}=0$ and $|dr|=1$ on $\del M$, then $\nu^\flat=dr|_{\del M}$. One way to construct such a function $r$ is by considering geodesic distance to $\del M$ on a neighbourhood of $\del M$, and then extend it smoothly to the rest of $M$. We will often work with such a function $r$, and around a given boundary point complete it to a coordinate system of the form $(t_1,\ldots,t_{m-1},r)$.
\end{rem}
\begin{lem}\label{15:19:51}  
The space of smooth sections in the domain of $P^\ast$ is given by
\begin{align}\label{16:09:01}
\mcD:=\text{Dom}(P^\ast)\cap C^\infty(E)=\{\varphi\in C^\infty(E)| \ \sigma(P^\ast_f,\nu^\flat)\varphi=0 \},
\end{align}
on which $P^\ast$ is equal to $P^\ast_f$. 
\end{lem}
\begin{rem}
Lem.\ref{15:10:32} gives the boundary term due to integrating by parts, and Lem.\ref{15:19:51} identifies the smooth sections in $\text{Dom}(P^\ast)$ as those for which this boundary term vanishes. The condition on $\varphi$ in (\ref{16:09:01}) is also called a \textsl{(zeroth-order) Neumann boundary condition}.   
\end{rem}
The \textsl{Laplacian} of $P$ is the second order differential operator on $E$ defined by 
$$\Delta:=PP^\ast_f+P^\ast_fP,$$
which we consider as an unbounded operator on $L^2(E)$ with domain $C^\infty_c(\mathring{M},E)$. There is a procedure by Friedrichs \cite{MR1512905} that constructs a self-adjoint extension of $\Delta$, which goes as follows. On $\text{Dom}(P)\cap \text{Dom}(P^\ast)$ we consider the bilinear form 
\begin{align*}
Q(\varphi,\psi):=(\varphi,\psi)+(P\varphi,P\psi)+(P^\ast \varphi,P^\ast \psi).
\end{align*}
Since $P$ and $P^\ast$ are closed, $\text{Dom}(P)\cap \text{Dom}(P^\ast)$ is complete with respect to the associated norm $Q(\varphi):=Q(\varphi,\varphi)^{1/2}$. We denote by $\overline{\mcD}$ 
the $Q$-closure of $\mcD$ in $\text{Dom}(P)\cap \text{Dom}(P^\ast)$. 

For $\varphi\in L^2(E)$ the linear map $(\varphi,\cdot):\overline{\mcD}\rightarrow \C$ is continuous, hence there is a unique element $T\varphi\in \overline{\mcD}$ satisfying $(\varphi,\psi)=Q(T\varphi,\psi)$ for all $\psi\in \overline{\mcD}$. The resulting linear map $T:L^2(E)\rightarrow L^2(E)$ is injective, self-adjoint and satisfies $||T||\leq 1$. We denote by $F$ the inverse of $T$, which is an unbounded, self-adjoint operator with $\text{Dom}(F):=\text{Im}(T)\subset \overline{\mcD}$. Note that both $F$ and $\text{Dom}(F)$ are completely characterized by 
\begin{align}\label{15:26:59}
Q(\varphi,\psi)&=(F\varphi,\psi)  &&\forall \varphi\in \text{Dom}(F),  \psi\in \overline{\mcD}.
\end{align}
\begin{lem}  \label{13:48:22}
For every $\varphi\in \text{Dom}(F)\cap C^\infty(E)$ we have $F\varphi=(\Delta+1)\varphi$. In addition we have
\begin{align*}
\{\varphi\in C^\infty(E)| \ \varphi\in \mcD\ \text{and} \ P\varphi\in \mcD\}\subset \text{Dom}(F)\cap C^\infty(E),
\end{align*}
which is an equality if we make the extra assumption that $\sigma(P^2,\nu^\flat)=0$.
\end{lem}
\begin{proof}
See \cite[Prop.1.3.5]{MR0461588}.
\end{proof}
The unbounded operator $\square:=F-1$, whose domain equals that of $F$, is self-adjoint and extends the unbounded operator $\Delta$ defined on $C^\infty_c(\mathring{M},E)$. The additional assumption on the symbol of $P$ (which always holds in the setting of pre-Lie algebroids) implies that the smooth sections in the domain of $\square$ are exactly those that satisfy the zeroth- and first-order Neumann boundary conditions $\varphi\in \mcD$ and $P\varphi\in \mcD$, respectively. 
\begin{rem}\label{15:39:48}
Consider the unbounded operator $\widetilde{\square}:=PP^\ast+P^\ast P$, defined on 
$$\text{Dom}(\widetilde{\square}):=\{\varphi\in \text{Dom}(P)\cap \text{Dom}(P^\ast)|\ P\varphi\in \text{Dom}(P^\ast) \ \text{and} \ P^\ast\varphi\in \text{Dom}(P)\},$$
which also extends the unbounded operator $\Delta$ and satisfies $\text{Ker}(\widetilde{\square})=\text{Ker}(P)\cap \text{Ker}(P^\ast)$. 
Under the extra assumption that $P^2=0$ one can show that $\widetilde{\square}$ is self-adjoint (\cite[Prop.1.3.8]{MR0461588}).
\end{rem}

\subsection{Elliptic regularity and Hodge theory}\label{21:39:11}

In this subsection we discuss elliptic regularity and its main consequences. 
Below, $F$ denotes the self-adjoint extension of $\Delta+1$ defined in the previous section. 
\begin{defn}\label{17:01:25}
A first-order differential operator $P:C^\infty(E)\rightarrow C^\infty(E)$ satisfies \textsl{elliptic regularity} if the following two conditions hold:
\begin{itemize}
\item[i)] If $\varphi\in \text{Dom}(F)$ has the property that $F\varphi$ is smooth, then $\varphi$ is smooth as well. 
\item[ii)] For every $k\in \mathbb{Z}_{\geq 0}$ there exists a constant $C$ such that
\begin{align}\label{16:09:03}
|| \varphi ||_{k+1}\leq C ||F \varphi ||_k \hspace{15mm} \forall \varphi\in \text{Dom}(F)\cap C^\infty(E). 
\end{align}
\end{itemize}
\end{defn}
 We now list the main consequences of this definition. Throughout, we assume that $P$ satisfies elliptic regularity and we use the notation of Sect.\ref{17:07:25}. 

\begin{prop}\label{16:49:10}
The operator $T=F^{-1}$ induces bounded linear operators $T:L^2_k(E)\rightarrow L^2_{k+1}(E)$, hence $T:L^2(E)\rightarrow L^2(E)$ is compact and its spectrum is a discrete subset of $(0,1]$. 
\end{prop}
\begin{proof}
See \cite[Prop.3.1.1]{MR0461588}.
\end{proof}
\noindent Prop.\ref{16:49:10} induces an orthogonal decomposition $L^2(E)=\bigoplus_{0<\lambda \leq 1} V_\lambda$ into eigenspaces for $T$, with each $V_\lambda$ finite-dimensional. We have $V_\lambda\subset \text{Im}(T)=\text{Dom}(F)$, on which $F$ acts by $\lambda^{-1}$. 
\begin{lem}\label{13:27:34}
Suppose that $\varphi\in L^2(E)$ satisfies $(F-\lambda)\varphi\in C^\infty(E)$ for some $\lambda\in \C$. Then $\varphi\in C^\infty(E)$ as well. In particular, $V_\lambda\subset C^\infty(E)$. 
\end{lem}
\begin{proof}
See \cite[Prop.3.1.2]{MR0461588}.
\end{proof}
Recall the operator $\square=F-1$, whose domain equals that of $F$, and which is a self-adjoint extension of the Laplacian $\Delta$ defined on $C^\infty_c(\mathring{M},E)$.
\begin{prop}
There is an orthogonal decomposition of $L^2(E)$ into finite dimensional eigenspaces for $\square$. Each eigenspace consists of smooth sections, and the eigenvalues are non-negative without finite accumulation point. Moreover, for all $\varphi\in \text{Dom}(\square)\cap C^\infty(E)$ we have 
\begin{align}\label{12:35:01}
||\varphi ||_{k+1}^2\leq C (||\square \varphi||_k^2+||\varphi||^2).
\end{align}
\end{prop}
\begin{proof}
See \cite[Prop.3.1.11]{MR0461588}.
\end{proof}
\begin{cor}\label{07:54:08}
We have a closed, orthogonal decomposition 
\begin{align}\label{18:59:40}
L^2(E)=\text{Im}(\square)\oplus \text{Ker}(\square)
\end{align}
with $\text{Ker}(\square)\subset C^\infty(E)$ finite-dimensional.
\end{cor}
\begin{proof}
See \cite[Prop.3.1.12]{MR0461588}.
\end{proof}
\noindent The orthogonal decomposition (\ref{18:59:40}) gives rise to the \textsl{Neumann operator}  
$N :L^2(E)\rightarrow L^2(E)$
as follows. If we denote by $\pi$ the orthogonal projection onto $\text{Ker}(\square)$, then 
\begin{align*}
N\varphi:=\begin{cases} 0 & \text{if}\ \varphi\in \text{Ker}(\square) \\ (1-\pi)\psi & \text{if}\ \varphi=\square\psi. \end{cases} 
\end{align*}
It is readily verified that $N$ is a self-adjoint linear map, whose image lies in $\text{Dom}(\square)$. 

\begin{thm}\label{08:47:18} 
\
\begin{itemize}
\item[(1)] For every $\varphi\in L^2(E)$ we have $\varphi=\square N \varphi+\pi \varphi$, and if $\varphi\in \text{Dom}(\square)$ we have $\varphi=N \square \varphi+\pi \varphi$. Moreover, $N \pi =\pi N =0$ and $N (C^\infty(E)) \subset C^\infty(E)$.
\item[(2)] The operator $N$ is bounded, and induces bounded operators $N:L^2_k(E)\rightarrow L^2_{k+1}(E)$. 
\end{itemize}
\end{thm}
\begin{proof}
See \cite[Thm.3.1.14]{MR0461588}.
\end{proof}

\begin{rem}
If $\varphi\in C^\infty(E)$ then $N \varphi\in C^\infty(E)\cap \text{Dom}(\square)$, so $\square N \varphi=\Delta N\varphi$. Similarly, if $\varphi\in C^\infty(E)\cap \text{Dom}(\square)$ then $\square \varphi=\Delta \varphi$. Hence, in these cases we obtain the usual Hodge-decompositions $\varphi=\Delta  N\varphi+\pi\varphi=N \Delta\varphi+\pi\varphi$.
\end{rem}

\begin{rem}\label{09:03:42}
If $E$ is graded and $P$ is of degree $1$ then $\square$ is of degree $0$. If $P$ satisfies elliptic regularity only in a certain degree $q$, i.e.\ if Def.\ref{17:01:25} holds only for $\varphi$ of degree $q$, then all the conclusions of this section hold in degree $q$. This will be the setting for pre-Lie algebroids. 
\end{rem}

Finally, we provide some additional information in the setting when $P^2=0$. In this case we have another self-adjoint extension of $\Delta$ given by $\widetilde{\square}=PP^\ast+P^\ast P$ (see Rem.\ref{15:39:48}).
\begin{prop}
If $P^2=0$ we have $\square=\widetilde{\square}$, in particular their domains coincide. 
\end{prop}
\begin{proof}
See \cite[Prop.3.1.10]{MR0461588}.
\end{proof}
\noindent The decomposition (\ref{18:59:40}), combined with the facts that $\text{Ker}(\widetilde{\square})=\text{Ker}(P)\cap\text{Ker}(P^\ast)$ and $\text{Im}(P)\perp \text{Im}(P^\ast)$ (due to $P^2=0$), yields the following refinement of (\ref{18:59:40}).
\begin{cor}\label{14:27:15}
If $P^2=0$ we have a closed, orthogonal decomposition 
\begin{align*}
L^2(E)=\text{Ker}(P)\cap\text{Ker}(P^\ast)\oplus \text{Im}(P)\oplus \text{Im}(P^\ast).
\end{align*}
\end{cor}
The condition $P^2=0$ also allows us to define the cohomology group
$$H(P):=\frac{\text{Ker} (P:C^\infty(E)\rightarrow C^\infty(E))}{\text{Im}(P:C^{{}^\infty}(E)\rightarrow C^{{}^\infty}(E))},$$
which is graded if $E$ is graded and $P$ is of degree $1$.
\begin{prop}\label{11:56:27}
There is a natural isomorphism $\text{Ker}(P)\cap\text{Ker}(P^\ast)\cong H(P)$.
\end{prop}
\begin{proof}
The natural map $\text{Ker}(P)\cap\text{Ker}(P^\ast)\rightarrow H(P)$ is injective because $\text{Ker}(P^\ast)$ is orthogonal to $\text{Im}(P)$. Given $\varphi\in C^\infty(E)$ with $P\varphi=0$, we have 
\begin{align}\label{16:27:05}
\varphi=\Delta N\varphi+\pi\varphi=P_f^\ast P N\varphi+PP_f^\ast  N\varphi+\pi\varphi.
\end{align}
We know that $P N\varphi\in \text{Dom}(P^\ast)$ and that $0=P(PN\varphi)\in \text{Dom}(P^\ast)$, hence $P N\varphi\in \text{Dom}(\square)$ by Lem.\ref{13:48:22}. Since $P N\varphi$ is perpendicular to $\text{Ker}(\square)$, it follows that
$$||PN\varphi||=||N\square PN\varphi||\leq C ||\square PN\varphi||=C||\Delta PN\varphi||=0 $$
because $P$ commutes with $\Delta$ and $P\varphi=P\pi \varphi=0$. Substituting this into (\ref{16:27:05}), we see that $\varphi$ is cohomologous to $\pi\varphi\in\text{Ker}(P)\cap\text{Ker}(P^\ast)$.
\end{proof}
\begin{rem}\label{14:27:58}
The above proof shows that if $H(P)=0$, the bounded operator $P^\ast N:L^2(E)\rightarrow L^2(E)$ selects a $P$-primitive for every $P$-closed element, i.e.\ $\varphi=P(P^\ast N)\varphi$ for all $\varphi\in \text{Ker}(P)$.
\end{rem}

\subsubsection*{Families of differential operators}

As explained in Sect.\ref{16:33:30}, our aim is to study elliptic regularity for families of operators depending on an additional parameter. We summarize the notation of Sect.\ref{16:33:30} here in the current notation. Let $H$ be another vector bundle on $M$ and $P:C^\infty(E)\times C^\infty(H)\rightarrow C^\infty(E)$ a differential operator with the property that for each $\ve\in C^\infty(H)$ the map $P_\ve:=P(\cdot,\ve)$ is a linear first-order differential operator on $E$. We will interpret these as deformations of $P:=P_0$, and only consider $\ve$ in a small neighbourhood $B\subset C^\infty(H)$ of zero with respect to the $C^\infty$-topology. In this context we want to specialize (\ref{16:09:03}) to include the dependence on the parameter $\ve$ as well. We will put a subscript $\ve$ on all relevant operators (and their domains) to emphasize their dependence on $\ve$, and use the notation introduced in Sect.\ref{16:33:30}. 

\begin{defn}\label{16:12:56}
The family $P:C^\infty(E)\times B\rightarrow C^\infty(E)$ satisfies \textsl{elliptic regularity} if: 
\begin{itemize}
\item[i)] If $\varphi\in \text{Dom}(F_\ve)$ has the property that $F_\ve \varphi $ is smooth, then $\varphi$ is smooth as well. 
\item[ii)] There exists an integer $a\in \mathbb{Z}_{\geq 0}$ such that for every $k\in \mathbb{Z}_{\geq 0}$ and all $\ve\in B$ we have
\begin{align}\label{16:21:04}
|| \varphi ||_{k+1}\leq \mcL(|\ve|_{k+a}; ||F_\ve \varphi ||_k) \hspace{15mm} \forall \varphi\in \text{Dom}(F_\ve)\cap C^\infty(E).
\end{align}
\end{itemize}
\end{defn}
\noindent If we fix $\ve$ then (\ref{16:21:04}) reduces to (\ref{16:09:03}) with $C$ incorporating the norm $|\ve|_{k+a}$. In particular, all the results of this subsection hold for each individual $P_\ve$. The reason we care about (\ref{16:21:04}) is that it allows us to study the dependence of the Neuman operators $N_\ve$ on the parameter $\ve$ (see Sect.\ref{10:54:51}). As mentioned before we are interested in $P_\ve$ for small $\ve$ which means that we can impose an arbitrary, but ultimately fixed, bound on a finite $C^k$-norm of $\ve$ (i.e.\ shrink the neighbourhood $B$), and we will do so wherever possible to simplify the estimates. Of course we cannot impose a bound on all $C^k$-norms of $\ve$ as that does not result in an open neighbourhood of $0$ in $C^\infty(H)$. Finally, the precise value of the integer $a$ above is not important to us, what matters is that it is fixed and independent of $k$. 
 
\subsection{A criterion for elliptic regularity}\label{12:54:47}

In this subsection we specify a set of conditions for a family $P_\ve$ of first order operators that guarantee elliptic regularity. In Sect.\ref{14:11:23} we will relate these conditions to the notion of convexity of pre-Lie algebroids. To phrase the exact conditions on $P_\ve$ we need to establish some notation. 

Let $(U,x^i)\subset \mathring{M}$ be a coordinate chart in the interior, which we may assume satisfies $U\cong \R^m$ and on which $E$ is trivialized. We denote by $\mathcal{S}$ the space of Schwartz functions on $\R^m$, and by $\mathcal{F}:\mathcal{S}\rightarrow \mathcal{S}$ the Fourier transform defined by
\begin{align*}
\mathcal{F}\varphi(\xi)=\widehat{\varphi}(\xi):= \int_{\R^m} e^{-i\langle \xi, x \rangle} \varphi(x)dx
\end{align*}
where all relevant factors of $2\pi$ are absorbed in the volume form on $\R^m$. The Fourier transform allows us to define the Sobolev norm $||\cdot ||_s$ on $C^\infty_c(U,E)$ for any real number $s$ by 
\begin{align*}
||\varphi||_s^2:=\int _{\R^m} (1+|\xi|^2)^s |\widehat{\varphi}(\xi)|^2 d\xi.
\end{align*}
Here we use the trivialization of $E$ on $U$ to identify $\varphi\in C^\infty_c(U,E)$ with a tuple $(\varphi_1,\ldots, \varphi_l)$ of functions in $C^\infty_c(U)$, and write $|\widehat{\varphi}|^2=\sum_i |\widehat{\varphi_i}|^2$. In particular, $||\varphi ||_s$ depends on both the chart and the trivialization.
 We have $||\varphi||_s=||\Lambda^s \varphi||$, where $\Lambda^s:\mathcal{S}\rightarrow \mathcal{S}$ is defined by 
\begin{align}\label{15:40:45}
\widehat{\Lambda^s \varphi}(\xi)=(1+|\xi|^2)^{s/2} \widehat{\varphi}(\xi).
\end{align}
Next, let $U$ be a chart that is a neighbourhood of a point in $\del M$, which we may assume satisfies $U\cong \R^m_-$, where $\R^m_-:=\R^{m-1}\times (-\infty,0]$ with coordinates $(t^1,\ldots,t^{m-1},r)=:(t,r)$ (see Rem.\ref{11:43:09}), and on which $E$ is again trivialized. We will refer to such a chart as a \textsl{boundary chart}. Denote by $\mathcal{S}(\R^m_-)$ the space of Schwarz functions\footnote{By definition, $\mathcal{S}(\R^m_-)$ is the set of functions obtained by restricting elements of $\mathcal{S}$ to $\R^m_{-}$.}  on $\R^m_-$ and let $\widetilde{\mathcal{F}}:\mathcal{S}(\R^m_-)\rightarrow \mathcal{S}(\R^m_-)$ denote the \textsl{tangential Fourier transform}, defined by 
\begin{align}\label{16:15:54}
\widetilde{\mathcal{F}}\varphi(\tau,r)=\widetilde{\varphi}(\tau,r):=\int_{\R^{m-1}} e^{-i\langle \tau,t \rangle } \varphi(t,r)dt.
\end{align}
For $s\in \R$, define the operator $
\Lambda^s_\del:\mathcal{S}(\R^m_-)\rightarrow \mathcal{S}(\R^m_-)$ by
\begin{align}\label{15:51:58}
\widetilde{\Lambda^s_\del \varphi} (\tau,r)=(1+|\tau|^2)^{s/2} \widetilde{\varphi}(\tau,r), 
\end{align}
and the \textsl{tangential Sobolev norm} $||\cdot ||_{\del,s}$ by
\begin{align*}
|| \varphi ||^2_{\del,s}:=||\Lambda^s_\del\varphi||^2=\int_{\R^m_-} (1+|\tau|^2)^s |\widetilde{\varphi}(\tau,r)|^2 d\tau dr =\int_{-\infty}^0 ||\varphi(\cdot, r)||^2_s \ dr,
\end{align*}
which measures the $L^2$-norm of $\varphi$ and of its tangential derivatives up to $s$. Since we are dealing with first order partial differential operators that are not necessarily tangential, we need to define one additional norm. For $s\in \R$ we define the norm $||\text{D}(\cdot)||_{\del,s}$ on $C^\infty_c(U,E)$ by 
\begin{align}\label{09:49:34}
||\text{D}\varphi||^2_{\del,s}:=\sum_{i=1}^m || \del_i\varphi||^2_{\del,s}+ ||\varphi||^2_{\del,s}=||\varphi||^2_{\del,s+1}+||\del_r \varphi||^2_{\del,s}.
\end{align}
Concretely, if $s\in \mathbb{Z}_{\geq 0}$ then $||\text{D}\varphi||_{\del,s}$ measures the $L^2$-norm of all those partial derivatives of $\varphi$ up to order $s+1$ that include at most one $r$-derivative. 

We can now state the main theorem of this section.

\begin{thm}\label{13:30:27}
Let $P:C^\infty(E)\times B\rightarrow C^\infty(E)$ be a family of first-order differential operators on $E$, where $B\subset C^\infty(H)$ is a neighbourhood of $0$ in the $C^\infty$-topology such that:
\begin{itemize}
\item[1)] The Laplacians $\Delta_\ve=P^\ast_{\ve,f}P_\ve+P_\ve P^\ast_{\ve,f}$ are elliptic for all $\ve\in B$,
\item[2)] The subspaces 
$\text{Ker}(\sigma(P^\ast_{\ve,f},dr))\subset E_x$ have the same rank for all $x\in \del M$ and $\ve\in B$,
\item[3)] For every $x\in \del M$ there exists a boundary chart $U$ around $x$ and a constant $C$\ such that \begin{align}\label{08:34:06}
\hspace{10mm} ||\text{D}\varphi||^2_{\del,-\frac{1}{2}}\leq C Q_\ve(\varphi,\varphi) \hspace{10mm} \forall \varphi\in \text{Dom}(P_\ve^\ast)\cap C^\infty_c(U,E), \ \ve\in B.
\end{align}
\end{itemize}
Then there exists a neighbourhood $B'\subset B$ of $0\in C^\infty(H)$ such that $P:C^\infty(E)\times B'\rightarrow C^\infty(E)$ satisfies elliptic regularity.
\end{thm}
\begin{rem}\label{14:08:00}
In the specific case of $P=\delbar$, the above theorem is proved in \cite{MR0461588} (without considering families and the corresponding Leibniz estimates). As in Rem.\ref{09:03:42}, if $E$ is graded and $P_\ve$ are of degree $1$, the above theorem makes sense for every individual degree.
\end{rem}  
The proof of Thm.\ref{13:30:27} (see page \pageref{11:00:52}) requires some preliminary work. First, in order to reduce all calculations to coordinate charts we consider a local version of Def.\ref{16:12:56}.
\begin{defn} \label{13:16:26} 
The family $P:C^\infty(E)\times B \rightarrow C^\infty(E)$ satisfies \textsl{local elliptic regularity} if for all $\rho, \rho_1\in C^\infty(M)$ satisfying $\rho_1|_{\text{supp}(\rho)}=1$ we have:
\begin{itemize}
\item[i)] If $\varphi\in \text{Dom}(F_\ve)$ has the property that $\rho_1 F_\ve \varphi $ is smooth, then $\rho \varphi$ is smooth as well. 
\item[ii)] There exists an integer $a\in \mathbb{Z}_{\geq 0}$ such that for every $k\in \mathbb{Z}_{\geq 0}$ and all $\ve\in B$ we have
\begin{align}\label{12:11:54}
 ||\rho \varphi||_{k+1}\leq \mcL(|\ve|_{k+a};||\rho_1F_\ve \varphi||_k)+\mcL(|\ve|_{k+a};||F_\ve \varphi||)
\end{align}
for all $\varphi\in \text{Dom}(F_\ve)\cap C^\infty(E)$. 
\end{itemize}
\end{defn}
\begin{rem}\label{13:43:52}
Throughout we will abbreviate the condition $\rho_1|_{\text{supp}(\rho)}=1$ by $\rho\subset \rho_1$. 
\end{rem}
 We recover Def.\ref{16:12:56} by setting $\rho=\rho_1=1$, but Def.\ref{13:16:26} has the advantage of being local in the sense of the following lemma, whose proof follows immediately. 
\begin{lem}\label{13:37:02}
If $\{U_\alpha\}$ is an open cover of $M$ then it suffices to verify the conditions of Def.\ref{13:16:26} for functions $\rho\subset\rho_1$ with support contained in a single $U_\alpha$. 
\end{lem}
The aim is to show that the family $P_\ve$ of Thm.\ref{13:30:27} satisfies Def.\ref{13:16:26}. 
The most technical aspect of the proof consists of finding $Q_\ve$-adjoints for certain operators $A:L^2(E)\rightarrow L^2(E)$. 
Compared to \cite{MR0461588} there are some additional difficulties that need to be overcome. Firstly, since we are considering families of unbounded operators, all the operator domains are varying as well. Secondly, to obtain (\ref{12:11:54}) we need Leibniz rules for Sobolev norms. These are provided in the appendix, with a subtle difference between positive and negative Sobolev degrees (c.f.\ Prop.\ref{17:24:46}). Since these subtleties are not present in \cite{MR0461588}, we provide detailed proofs.

\begin{lem}\label{13:40:34}
For $\ve\in B$ fixed, let $A$ be an operator on $L^2(E)$ such that both $A$ and its Hilbert space adjoint $A^\ast$ are defined on $C^\infty(E)$, preserving both $C^\infty(E)$ and its subspace $\mcD_\ve=\text{Dom}(P_\ve^\ast)\cap C^\infty(E)$. 
Then for all $\varphi\in \mcD_\ve$ we have
\begin{align}\label{17:23:47}
Q_\ve(A\varphi,A\varphi)&=\text{Re}\ Q_\ve(\varphi,A^\ast A\varphi) + \underbrace{([P_\ve,A]\varphi,[P_\ve,A]\varphi)         \rule[-4pt]{0pt}{5pt}}_{{(1)}} +\underbrace{([P^\ast_\ve,A]\varphi,[P^\ast_\ve,A]\varphi)         \rule[-4pt]{0pt}{5pt}}_{(1^\ast)} \\
&+\text{Re}\Big( \underbrace{ (P_\ve \varphi,	(A^\ast-A)[P_\ve,A]\varphi)         \rule[-4pt]{0pt}{5pt}}_{(2)}	+\underbrace{(P_\ve  \varphi, [A^\ast-A,P_\ve]A\varphi)\rule[-4pt]{0pt}{5pt}}_{(3)}+\underbrace{(P_\ve \varphi,[A,[P_\ve,A]]\varphi)\rule[-4pt]{0pt}{5pt}}_{(4)}	\Big).			\nonumber\\
&+\text{Re}\Big(\underbrace{(P^\ast_\ve \varphi,	(A^\ast-A)[P^\ast_\ve,A]\varphi)         \rule[-4pt]{0pt}{5pt}}_{(2^\ast)}	+\underbrace{(P_\ve^\ast \varphi, [A^\ast-A,P^\ast_\ve]A\varphi)\rule[-4pt]{0pt}{5pt}}_{(3^\ast)}+\underbrace{(P^\ast_\ve \varphi,[A,[P^\ast_\ve,A]]\varphi)\rule[-4pt]{0pt}{5pt}}_{(4^\ast)}	\Big).			\nonumber
\end{align}
\end{lem}
\begin{proof}
Follows directly from writing out both sides, using only the definition of $Q_\ve$.
\end{proof}
\begin{rem}
The lemma basically says that we can bring $A$ to the other side of the inner product $Q_\ve$, at the cost of the eight additional terms appearing on the right of (\ref{17:23:47}). In the applications below $A$ will be of the form $\Lambda^k$ or $\Lambda^k_\del$ (see Sect.\ref{12:54:47}), and combined with Prop.\ref{17:24:46} and Prop.\ref{11:36:52} we will see that the additional terms in (\ref{17:23:47}) are all of \lq\lq lower order\rq\rq\ in $\varphi$, and can therefore be ignored in induction-based proofs.   
\end{rem}
We will use Lemma \ref{13:40:34} in the following three cases. 
\newline
{ \textbf{I)}} Let $U\cong \R^m_-$ be a boundary chart with coordinates $(t^1,\ldots,t^{m-1},r)$ on which $E$ is trivialized, and let $\rho,\rho_1\in C^\infty_c(U)$ satisfy $\rho\subset \rho_1$ (see Rem.\ref{13:43:52}). For $k\in \R$ fixed we define $A:=\rho_1\Lambda^k_\del \rho$, where $\Lambda^k_\del$ was defined in (\ref{15:51:58}). We want to apply Lem.\ref{13:40:34} to $A$, however $A$ does not necessarily preserve $\mcD_\ve$. Here is where condition 2) of Thm.\ref{13:30:27} comes in. Recall that 
\begin{align*}
\mcD_\ve\cap C^\infty_c(U,E)=\{\varphi\in C^\infty_c(U,E)| \ \sigma(P_{\ve,f}^\ast,dr)\varphi|_{\del M}=0 \}.
\end{align*}
For $\ve=0$ we can use condition 2) to pick the unitary trivialization $E|_U\cong U\times \C^l$ in such a way that $\text{Ker}(\sigma(P^\ast_{0,f},dr))$ coincides with $(U\cap \del M)\times (\C^{l'}\times \{0\})$ for some $0\leq l'\leq l$. For this choice of trivialization the space $\mcD_0$ is preserved by $A$ because $\Lambda^k_\del$ is tangential and acts diagonally on sections of $U\times \C^l$. For nonzero $\ve \in B$ we have the following lemma, whose proof follows immediately from condition 2) and contractibility of $B$. 
\begin{lem}\label{11:50:55}
There exists a smooth family of unitary automorphisms $\kappa_\ve:E\rightarrow E$ with the property that $\kappa_\ve(\mcD_0\cap C^\infty(E))=\mcD_\ve\cap C^\infty(E)$.
\end{lem}
Consequently, the operator 
$$A_\ve:= \kappa_\ve A \kappa^{-1}_\ve=\rho_1\kappa_\ve \Lambda^k_\del \kappa^{-1}_\ve\rho$$ preserves $\mcD_\ve$, as does $A^\ast_\ve$ because the latter is given by   
\begin{align}\label{14:42:49}
A^\ast_\ve =  \rho \kappa_\ve\big(\Lambda^k_\del  + \tfrac{1}{\sqrt{g}}  [\Lambda^k_\del,\sqrt{g}] \big) \kappa^{-1}_\ve\rho_1,
\end{align} 
where $\sqrt{g}dtdr$ denotes the Riemannian volume form of $(M,g)$ on the chart $U$.

\begin{lem}\label{10:28:10}
For sufficiently small $B\subset C^\infty(H)$ there exists an $a\in\mathbb{Z}_{\geq 0}$ such that 
\begin{align*}
Q_\ve(A_\ve \varphi,A_\ve \varphi)=\text{Re}\ Q_\ve(\varphi,A_\ve^\ast A_\ve \varphi)+\mathcal{L}(|\ve|^2_{\lceil k+a\rceil}; ||\text{D}\varphi||^2_{\del,k-1})
\end{align*}
for all $k\in \R_{\geq -\frac{1}{2}}$, $\ve\in B$ and $\varphi\in \mcD_\ve\cap C^\infty_c(U,E)$ (see (\ref{09:49:34}) for the definition of $||\text{D}(\cdot)||_{\del,k-1}$). 
\end{lem}

\begin{proof}
We need to estimate the remainder terms $\textbf{(1)}$-$\textbf{(4)}$ and $\textbf{(1}^\ast\textbf{)}$-$\textbf{(4}^\ast\textbf{)}$ in (\ref{17:23:47}). By Lem.\ref{15:19:51} and the fact that $A_\ve$ preserves $\mcD_\ve$, we can replace $P_\ve^\ast$ by $P_{\ve,f}^\ast$ everywhere in the estimates. This puts $P_\ve$ and $P_{\ve,f}^\ast$ on equal footing (both are first-order differential operators), and so by symmetry it suffices to discuss $\textbf{(1)}$-$\textbf{(4)}$. On the chart $U$ we can write $P_\ve=b_{\ve,i}\del_i+c_\ve$ (repeated indices are implicitly summed over), where $b_{\ve,i} $ and $ c_\ve $ are endomorphisms of $E$ that depend on $\ve$ and finitely many of its derivatives. We will also use
\begin{align}\label{16:00:32}
A_\ve=\rho \Lambda^k_\del + \rho_1 \kappa_\ve [\Lambda^k_\del,\kappa_\ve^{-1}\rho]=:\rho \Lambda^k_\del+A_\ve',
\end{align}
where $A'_\ve$ is of order $k-1$ due to Prop.\ref{11:36:52}$ii)$ (even though $\kappa_\ve$ is a matrix, $\Lambda^k_\del$ acts diagonally so $[\Lambda^k_\del,\kappa_\ve^{-1}\rho]$ is a matrix whose entries are commutators of functions and $\Lambda^k_\del$ to which \ref{11:36:52}$ii)$ applies). 
To bound $\textbf{(1)}$, we first use (\ref{16:00:32}) to estimate
\begin{align*}
||[P_\ve,A_\ve]\varphi||\leq & ||\rho[P_\ve,\Lambda^k_\del]\varphi||+||[P_\ve,\rho]\Lambda^k_\del\varphi||+||[P_\ve,A'_\ve] \varphi||  \\ \leq & ||\rho[P_\ve,\Lambda^k_\del]\varphi||+ \mathcal{L}(|\ve|_{\lceil k+a\rceil}; ||\text{D}\varphi||_{\del,k-1}).
\end{align*}
Note that we have to use $||\text{D}\varphi||_{\del,k-1}$ instead of $||\varphi||_{\del,k}$ because $P_\ve$ is not necessarily tangential. Next, we use that $\del_i$ and $\Lambda^k_\del$ commute together with (\ref{11:44:35}) to obtain  
\begin{align*}
||\rho[P_\ve,\Lambda^k_\del]\varphi||=&||\rho([b_{\ve,i} ,\Lambda_\del^k] \del_i+[c_\ve,\Lambda_\del^k]) \varphi|| 	
\leq  \mathcal{L}(|\ve|_{\lceil k+a\rceil}; ||\text{D}\varphi||_{\del,k-1}).
\end{align*}
This gives the desired bound for $\textbf{(1)}$. For $\textbf{(2)}$, we compute
\begin{align*}
(P_\ve \varphi, (A_\ve^\ast-A_\ve)[P_\ve,A_\ve]\varphi)=((A_\ve-A_\ve^\ast)P_\ve \varphi, [P_\ve,A_\ve]\varphi)	\leq || (A_\ve-A_\ve^\ast)P_\ve \varphi|| \cdot ||[P_\ve,A_\ve]\varphi||.
\end{align*} 
The last term was bounded above, while for the first term 
we can use (\ref{11:44:35}) together with
\begin{align}\label{16:46:01}
A_\ve-A_\ve^\ast=&\kappa_\ve(A-A^\ast)\kappa^{-1}_\ve=\kappa_\ve(\rho_1[\Lambda^k_\del, \rho] -\rho [\Lambda^k_\del, \rho_1] -\tfrac{1}{\sqrt{g}} \rho [\Lambda^k_\del,\sqrt{g}]\rho_1)\kappa_\ve^{-1}
\end{align} 
(see (\ref{14:42:49})) to obtain $|| (A_\ve-A_\ve^\ast)P_\ve \varphi||\leq \mcL(|\ve|_{\lceil k+a\rceil};||D\varphi||_{\del,k-1})$, giving the desired bound for $\textbf{(2)}$. For $\textbf{(3)}$ we compute 
\begin{align}\label{08:59:55}
(P_\ve \varphi,[A_\ve^\ast-A_\ve,P_\ve]A_\ve\varphi)=&(P_\ve \varphi, \big([A_\ve^\ast-A_\ve, b_{\ve,i} ]\del_i+ b_{\ve,i} [A_\ve^\ast-A_\ve,\del_i] +[A_\ve^\ast-A_\ve, c_\ve ]	\big)A_\ve\varphi)\nonumber\\
=& ([b_{\ve,i}^\ast,A_\ve-A_\ve^\ast]P_\ve \varphi, \del_iA_\ve\varphi)+( b_{\ve,i}^\ast P_\ve \varphi, [A_\ve^\ast-A_\ve,\del_i]A_\ve\varphi)\nonumber\\
&+([c_\ve^\ast,A_\ve-A_\ve^\ast  ]P_\ve \varphi, A_\ve\varphi).
\end{align} 
We will bound these last three terms separately. Using $(\varphi,\psi)\leq ||\varphi||_{\del,-s}||\psi||_{\del,s}$, valid for all $s\in \R$, we obtain
\begin{align*}
([b_{\ve,i}^\ast,A_\ve-A_\ve^\ast]P_\ve \varphi, \del_iA_\ve\varphi)\leq & || [b_{\ve,i}^\ast,A_\ve-A_\ve^\ast]P_\ve \varphi||_{\del,1} \cdot ||\del_iA_\ve\varphi ||_{\del,-1}.
\end{align*} 
Using (\ref{11:44:44}), we obtain  
\begin{align*}
||\del_iA_\ve\varphi ||_{\del,-1}=||\Lambda^{-1}_\del \del_iA_\ve\varphi ||_{}\leq & ||\del_i[\Lambda^{-1}_\del, \kappa_\ve\rho_1] \Lambda^k_\del \rho\kappa_\ve^{-1}\varphi||_{}	+||\del_i \kappa_\ve\rho_1 \Lambda^{k-1}_\del \rho\kappa_\ve^{-1}\varphi||_{}\\	\leq & \mcL(|\ve|_{\lceil k+a\rceil};||\text{D}\varphi||_{\del,k-1})	
\end{align*}
To obtain the same bound on $|| [b_{\ve,i}^\ast,A_\ve-A_\ve^\ast]P_\ve \varphi||_{\del,1}$ we need to show that $[b_{\ve,i}^\ast,A_\ve-A_\ve^\ast]$ is of order $k-2$. To this end, we use (\ref{16:46:01}) to write $A_\ve-A_\ve^\ast=A-A^\ast + R_\ve$, where 
\begin{align}
R_\ve:=\kappa_\ve(\rho_1[[\Lambda^k_\del, \rho],\kappa_\ve^{-1}] -\rho [[\Lambda^k_\del, \rho_1],\kappa_\ve^{-1}] -\tfrac{1}{\sqrt{g}} \rho [[\Lambda^k_\del,\sqrt{g}],\kappa_\ve^{-1}]\rho_1).
\end{align} 
Since $R_\ve$ is of order $k-2$ so is $[b_{\ve,i},R_\ve]$, while $[b_{\ve,i},A-A^\ast]$ is of order $k-2$ due to (\ref{11:45:07}). All in all this yields 
\begin{align*}
|| [b_{\ve,i}^\ast,A_\ve-A_\ve^\ast]P_\ve \varphi||_{\del,1}\leq \mcL(|\ve|_{\lceil k+a\rceil};||\text{D}\varphi||_{\del,k-1}),
\end{align*}
completing the bound on the first term in (\ref{08:59:55}). The third term in in (\ref{08:59:55}), involving $c_\ve$, is treated similarly. For the second term in (\ref{08:59:55}), we first write
\begin{align}\label{11:26:50}
( b_{\ve,i}^\ast P_\ve \varphi, [A_\ve^\ast-&A_\ve,\del_i]A_\ve\varphi)= ( [\kappa_\ve,\del_i]^\ast b_{\ve,i}^\ast P_\ve \varphi, (A^\ast-A)\kappa_\ve^{-1}A_\ve\varphi)\\ & +(\kappa^{-1}_\ve b_{\ve,i}^\ast P_\ve \varphi, [A^\ast-A,\del_i]\kappa_\ve^{-1}A_\ve\varphi)+( \kappa_\ve^{-1}b_{\ve,i}^\ast P_\ve \varphi, (A^\ast-A)[\kappa_\ve^{-1},\del_i]A_\ve\varphi).\nonumber
\end{align}
Note that there are no boundary contributions because $[\kappa_\ve,\del_i]$ is of order zero. We have 
\begin{align*}
( [\kappa_\ve,\del_i]^\ast b_{\ve,i}^\ast P_\ve \varphi, (A^\ast-A)\kappa_\ve^{-1}A_\ve\varphi)\leq & ||[\kappa_\ve,\del_i]^\ast b_{\ve,i}^\ast P_\ve \varphi ||_{\del,k-1}||(A^\ast-A)\kappa_\ve^{-1}A_\ve\varphi ||_{\del,1-k}\\
\leq & \mcL(|\ve|_{\lceil k+a\rceil}; ||\text{D}\varphi||_{\del,k-1}) \cdot ||\kappa_\ve^{-1}A_\ve\varphi ||_{\del,0}\\
\leq & \mcL(|\ve|^2_{\lceil k+a\rceil}; ||\text{D}\varphi||^2_{\del,k-1}).
\end{align*}
The other two terms in (\ref{11:26:50}) can be bound similarly, completing the bound on the second term in (\ref{08:59:55}) and finishing the bound on $\textbf{(3)}$.
Finally, for \textbf{(4)} we compute 
\begin{align}\label{09:17:21}
(P_\ve \varphi,[A_\ve,[A_\ve,P_\ve]]\varphi)=&(P_\ve \varphi, [A_\ve,[A_\ve, b_{\ve,i} ]]\del_i\varphi)+(P_\ve \varphi,2[A_\ve, b_{\ve,i} ][A_\ve,\del_i]\varphi)\nonumber\\&+(P_\ve \varphi,b_{\ve,i}[A_\ve,[A_\ve,\del_i]]\varphi)+(P_\ve \varphi,[A_\ve,[A_\ve, c_\ve ]]\varphi).
\end{align} 
Using the same steps as before we see that the second term in this expression satisfies 
$$(P_\ve \varphi,2[A_\ve, b_{\ve,i} ][A_\ve,\del_i]\varphi)=2([ b_{\ve,i}^\ast ,A_\ve^\ast]P_\ve \varphi,[A_\ve,\del_i]\varphi)\leq \mcL(|\ve|_{\lceil k+a\rceil}^2;||\text{D}\varphi||^2_{\del,k-1}).
$$
For the third term in (\ref{09:17:21}) we first write
\begin{align}
(P_\ve \varphi,b_{\ve,i} [A_\ve,[A_\ve,\del_i]]\varphi)=&(b_{\ve,i}^\ast P_\ve \varphi, \kappa_\ve[A,[A,\kappa_\ve^{-1}\del_i\kappa_\ve]]\kappa_\ve^{-1}\varphi)\nonumber \\
=&(\kappa_\ve^{-1}b_{\ve,i}^\ast P_\ve \varphi,([A,[A,\del_i]]+[A,[A,\kappa_\ve^{-1}[\del_i,\kappa_\ve]]]) \kappa_\ve^{-1}\varphi).\label{12:50:23}
\end{align} 
Using (\ref{11:44:51}) it follows that
\begin{align*}
(\kappa_\ve^{-1}b_{\ve,i}^\ast P_\ve \varphi,[A,[A,\del_i]]\kappa_\ve^{-1}\varphi)\leq & ||\kappa_\ve^{-1}b_{\ve,i}^\ast P_\ve \varphi||_{\del,k-1} ||[A,[A,\del_i]]\kappa_\ve^{-1}\varphi||_{\del,1-k}\\
\leq & \mcL(|\ve|_{\lceil k+a\rceil}^2;||\text{D}\varphi||^2_{\del,k-1}),
\end{align*} 
while the second term in (\ref{12:50:23}) can be written as (by expanding out the commutator)
\begin{align*}
(A^\ast \kappa_\ve^{-1}b_{\ve,i}^\ast P_\ve \varphi,[A,\kappa_\ve^{-1}[\del_i,\kappa_\ve]] \kappa_\ve^{-1}\varphi)-([A,\kappa_\ve^{-1}[\del_i,\kappa_\ve]]^\ast\kappa_\ve^{-1}b_{\ve,i}^\ast P_\ve \varphi, A\kappa_\ve^{-1}\varphi),
\end{align*} 
which can be bounded as above, using $(\cdot,\cdot)\leq ||\cdot||_{\del,-1}||\cdot||_{\del,1}$ for the first term and $(\cdot,\cdot)\leq ||\cdot||_{\del,0}||\cdot||_{\del,0}$ for the second. We have now bounded two of the terms in (\ref{09:17:21}). The remaining two require some additional work. Writing $A_\ve=\rho\Lambda^k_\del+A_\ve'$ as in (\ref{16:00:32}), we have
\begin{align}\label{13:51:30}
[A_\ve,[A_\ve, b_{\ve,i} ]]=&[\rho\Lambda^k_\del,[\rho\Lambda^k_\del,b_{\ve,i}]]+[[\rho\Lambda^k_\del,A_\ve'],b_{\ve,i}]+2[A_\ve',[\rho\Lambda^k_\del,b_{\ve,i}]]+[A'_\ve,[A'_\ve, b_{\ve,i} ]].
\end{align} 
The last two terms are commutators of operators of order $k-1$, which we can expand out and distribute inside the inner product $(P_\ve \varphi, [A_\ve,[A_\ve, b_{\ve,i} ]]\del_i\varphi)$ and subsequently bound with the same steps as before. The first two terms in (\ref{13:51:30}) need additional care. 
We have
\begin{align*}
[\rho\Lambda^k_\del,[\rho\Lambda^k_\del, b_{\ve,i} ]]=&\rho[\rho,[\Lambda^k_\del,b_{\ve,i}]]\Lambda^k_\del +\rho [\Lambda^k_\del,\rho][\Lambda^k_\del,b_{\ve,i}]+\rho^2[\Lambda^k_\del,[\Lambda^k_\del,b_{\ve,i}]].
\end{align*} 
All terms, except for the last one, contain a product of two separate terms involving $\Lambda^k_\del$. Bringing one of these to the other side of the inner product $(P_\ve \varphi,[\rho\Lambda^k_\del,[\rho\Lambda^k_\del, b_{\ve,i} ]]\del_i\varphi)$, we obtain the desired estimate by using similar arguments as before. This trick fails with the last term as it involves the double commutator $[\Lambda^k_\del,[\Lambda^k_\del, b_{\ve,i} ]]$. Write $\Lambda^k_\del=(\Lambda^{k/4}_\del)^4$, so that 
\begin{align*}
[\Lambda^k_\del,[\Lambda^k_\del, b_{\ve,i} ]]=\sum_{j,l=0}^3\Lambda_\del^{\frac{(j+l)k}{4}} [\Lambda_\del^{\frac{k}{4}},[\Lambda_\del^{\frac{k}{4}}, b_{\ve,i} ]]\Lambda_\del^{\frac{(6-j-l)k}{4}}.
\end{align*} 
All of the summands, except for $j+l=3$, can be redistributed inside the inner product $(P_\ve \varphi,[\rho\Lambda^k_\del,[\rho\Lambda^k_\del, b_{\ve,i} ]]\del_i\varphi)$ so that both sides contain four factors of $\Lambda^{k/4}_\del$, and these can be estimated as before. The terms resulting from $j+l=3$ are treated as follows:
\begin{align*}
(P_\ve \varphi,\rho^2\Lambda_\del^{\frac{3k}{4}}[\Lambda_\del^{\frac{k}{4}},[\Lambda_\del^{\frac{k}{4}}, b_{\ve,i} ]]\Lambda^{\frac{3k}{4}}_\del \del_i\varphi)
=& (\Lambda^{\frac{3k}{4}}_\del\rho^2P_\ve \varphi,[\Lambda_\del^{\frac{k}{4}},[\Lambda_\del^{\frac{k}{4}}, b_{\ve,i} ]]\Lambda_\del^{\frac{3k}{4}} \del_i\varphi)\\
\leq & || \Lambda_\del^{\frac{3k}{4}}\rho^2P_\ve \varphi ||_{\del,\frac{k}{4}-1} ||[\Lambda_\del^{\frac{k}{4}},[\Lambda_\del^{\frac{k}{4}}, b_{\ve,i} ]]\Lambda_\del^{\frac{3k}{4}} \del_i\varphi||_{\del,1-\frac{k}{4}}\\
\leq & C||P_\ve \varphi||_{\del,k-1}\cdot |b_{\ve,i}|_{\lceil\frac{k}{4}+a\rceil}\cdot ||\text{D}\varphi||_{\del,k-1}\\
\leq & \mcL(|\ve|^2_{\lceil k+a\rceil};||\text{D}\varphi||^2_{\del,k-1}),
\end{align*} 
where in the third step we used (\ref{11:44:59}). This completes the bound for the term in (\ref{09:17:21}) involving $b_{\ve,i}$.  
The same type of reasoning works for the term involving $c_\ve$, finishing the bound on \textbf{(4)} and thereby completing the proof.
\end{proof}

\noindent \textbf{II)} Let $(U,x^i)$ be a coordinate chart with $\overline{U}\cap \del M=\emptyset$ over which $E$ is trivialized, and let $\rho,\rho_1\in C^\infty_c(U)$ satisfy $\rho\subset \rho_1$. For $k\in \mathbb{Z}_{\geq 0}$ fixed we consider the operator $A:=\rho_1\Lambda^k \rho$ on $L^2(E)$. Since $\overline{U}\cap \del{M}=\emptyset$, both $A$ and $A^\ast=  \rho \big(\Lambda^k_\del  + \tfrac{1}{\sqrt{g}}  [\Lambda^k_\del,\sqrt{g}] \big) \rho_1$ preserve $\mcD_\ve$ for all $\ve\in B$. 

\begin{lem}\label{10:03:26}
For $B\subset C^\infty(H)$ sufficiently small there exists an $a\in\mathbb{Z}_{\geq 0}$ such that
\begin{align*}
Q_\ve(A\varphi,A\varphi)=\text{Re}\ Q_\ve(\varphi,A^\ast A\varphi)+\mathcal{L}(|\ve|^2_{k+a}; ||\varphi||^2_k) 
\end{align*}
for all $k\in \mathbb{Z}_{\geq 0}$, $\ve\in B$ and $\varphi\in \mcD_\ve$.
\end{lem}
\begin{proof}
Almost identical to the proof of Lem.\ref{10:28:10}; one only has to replace $||\cdot ||_{\del,s}$ and $\Lambda^k_\del$ by $||\cdot ||_s$ and $\Lambda^k$, respectively. Also, since $A$ and $A^\ast$ already preserve $D_\ve$, the proof simplifies because we can set $\kappa_\ve$ equal to the identity.  
\end{proof}

\noindent \textbf{III)} Finally we present a lemma that can be thought of as a special case of Lem.\ref{10:28:10} and Lem.\ref{10:03:26} for $k=0$, the main difference being that it is not restricted to a coordinate chart.
\begin{lem}\label{14:50:22}
For every real-valued function $\rho\in C^\infty(M;\R)$ we have
\begin{align*}
\hspace{15mm} Q_\ve(\rho \varphi,\rho \varphi)=\text{Re} \ Q_\ve(\varphi,\rho^2\varphi)+\mathcal{O}(||\varphi||^2) \leq CQ_\ve(\varphi,\varphi)\hspace{15mm} \forall \varphi\in \mcD_\ve . 
\end{align*}
Here $\mathcal{O}(||\varphi||^2)$ denotes a term that can be bounded by $C||\varphi||^2$ for some constant $C$. 
\end{lem}
\begin{proof}
The equality follows from Lem.\ref{13:40:34} by setting $A$ equal to multiplication by $\rho$, which is self-adjoint, bounded and preserves $\mcD_\ve$ by Lem.\ref{15:19:51}. For the inequality we compute
\begin{align*}
(P_\ve \varphi, P_\ve \rho^2 \varphi)=(P_\ve \varphi,[P_\ve,\rho^2] \varphi)+(P_\ve \varphi, \rho^2 P_\ve \varphi) &\leq C ||P_\ve \varphi|| ||\varphi||+ C||P_\ve \varphi||^2\\&\leq C ||\varphi||^2+C||P_\ve \varphi||^2\leq CQ_\ve(\varphi,\varphi).
\end{align*}
The same estimate holds for $P_\ve$ replaced by $P^\ast_\ve$, so the inequality  follows.
\end{proof}


Having dealt with all the integration by parts, we can discuss the bounds on interior- and boundary charts that are needed for the proof of Thm.\ref{13:30:27}. 

\subsubsection*{Interior charts}

Let $P_\ve$ a family of operators as in Thm.\ref{13:30:27}.
\begin{thm}\label{13:13:59} (\lq\lq G\"arding's inequality\rq\rq) 
For $B$ sufficiently small 
there is a constant $C$ such that $||\varphi||_1^2\leq C Q_\ve(\varphi,\varphi)$ for all $ \varphi\in C^\infty_c(\mathring{M},E)$ and $\ve\in B$. 
\end{thm}
\begin{proof}
This is a standard result for elliptic operators, see \cite[Thm.2.2.1]{MR0461588} for the case of the $\delbar$-operator. We can choose $C$ independent of $\ve\in B$ because the estimate only involves finitely many derivatives, so it suffices to impose a finite (but arbitrarily large) $C^k$-bound on $B$.   
\end{proof}
Let $U$ be a chart in the interior and let $\rho,\rho_1\in C^\infty_c(U)$ satisfy $\rho\subset \rho_1$.
\begin{prop}\label{11:38:30}
For $B\subset C^\infty(H)$ sufficiently small there exists an $a\in \mathbb{Z}_{\geq 0}$ such that 
\begin{align}
||\rho \varphi||^2_{k+2} \leq & \mcL(|\ve|_{k+a}^2; ||\rho_1F_\ve \varphi||_{k}^2)+\mcL(|\ve|_{k+a}^2;||F_\ve \varphi||^2)  	\label{14:46:19}
\end{align}
for all $k\geq 0$, $\varphi \in \text{Dom}(F_\ve)\cap C^\infty(E)$ and $\ve\in B$.
\end{prop}
\begin{proof} 
Using Thm.\ref{13:13:59}, Lem.\ref{14:50:22} and the bound $||\varphi||\leq ||F_\ve \varphi||$, we have
\begin{align}\label{18:19:44}
||\rho \varphi||_1^2\leq C Q_\ve(\rho \varphi,\rho \varphi)\leq C Q_\ve(\varphi,\varphi)=&C(F_\ve \varphi, \varphi)\leq C ||F_\ve \varphi||^2.
\end{align}
In the same way we get $||\rho_1\varphi||_1^2\leq C||F_\ve \varphi||^2$. Now we prove (\ref{14:46:19}) by induction on $k\geq 0$. For $k=0$, we use the equality $\rho=\rho\rho_1$ together with Prop.\ref{17:24:46} $ii)$ to obtain  
\begin{align*}
||\rho \varphi||_2^2=||\Lambda^1 \rho \varphi||^2_1=||\Lambda^1 \rho_1\rho \varphi||^2_1\leq ||[\Lambda^1,\rho_1]\rho \varphi||_1^2+ ||\rho_1\Lambda^1\rho \varphi||^2_1\leq C ||\rho \varphi ||_1^2 +  C Q_\ve(A\varphi,A\varphi),
\end{align*}
where $A:=\rho_1\Lambda^1\rho$. 
Here is where Lem.\ref{10:03:26} comes in, yielding (note that $A=A\rho_1$)
\begin{align*}
Q_\ve(A\varphi,A\varphi)= Q_\ve(A\rho_1\varphi,A\rho_1\varphi)&=\text{Re}\ Q_\ve(\rho_1\varphi,A^\ast A\rho_1\varphi)+\mcL(|\ve|_{a}^2;||\rho_1\varphi||_1^2).
\end{align*}
Using that $\rho\subset \rho_1$, together with the definition of $F_\ve$, we obtain 
\begin{align*}
Q_\ve(\rho_1\varphi,A^\ast A\rho_1 \varphi)= Q_\ve(\varphi,A^\ast A \varphi)=(F_\ve \varphi, A^\ast A \varphi)&= (A\rho_1 F_\ve\varphi, A \varphi) \leq ||A\rho_1F_\ve \varphi||_{-1}||A \varphi||_1\\
&\leq \tfrac{1}{\delta} ||\rho_1F_\ve \varphi||^2 + \delta ||\rho \varphi||_2^2
\end{align*}
where $\delta>0$ can be chosen arbitrarily. For $\delta$ small enough we obtain (\ref{14:46:19}) for $k=0$. 

Now suppose that (\ref{14:46:19}) holds in degrees $\leq k-1$ for any pair of functions $\rho_1\subset \rho_1$. Then
\begin{align*}
||\rho \varphi ||^2_{k+2}=||\Lambda^{k+1} \rho_1\rho \varphi ||^2_{1}\leq || [\Lambda^{k+1}, \rho_1]\rho \varphi||_1^2+|| \rho_1 \Lambda^{k+1}\rho \varphi||_1^2\leq C||\rho \varphi ||^2_{k+1} + C Q_\ve(A\varphi,A\varphi),
\end{align*}
where now $A:=\rho_1 \Lambda^{k+1} \rho$. The first term is bounded by induction, while for the second term we choose $\rho_2\in C^\infty_c(U)$ satisfying $\rho\subset\rho_2\subset \rho_1$ to compute
\begin{align*}
Q_\ve(A\varphi,A\varphi)= Q_\ve(A\rho_2\varphi,A\rho_2\varphi)&=\text{Re}\ Q_\ve(\rho_2\varphi,A^\ast A\rho_2\varphi)+\mcL(|\ve|_{k+a}^2;||\rho_2\varphi||_{k+1}^2).
\end{align*}
The last term can be bounded using the induction hypothesis (applied to the pair $(\rho_2,\rho_1)$ instead of $(\rho,\rho_1)$) together with Rem.\ref{16:50:34}, while the first term is treated as before;
\begin{align*}
Q_\ve(\rho_2 \varphi,A^\ast A \rho_2\varphi)=(F_\ve \varphi, A^\ast A \varphi)&=(A\rho_1 F_\ve \varphi, A \varphi)\leq 	||A\rho_1F_\ve \varphi||_{-1}||A \varphi||_1\\&\leq \tfrac{1}{\delta} ||\rho_1F_\ve \varphi||^2_k+\delta||\rho \varphi||^2_{k+2},
\end{align*}
which for $\delta>0$ sufficiently small yields (\ref{14:46:19}) in degree $k$. 
\end{proof}

\subsubsection*{Boundary charts}

Let $P_\ve$ again be a family of operators as in Thm.\ref{13:30:27} and $U$ a boundary chart.  
\begin{lem}\label{14:41:51}
Let $\{ \rho_k \}_{k\geq 1}$ be a sequence in $C^\infty_c(U)$ satisfying $\rho_{k+1}\subset\rho_k$. Then there exists an $a\in \mathbb{Z}_{\geq 0}$ such that for all $k\geq 1$, $\ve\in B$ and $\varphi \in \text{Dom}(F_\ve)\cap C^\infty(E)$ we have
\begin{align}\label{08:54:55}
|| \text{D} \rho_k \varphi ||_{\del,\frac{k}{2}-1} \leq \mcL(|\ve|_{\lceil \frac{k}{2}+a\rceil};||\rho_1 F_\ve \varphi||_{\del,\frac{k}{2}-1})+\mcL(|\ve|_{\lceil \frac{k}{2}+a\rceil};||F_\ve \varphi||). 
\end{align}
\end{lem}
\begin{proof}
By induction on $k\geq 1$. For $k=1$ we use assumption (\ref{08:34:06}) and Lem.\ref{14:50:22} to obtain
\begin{align*}
|| \text{D} \rho_1 \varphi ||^2_{\del,-\frac{1}{2}} \leq C Q_\ve(\rho_1 \varphi, \rho_1 \varphi) \leq C Q_\ve(\varphi,\varphi)
\leq   C||F_\ve \varphi||^2,
\end{align*}
proving (\ref{08:54:55}) for $k=1$. 
Assume (\ref{08:54:55}) holds in degrees $\leq k-1$ for some $k>1$. By definition,
\begin{align*}
||\text{D} \rho_k \varphi ||^2_{\del,\frac{k}{2}-1} =\sum_j ||\del_j \rho_k \varphi||^2_{\del,\frac{k}{2}-1} +||\rho_k \varphi||^2_{\del,\frac{k}{2}-1}. 
\end{align*}
Using $\rho_k=\rho_{k-1}\rho_k$ and $||\cdot ||_{\del,s}\leq || \text{D}(\cdot) ||_{\del,s-1}$, we see that the second term above satisfies 
\begin{align*}
||\rho_k \varphi||^2_{\del,\frac{k}{2}-1}\leq ||\text{D}\rho_{k-1}\rho_k \varphi||^2_{\del,\frac{k}{2}-2}\leq C||\text{D}\rho_{k-1}\varphi||^2_{\del,\frac{k-1}{2}-1},
\end{align*}
which can be bounded using the induction hypothesis. Next, observe that 
\begin{align}\label{15:02:20}
||\del_j \rho_k \varphi ||_{\del,\frac{k}{2}-1}=||\Lambda^{(k-1)/2}_\del \del_j \rho_k\varphi||_{\del,-\frac{1}{2}}=||\Lambda \del_j \rho_1\rho_k\varphi||_{\del,-\frac{1}{2}},
\end{align}
where $\Lambda:=\Lambda^{(k-1)/2}_\del$. To estimate this term, we first write 
\begin{align}\label{15:03:39}
\Lambda \del_j \rho_1\rho_k =[\Lambda,[\del_j,\rho_1]]\rho_k\rho_{k-1}+[\Lambda,\rho_1][\del_j,\rho_k]\rho_{k-1}+[\Lambda,\rho_1]\rho_k \del_j \rho_{k-1}+\del_j \rho_1 \Lambda \rho_k,
\end{align}
using that $\Lambda$ and $\del_j$ commute. Combining (\ref{15:02:20}) and (\ref{15:03:39}), we obtain
\begin{align*}
||\del_j \rho_k \varphi||_{\del,\frac{k}{2}-1}=||\Lambda \del_j \rho_1\rho_k \varphi||_{\del,-\frac{1}{2}}\leq&	C(||\text{D} \rho_{k-1}\varphi ||_{\del,\frac{k}{2}-2}+	 ||\text{D} \rho_1\Lambda\rho_k\varphi||_{\del,-\frac{1}{2}})\\
\leq & C(||\text{D} \rho_{k-1}\varphi ||_{\del,\frac{k-1}{2}-1}+||\text{D} A \varphi||_{\del,-\frac{1}{2}}),
\end{align*}
where $A:=\rho_1\Lambda\rho_k$. The first term is bounded by the induction hypothesis. Writing 
$
A=A_\ve + [A,\kappa_\ve]\kappa_\ve^{-1}$, 
where $A_\ve:=\kappa_\ve A \kappa_\ve^{-1}$ (see Lem.\ref{11:50:55}), we obtain 
\begin{align*}
||\text{D}A\varphi||_{\del,-\frac{1}{2}}\leq ||\text{D}A_\ve \varphi||_{\del,-\frac{1}{2}}+||\text{D}[A,\kappa_\ve]\kappa_\ve^{-1} \varphi||_{\del,-\frac{1}{2}}.
\end{align*}
The second term is of lower order in $\varphi$, but in order to use Prop.\ref{11:36:52} to obtain a Leibniz bound we need to get rid of the negative Sobolev norm. By definition we have 
$$||\text{D}[A,\kappa_\ve]\kappa_\ve^{-1} \varphi||_{\del,-\frac{1}{2}}=||\text{D}\Lambda^{-1/2}_\del[A,\kappa_\ve]\kappa_\ve^{-1} \varphi||_{\del,0},$$ and we we can write 
$$[A,\kappa_\ve]=\rho_1 (\Lambda^{1/2}_\del [\Lambda^{k/2-1}_\del,\kappa_\ve]+[\Lambda^{1/2}_\del,\kappa_\ve]\Lambda^{k/2-1}_\del)\rho_k.$$ 
In particular, 
\begin{align*}
||\text{D}[A,\kappa_\ve]\kappa_\ve^{-1} \varphi||_{\del,-\frac{1}{2}}=&||\text{D}\Lambda^{-\frac{1}{2}}_\del \rho_1 (\Lambda^{1/2}_\del [\Lambda^{k/2-1}_\del,\kappa_\ve]+[\Lambda^{1/2}_\del,\kappa_\ve]\Lambda^{k/2-1}_\del) \kappa_\ve^{-1} \rho_k\varphi||_{\del,0}\\
\leq & \mcL(|\ve|_{\lceil \frac{k}{2}+a\rceil}; ||\text{D}\rho_{k-1}\varphi||_{\del,\frac{k-1}{2}-1}),
\end{align*}
to which we can apply the induction hypothesis. Next, we use (\ref{08:34:06}) and Lem.\ref{10:28:10} to write 
\begin{align*}
||\text{D}A_\ve \varphi||^2_{\del,-\frac{1}{2}}&\leq CQ_\ve(A_\ve \rho_{k-1}\varphi,A_\ve \rho_{k-1}\varphi)\\&=C\text{Re} \ Q_\ve(\varphi,A^\ast_\ve A_\ve \varphi)+\mathcal{L}(|\ve|^2_{\lceil \frac{k}{2}+a\rceil }; ||\text{D}\rho_{k-1}\varphi||^2_{\del,\frac{k-1}{2}-1}).
\end{align*}
The second term is dealt with by the induction hypothesis, while the definition of $F_\ve$ yields
\begin{align}\label{15:25:09}
Q_\ve(\varphi,A^\ast_\ve A_\ve \varphi)= (F_\ve \varphi, A_\ve^\ast A_\ve \varphi)=(A_\ve \rho_1F_\ve \varphi,A_\ve \varphi)\leq \tfrac{1}{\delta} || A_\ve \rho_1F_\ve \varphi||^2_{\del,-\frac{1}{2}}+\delta || A_\ve \varphi||^2_{\del,\frac{1}{2}}   
\end{align}
for $\delta>0$ arbitrary. The first term is bounded by $\mcL(|\ve|^2_{\lceil \frac{k}{2}+a\rceil}; ||\rho_1F_\ve \varphi||^2_{\del,\frac{k}{2}-1})$, while
\begin{align*}
|| A_\ve \varphi||_{\del,\frac{1}{2}}\leq ||A\varphi||_{\del,\frac{1}{2}}+ || [A,\kappa_\ve]\kappa_\ve^{-1} \varphi||_{\del,\frac{1}{2}}\leq & C||\rho_k\varphi||_{\del,\frac{k}{2}}+\mcL(|\ve|_{\lceil \frac{k}{2}+a\rceil}; || \rho_{k-1}\varphi||_{\del,\frac{k}{2}-1})\\
\leq & C||\text{D}\rho_k \varphi||_{\del,\frac{k}{2}-1}+\mcL(|\ve|_{\lceil \frac{k}{2}+a\rceil }; || \text{D}\rho_{k-1}\varphi||_{\del,\frac{k}{2}-2}).
\end{align*}
Choosing $\delta$ in (\ref{15:25:09}) small enough and using the induction hypothesis we obtain (\ref{08:54:55}) in degree $k$ as well. 
\end{proof}
Let $\rho,\rho_1\in C^\infty_c(U)$ be functions satisfying $\rho\subset \rho_1$.
\begin{prop}\label{16:34:46}
For $B\subset C^\infty(H)$ sufficiently small there exists an $a\in \mathbb{Z}_{\geq 0}$ such that   
\begin{align}\label{14:55:12}
||  \rho  \varphi ||_{k+1} \leq \mcL(|\ve|_{k+a}; ||\rho_1 F_\ve \varphi||_k) + \mcL(|\ve|_{k+a}; ||F_\ve \varphi||) 
\end{align}
for all $k\geq 0$, $\varphi \in \text{Dom}(F_\ve)\cap C^\infty(E)$ and $\ve\in B$.
\end{prop}
\begin{proof}
By induction on $k$. For $k=0$, we use Lem.\ref{14:41:51} with $k=2$ and $\rho=\rho_2$ to obtain
\begin{align*}
||\rho \varphi||_1=||\text{D}\rho \varphi ||_{\del,0}\leq \mcL(|\ve|_{a}; ||F_\ve \varphi||).
\end{align*}
Suppose that (\ref{14:55:12}) holds in degrees $\leq k-1$ for some $k\geq 1$. Then  
\begin{align*}
||\rho \varphi||^2_{k+1}=\sum_{|\alpha|=k+1}||\del^\alpha \rho \varphi ||^2+ ||\rho \varphi||^2_k,
\end{align*}
and it suffices to estimate the terms $||\del^\alpha \rho \varphi ||^2$. If $\del^\alpha$ contains at most one $r$-derivative, then 
\begin{align*}
||\del^\alpha\rho \varphi||\leq || D\rho \varphi||_{\del,k}\leq \mcL(|\ve|_{k+a};||\rho_1 F_\ve \varphi||_k)+\mcL(|\ve|_{k+a};||F_\ve \varphi||),
\end{align*}
where we used Lem.\ref{14:41:51} (applied to a sequence $\rho=\rho_{2k+2}\subset\ldots \subset \rho_1$) together with the bound $||\cdot ||_{\del,k}\leq ||\cdot ||_k$. It therefore remains to bound the terms $||\del^\alpha_t \del^l_r \rho \varphi||^2$, where $l\geq 2$ and $|\alpha|+l=k+1$. Since $\varphi$ is smooth we have $F_\ve \varphi =(\Delta_\ve+1) \varphi$, which is an elliptic second-order operator and hence we can write 
\begin{align}\label{16:13:27}
F_\ve \varphi=A_\ve^{rr}\del_r^2\varphi+\sum_i A_{\ve,i}^{tr} \del_{t^i} \del_r\varphi+\sum_{i,j} A^{tt}_{\ve, ij} \del_{t^i}\del_{t^j}\varphi + B_\ve^{r}\del_r\varphi+\sum_i B^t_{\ve,i}\del_{t^i}\varphi+     C_\ve \varphi,
\end{align}
with $A_\ve^{rr}$ invertible. This allows us to write a derivative $\del^\alpha_t\del^l_r \rho \varphi$ with $l\geq 2$ in terms of derivatives of $\rho_1F_\ve\varphi$ of two degrees less and of derivatives of $\varphi$ that contain at most $l-1$ $r$-derivatives. An extra induction on $2\leq l \leq k+1$ then concludes the induction step in degree $k$ and therefore completes the proof. 
\end{proof}

\noindent \textsl{Proof of Thm.\ref{13:30:27}.} \label{11:00:52} Let $P_\ve$ be a family of first-order differential operators satisfying conditions 1), 2) and 3) of Thm.\ref{13:30:27}. We claim that Def.\ref{13:16:26} is satisfied. Note that i) of this definition is a statement about each individual operators $P_\ve$, so for this part it is unnecessary to consider families and the proof follows mutatis mutandis from \cite{MR0461588} after replacing $\delbar$ by $P$. Specifically, for $\rho,\rho_1$ with support in an interior chart, see Lemmas 2.2.7, 2.2.8 and Thm.2.2.9 in \cite{MR0461588}. For $\rho,\rho_1$ with support in a boundary chart, see \cite[ChII.3]{MR0461588} and \cite[ChII.5]{MR0461588}. For part ii) of Def.\ref{13:16:26}, if $\rho,\rho_1$ are supported in an interior chart then (\ref{12:11:54}) follows from Prop.\ref{16:34:46}, while for $\rho,\rho_1$ with support in a boundary chart the bound (\ref{12:11:54}) is a consequence of Prop.\ref{11:38:30}.  
\qed

\subsection{Uniform properties of the family of Neumann operators}\label{10:54:51}

Thm.\ref{08:47:18} establishes the main properties of the individual Neumann operators $N_\ve$. For the applications in \cite{NT} however, we need to understand some variational properties of $\mathcal{H}_\ve:=\text{Ker}(\Delta_\ve)$ as well as uniform bounds on the operator norms of $N_\ve$. Below we will assume that $P_\ve$ satisfies the conditions of Thm.\ref{08:47:18}, as well as the condition $\sigma(P^2,\nu^\flat)=0$ of Lem.\ref{13:48:22}. 
The proofs of (1) and (2) below are based on \cite[Lem.4.1.1]{MR0461588} and \cite[Lem.4.1.2]{MR0461588}, respectively.

\begin{prop}\label{08:17:09} Let $B'$ be as in Thm.\ref{13:30:27}. Then there exists a neighbourhood $B''\subset B'$ of $0$ in $C^\infty(H)$ such that:
\begin{itemize}
\item[(1)] The projection $\pi_{0}:\mathcal{H}_\ve\rightarrow \mathcal{H}_0$ is injective for all $\ve\in B''$.
\end{itemize}
Moreover, if $\mathcal{H}_0=0$ then there exists an integer $a\geq 0$ such that  
\begin{itemize}
\item[(2)] $||N_\ve \varphi||_{k+1}\leq \mcL(|\ve|_{k+a};||\varphi||_k)$ for all $k\geq 0$, $\varphi\in L^2_k(E)$ and $\ve\in B''$, 
\item[(3)] the map $B''\rightarrow \text{End}(L^2_k(E))$ given by $\ve\mapsto N_\ve$ is continuous with respect to the uniform operator topology on $\text{End}(L^2_k(E))$, for all $k\geq 0$. 
\end{itemize}
\end{prop}
\begin{proof}
(1): 
If (1) is false, we can find a sequence $\ve_j$ converging to $0$ and a sequence $\varphi_j\in \mathcal{H}_{\ve_j}$ satisfying $||\varphi_j||=1$ and $\pi_{0}(\varphi_j)=0$, which in particular implies that $Q_{\ve_j}(\varphi_j,\varphi_j)=(\varphi_j,\varphi_j)=1$. Using Thm.\ref{13:13:59} and condition (3) of Thm.\ref{13:30:27}, Rellich's lemma (c.f.\ Prop.\ref{18:03:02}) implies that a subsequence of $\varphi_j$ converges to 
some $\varphi\in L^2(E)$, satisfiesfying $||\varphi||=1$ and $\pi_{0}\varphi=0$. For any $\psi\in \text{Dom}(\square_0)\cap C^\infty(E)$ we have
\begin{align*}
(\varphi, \square_{0}\psi)=\lim_j(\varphi_j,\Delta_{0}\psi)=\lim_j(\varphi_j,(\Delta_0-\Delta_{\ve_j})\psi)+\lim_j(\varphi_{j},\Delta_{\ve_j}\psi).
\end{align*}
The first term vanishes because $||\varphi_j||$ is bounded and $\psi$ is smooth, so that $\Delta_{\ve_j}\psi$ converges to $\Delta_0\psi$. The second inner product we can rewrite using Lem.\ref{15:10:32}:
\begin{align*}
 (\varphi_j,\Delta_{\ve_j}\psi)=& (\varphi_j,P_{\ve_j}P^\ast_{\ve_j,f}\psi)+(\varphi_j,P^\ast_{\ve_j,f}P_{\ve_j}\psi)\\
 =&(P^\ast_{\ve_j,f}\varphi_j,P^\ast_{\ve_j,f}\psi)+(P_{\ve_j}\varphi_j,P_{\ve_j}\psi)+\int_{\del M} \langle \varphi_j,\sigma(P^\ast_{\ve_j,f},dr)P_{\ve_j}\psi	\rangle.
\end{align*}
There is only one boundary contribution because $\varphi_j\in \text{Dom}(P^\ast_{\ve_j})$. The first two terms vanish since $\varphi_j\in \mathcal{H}_{\ve_j}$, while the last integral converges to zero because $\sigma(P^\ast_{f,0},dr)P_0\psi=0$ on $\del M$ (since $\psi\in \text{Dom}(\square_0)\cap C^\infty(E)$), and because (using that $\varphi_j\in \mcH_{\ve_j}$)
\begin{align*}
\int_{\del M}|\varphi_j|^2\leq C ||\varphi_j||^2_1\leq C ||\varphi_j||^2=C.
\end{align*}
Hence, $\varphi$ is perpendicular to $\square_{0}(\text{Dom}(\square_0)\cap C^\infty(E))$. 
We claim that this implies that $\varphi$ is perpendicular to the entire image of $\square_{0}$, which gives a contraction because then we would have $\varphi\in \mcH_0=0$. Regarding this last claim, observe that by definition of $F_0$ we have 
\begin{align*}
\text{graph}(F_0)=\{(\psi,F_0\psi)|\ \psi\in \text{Dom}(F_0)\}=\{(T_0\zeta,\zeta)|\ \zeta\in L^2(E)\}. 
\end{align*}
Since $C^\infty(E)\subset L^2(E)$ is dense and preserved by $T_0$, which is a bounded operator,  
\begin{align*}
\{(\psi,F_0\psi)|\ \psi\in \text{Dom}(F_0)\cap C^\infty(E)\}\subset \text{graph}(F_0)
\end{align*}
is dense as well. Since $\square_{0}=F_0-1$, this proves the claim.

(2): 
First we need a uniform bound on the $L^2$-operator norms of $N_\ve$, 
which we argue again by contradiction. If not, there would be a sequence $\ve_j$ converging to $0$ and a sequence $\varphi_j\in C^\infty(E)$ satisfying $||\varphi_j||=1$ but $||N_{\ve_j}\varphi_j||\geq j$. Set $\psi_j:=N_{\ve_j}\varphi_j/||N_{\ve_j}\varphi_j||$, so that 
\begin{align*}
Q_{\ve_j}(\psi_j,\psi_j)=((\square_{{\ve_j}}+1)\psi_j,\psi_j)=\frac{(\varphi_j,N_{\ve_j}\varphi_j)}{||N_{\ve_j}\varphi_j||^2}+1\leq 2,
\end{align*} 
where we used that $\mathcal{H}_{\ve_j}=0$, which we know from part (1). As in the proof of (1), a subsequence of $\psi_j$ converges to $\psi$ in $L^2(E)$, and by assumption $\square_{\ve_j}\psi_j=\varphi_j/||N_{\ve_j}\varphi_j||$ converges to $0$ in $L^2(E)$. Let $\zeta\in \text{Dom}(\square_0)\cap C^\infty(E)$. Then, similar to the proof of (1), we have
\begin{align*}
(\psi,\square_{0}\zeta)=\lim_j (\psi_j,\Delta_0\zeta)=&\lim_j \big( (\psi_j,(\Delta_0-\Delta_{\ve_j})\zeta) + (\psi_j,\Delta_{\ve_j}\zeta)\big)\\
=&\lim_j\Big( (P_{\ve_j}\psi_j,P_{\ve_j} \zeta)+(P^\ast_{\ve_j,f}\psi_j,P^\ast_{\ve_j,f} \zeta)	+\int_{\del M} \langle \psi_j,\sigma(P^\ast_{\ve_j,f},dr)P_{\ve_j}\zeta \rangle		\Big)\\
=&\lim_j \big( Q_{\ve_j}(\psi_j,\zeta)-(\psi_j,\zeta)\big).	
\end{align*}
The integral term vanishes because $\sigma(P^\ast_{\ve_j,f},dr)P_{\ve_j}\zeta$ goes to zero while $||\psi_j||_1\leq C ||F_{\ve_j}\psi_j||\leq C ||\square_{\ve_j}\psi_j||+C||\psi_j||$ is uniformly bounded. Using Lem.\ref{11:50:55} we deduce that
\begin{align*}
(\psi,\square_{0}\zeta)=&\lim_j \big( Q_{\ve_j}(\psi_j,\kappa_{\ve_j}\zeta)-(\psi_j,\kappa_{\ve_j}\zeta)\big)+\lim_j \big( Q_{\ve_j}(\psi_j,\zeta-\kappa_{\ve_j}\zeta)-(\psi_j,\zeta-\kappa_{\ve_j}\zeta)\big)\\
=&\lim_j  (\square_{\ve_j}\psi_j,\kappa_{\ve_j}\zeta)=0,
\end{align*}
where we used that $Q_{\ve_j}(\psi_j,\zeta-\kappa_{\ve_j}\zeta)\leq C ||\psi_j||_1||\zeta-\kappa_{\ve_j}\zeta||_1$ which goes to zero as $j$ goes to infinity.
Again, we deduce that $\psi$ is orthogonal to $\text{Im}(\square_{0})$, so $\psi \in\mathcal{H}_0=0$ which is a contradiction.
Now we can prove (2). 
Using (\ref{16:21:04}), we compute 
\begin{align*}
||N_\ve  \varphi ||_{k+1}\leq \mcL(|\ve|_{k+a};||F_\ve N_\ve \varphi||_k)
\leq& \mcL(|\ve|_{k+a};||\square_{\ve} N_\ve \varphi||_k) +  \mcL(|\ve|_{k+a}; ||N_\ve  \varphi ||_k)\\
\leq& \mcL(|\ve|_{k+a};||\varphi ||_k)+  \mcL(|\ve|_{k+a}; ||N_\ve  \varphi ||_k).
\end{align*}
For $k=0$ we use the uniform $L^2$-operator bound on $N_\ve$, and the remaining cases follow by induction over $k$. 

(3): By homogeneity it suffices to verify continuity at $0\in B''$, which we prove again by contradiction. If the map $B''\rightarrow \text{End}(L^2_k(E))$ were not continuous at $0$ for some $k\geq 0$, then there would exist sequences $\ve_j\rightarrow 0$, $\varphi_j\in C^\infty(E)$ and a constant $\delta>0$ such that $||\varphi_j||_k=1$ and $||N_{\ve_j}\varphi_j-N_0\varphi_j||_k\geq \delta$ for all $j$. From (2) we obtain uniform bounds on $||N_{\ve_j}\varphi_j||_{k+1}$ and $||N_{0}\varphi_j||_{k+1}$, so that after passing to a subsequence we may assume that $N_{\ve_j}\varphi_j$ and $N_{0}\varphi_j$ converge to elements $\psi$ and $\widetilde{\psi}$ in $L^2_k(E)$, respectively, satisfying $||\psi-\widetilde{\psi}||_k\geq \delta$. 
On the other hand, for $\zeta\in \text{Dom}(\square_0)\cap C^\infty(E)$ we have
\begin{align*}
(\psi-\widetilde{\psi},\square_0\zeta)=&\lim_j (N_{\ve_j}\varphi_j-N_{0}\varphi_j,\square_0\zeta)\\
=&\lim_j \big((N_{\ve_j}\varphi_j,\Delta_{\ve_j}\zeta)+(N_{\ve_j}\varphi_j,(\Delta_0-\Delta_{\ve_j})\zeta)-(\varphi_j,\zeta) \big).
\end{align*}
We can now proceed as in (2), using the uniform bound $||N_{\ve_j}\varphi_j||_1\leq C(||\varphi_j||+||N_{\ve_j}\varphi_j||)\leq C$, to derive that $(\psi-\widetilde{\psi},\square_0\zeta)=0$ and hence $\psi-\widetilde{\psi}=0$, a contradiction. 
\end{proof}

\begin{rem}
If $\del M=\emptyset$ then Prop.\ref{08:17:09} simplifies considerably. In that case $\Delta_\ve:L^2_{k+2}(E)\rightarrow L^2_k(E)$ is a smooth family of bounded operators, which if $\mcH_0=0$ are all invertible for small $\ve$, implying that the family of inverses $N_\ve:L^2_{k}(E)\rightarrow L^2_{k+2}(E)$ depends smoothly on $\ve$. 
\end{rem}
Using the Sobolev embedding theorem we can convert Prop.\ref{08:17:09} into a statement about the Fr\'echet space $C^\infty(E)$. 
\begin{cor}\label{14:25:37}
There exists an integer $a\geq 0$ such that for every $k\geq 0$ we have  
\begin{align*}
|N_{\ve}\varphi |_k\leq \mcL(|\ve|_{k+a};| \varphi |_{k+a}),
\end{align*}
for all $\varphi \in C^\infty(E)$ and $\ve\in B''$. Moreover, given a continuous map $\ve\mapsto \varphi_\ve$ from $B''$ to $C^\infty(E)$, the map $\ve\mapsto N_\ve\varphi_\ve$ is also continuous from $B''$ to $C^\infty(E)$.  
\end{cor}
As remarked above, when $\del M\neq \emptyset$ we can only prove that $N:B''\rightarrow \text{End}(L^2_k(E))$ is continuous, while if $\del M=\emptyset$ we know that $N:B''\rightarrow \text{Hom}(L^2_k(E),L^2_{k+2}(E))$ is smooth. It is unknown to the author whether $N:B''\rightarrow \text{End}(L^2_k(E))$ is smooth when $\del M\neq \emptyset$, but we will discuss here a special case that is relevant for the Moser trick (c.f.\ Prop.\ref{09:19:15}). 
\begin{prop}
Let $t\mapsto \ve_t$ and $t\mapsto \varphi_t$ be $C^1$-maps from $[0,1]$ to $B''$ and $C^\infty(E)$, respectively, such that $\sigma(P^2_{\ve_t},dr)=0$, $P_{\ve_t}\varphi_t=0$ and $\mcH_{\ve_t}=0$ (here $r$ is as in Rem.\ref{11:43:09}). Then the map $t\mapsto N_{\ve_t}\varphi_t$ from $[0,1]$ to $C^\infty(E)$ is $C^1$ as well..
\end{prop} 
\begin{proof}
Abbreviate $N_t:=N_{\ve_t}$ and so forth. To study $N_{t}\varphi_t$ we would like to use the identity $\Delta_tN_{t}\varphi_t=\varphi_t$, but we need to be careful because $\Delta_t$ is unbounded.
Recall that 
\begin{align*}
\text{Dom}(\square_t)\cap C^\infty(E)=&\{\psi\in C^\infty(E)|\ \sigma(P_{t,f}^\ast,dr)\psi|_{\del M}=0=\sigma(P_{t,f}^\ast,dr)P_t\psi|_{\del M}\}=\mcD_t\cap P_t^{-1} \mcD_t
\end{align*} 
on which $\square_t=\Delta_t$. Lem.\ref{11:50:55} gives a smooth family of automorphisms $\kappa_t$ of $E$ satisfying $\kappa_0=\text{Id}$ and $\kappa_t(\mcD_0)=\mcD_t$. For $\psi\in \text{Dom}(\square_{t_0})\cap C^\infty(E)$ we claim that\footnote{Note that $\sigma(\Delta_t,dr)$ may only be invertible  around $\del M$, but we can extend it in an arbitrary (but smooth) way to the interior of $M$.}
\begin{align*}
A_{t,t_0}\psi:=\kappa_t\kappa_{t_0}^{-1}\psi-r\sigma(\Delta_t,dr)^{-1}\sigma(P_{t,f}^\ast,dr)P_t\kappa_t\kappa_{t_0}^{-1}\psi\in \  \text{Dom}(\square_{t})\cap C^\infty(E).
\end{align*} 
Indeed, $A_{t,t_0}\psi|_{\del M}=\kappa_t\kappa_{t_0}^{-1}\psi|_{\del M}$ shows that $A_{t,t_0}\psi\in\mcD_t$, while $P_tA_{t,t_0}\psi\in \mcD_t$ because 
\begin{align*}
\sigma(P_{t,f}^\ast,dr)P_tA_{t,t_0}\psi|_{\del M}=&\sigma(P_{t,f}^\ast,dr)P_t\kappa_t\kappa_{t_0}^{-1}\psi|_{\del M} \\ 
&-\sigma(P_{t,f}^\ast,dr)\sigma(P_{t},dr)\sigma(\Delta_t,dr)^{-1}\sigma(P_{t,f}^\ast,dr)P_t\kappa_t\kappa_{t_0}^{-1}\psi|_{\del M}
=0.
\end{align*} 
Here we used the assumption $\sigma(P^2_{\ve_t},dr)=0$ to deduce that $\sigma(\Delta_t,dr)$ commutes with $\sigma(P_t,dr)$ and $\sigma(P_{f,t}^\ast,dr)$. The map $t\mapsto A_{t,t_0}\varphi_{t_0}$ is $C^1$ from $B''$ to $C^\infty(E)$ for every fixed $t_0$, and we have $A_{t_0,t_0}\varphi_{t_0}=\varphi_{t_0}$ because of the assumption that $P_{t_0}\varphi_{t_0}=0$. We can now write
\begin{align}\label{18:01:30}
N_t\varphi_t=A_{t,t_0}N_{t_0}\varphi_t+N_t(1-\Delta_tA_{t,t_0}N_{t_0})\varphi_t,
\end{align} 
valid because $N_t\Delta_tA_{t,t_0}N_{t_0}\varphi_t=N_t\square_tA_{t,t_0}N_{t_0}\varphi_t=A_{t,t_0}N_{t_0}\varphi_t$. Even though $N_t$ appears on both sides, on the right it multiplies $(1-\Delta_tA_{t,t_0}N_{t_0})\varphi_t$, which depends $C^1$ on $t$ and vanishes when $t=t_0$ (see the proof of Prop.\ref{11:56:27} as to why $P_{t_0}N_{t_0}\varphi_{t_0}=0$).
Therefore, using (\ref{18:01:30}) and Cor.\ref{14:25:37}, we see that $N_t\varphi_t$ is differentiable at $t_0$ and we have
\begin{align}\label{18:14:52}
\left. \frac{d}{dt}\right|_{t=t_0}(N_t\varphi_t)=\left. \frac{d}{dt}\right|_{t=t_0} (A_{t,t_0}N_{t_0}\varphi_{t})
+N_{t_0}\left. \frac{d}{dt}\right|_{t=t_0} (1-\Delta_tA_{t,t_0}N_{t_0})\varphi_t.
\end{align} 
From Cor.\ref{14:25:37} again we deduce that the right hand-side is continuous with respect to $t_0$, hence the map $t\mapsto N_t\varphi_t$ is $C^1$. 
\end{proof}
\begin{rem}
It is tempting to try to conclude from (\ref{18:14:52}) by induction that $t\mapsto N_t\varphi_t$ is smooth whenever $t\mapsto \ve_t$ and $t\mapsto \varphi_t$ are. However, expressions like $N_{t}\dot{\varphi}_{t}$ that appear on the right in (\ref{18:14:52}) are continuous in $t$, but we can not conclude that they are $C^1$ (using the methods of this proof) because we do not have $P_t\dot{\varphi}_{t}=0$ in general. 
\end{rem}

\section{Proof of the main result}\label{14:11:23}  

In this section we will prove Thm.\ref{08:54:49}, by showing that the family of operators $P_\ve:=d_{L_\ve}$ on the bundle $\Lambda^{\bullet} L^\ast \otimes V$ satisfies the conditions of Thm.\ref{13:30:27}, combined with the results of sections \ref{21:39:11} and \ref{10:54:51}. We need some preparation. 
\newline
\newline
The main problem is verifying condition 3) of Thm.\ref{13:30:27}. As in \cite{MR0461588}, this will accomplished by introducing an auxiliary family of norms $E_\ve$ on $C^\infty(\Lambda^\bullet L^\ast \otimes V)$ and by showing that $|||\text{D}(\cdot)|||_{-1/2}\leq C E_\ve$ and $E_\ve\leq CQ_\ve$ for some constant $C$. We start by defining $E_\ve$. Let $\nabla$ be a $TM$-connection on $\Lambda^\bullet L^\ast \otimes V$, inducing $L_\ve$-connections 
$${\nabla}^\ve:\Lambda^q L^\ast \otimes V\rightarrow L^\ast\otimes \Lambda^{q} L^\ast \otimes V$$ 
defined by ${\nabla}^\ve_uv:=\nabla_{\rho_\ve(u)}v$ for all $u\in L$ and $v\in\Lambda^{q} L^\ast \otimes V$. Recall that $M$ is equipped with a Riemannian metric and $L$ and $V$ with Hermitian metrics.
\begin{defn}
The norm $E_\ve$ on $C^\infty(\Lambda^\bullet L^\ast \otimes V)$ is defined by
$$E_\ve(\varphi)^2:=||\nabla^\ve\varphi||^2+||\varphi||^2+\int_{\del M}|\varphi|^2$$
\end{defn}
We can think of $E_\ve$ as a weakened version of $||\cdot||_1$. Instead of measuring all first-derivatives, $E_\ve$ only measures derivatives in the directions of $\rho_\ve(L)\subset TM_\C$, and to compensate it also takes into account the boundary norm. One readily verifies that $E_\ve\lesssim ||\cdot ||_1$, and if $\rho_\ve$ is surjective then $E_\ve\sim ||\cdot ||_1$. Concretely, let $U\subset M$ be an open subset with frames $\{w_i\}_{i=1}^l$ and $\{e_\mu\}_{\mu=1}^{\text{rank}(V)}$ for $L$ and $V$, and let $\{\omega^i\}_{i=1}^l$ denote the dual frame for $L^\ast$. Every $\varphi\in C^\infty(U,\Lambda^qL^\ast\otimes V)$ can be written as
\begin{align}\label{13:31:10}
\varphi=\sum_{I, \mu} \varphi^\mu_I \ \omega^I\otimes e_\mu
\end{align}
for certain functions $\varphi^\mu_I\in C^\infty(U)$, where $I$ runs over all strictly increasing ordered subsets of $\{1,\ldots, l\}$ of length $|I|=q$, and $\omega^I:=\omega^{i_1}\wedge \ldots \wedge \omega^{i_q}$ if $I=\{i_1<\ldots <i_q\}$. Then
\begin{align}
E_{\ve}(\varphi)^2 \sim \sum_{I,\mu,i} || \rho_\ve(w_i)\cdot \varphi_I^\mu||^2  + ||\varphi||^2+ \int_{\del M} |\varphi|^2.
\end{align} 

We first consider the estimate $|||\text{D}(\cdot)|||_{-1/2}\leq C E_\ve$ on boundary charts, which only requires ellipticity of the pre-Lie algebroids $(L,d_{L_\ve})$. 

\begin{prop}\label{09:06:25}
For every $x\in \del M$ there exists a boundary chart $U$ containing $x$ and a constant $C$ such that $|||\text{D}\varphi|||_{-1/2}\leq C E_\ve(\varphi)$ for all $\ve\in B$ and $\varphi\in C^\infty_c(U,\Lambda^q L^\ast\otimes V)$.  
\end{prop}
\begin{proof}
Let $\{w_i\}_{i=1}^l$ be a unitary frame for $L|_U$ as above, and set $v_{i,\ve}:=\rho_\ve(w_i)$. By definition of $||\text{D}(\cdot)||_{\del,-1/2}$ (see (\ref{09:49:34})) and $E_\ve$ it suffices to prove the (slightly stronger) estimate
\begin{align}\label{08:23:20}
\sum_{i=1}^{m} ||\del_i f||^2_{\del,-1/2}\leq C \sum_{i=1}^l || v_{i,\ve} f||^2_{\del,-1/2}+C\int_{\del M} |f|^2 \hspace{10mm} \forall f\in C^\infty_c(U).
\end{align}
This is exactly what is proved in \cite[Thm. 2.4.5]{MR0461588} when the vector fields $v_1,\ldots,v_l$ are fixed, i.e.\ not varying in a family, so we only have to argue that $C$ can be chosen independently of $\ve\in B$. This follows because (\ref{08:23:20}) involves only finitely many derivatives, so it suffices to impose a bound on a finite $C^k$-norm on the elements of $B$. 
\end{proof}

Next, we consider the \lq\lq basic estimate\rq\rq\ $E_\ve\leq CQ_\ve$ (first introduced by Morrey \cite{MR0099060}), which forms the bridge between the analysis of Sect.\ref{10:23:50} and Lie algebroids. 
 
Let $U$ be a boundary chart, equipped with unitary frames $\{w_{\ve,i}\}_{i=1}^l$ and $\{e_\mu\}_{\mu=1}^{\text{rank}(V)}$ for $L$ and $V$ satisfying $\rho_\ve(w_{i,\ve})\in T\del M_\C$ for $i<l$ and 
\begin{align}\label{15:09:20}
\rho_\ve(w_{l,\ve})=\gamma_\ve\nu+T\del M_\C,
\end{align}
where $\nu$ denotes the outward normal vector to $\del M$ and $\gamma_\ve >0$. The conormal bundle $N^\ast_{\del M}$ is spanned by $\nu^\flat:=g(\nu)$, and by construction we have $\rho_\ve^\ast (\nu^\flat)=\gamma_\ve\omega_\ve^l$. Since the dual of wedging with $\omega_\ve^l$ is interior contraction by $w_{l,\ve}$, it follows from Lem.\ref{15:19:51} that, with respect to (\ref{13:31:10}), 
\begin{align}\label{14:24:20}
\text{Dom}(d^\ast_{L_\ve})\cap C^\infty_c(U,\Lambda^q L^\ast \otimes V)=\{\varphi \in C^\infty_c(U,\Lambda^q L^\ast \otimes V)| \ \varphi^\mu_{I,\ve}|_{\del M}=0 \ \text{if} \ l\in I\}.
\end{align}
We will also need explicit expressions for $d_{L_\ve}$ and $d_{L_\ve,f}^\ast$. Setting $v_{i,\ve}:=\rho_\ve(w_{i,\ve})$, we have 
\begin{align}\label{09:13:01}
d_{L_\ve}\varphi=&\sum_{I,\mu,i} (v_{i,\ve}\cdot\varphi^\mu_{I,\ve})\omega_\ve^i \wedge \omega_\ve^I\otimes e_\mu +\mathcal{O}(|\varphi|) = \sum_{IK,\mu,i} \delta^{iI}_K(v_{i,\ve}\cdot\varphi^\mu_{I,\ve})\omega_\ve^K\otimes e_\mu +\mathcal{O}(|\varphi|)
\end{align}
where $\delta^{iI}_K$ denotes the sign of the unique permutation taking $iI:=\{i,i_1,\ldots,i_q\}$ to the strictly increasing set $K$, or zero if there is no such permutation (in particular, it is zero for all $K$ if $i\in I$). Since the frames are unitary, it follows that 
\begin{align*}
d_{L_\ve,f}^\ast(\varphi)=-\sum_{IK,\mu,i} \delta^{iI}_K(\overline{v_{i,\ve}}\cdot\varphi^\mu_{K,\ve}) \omega_\ve^I\otimes e_\mu +\mathcal{O}(|\varphi|). 
\end{align*}
\begin{lem}\label{13:08:22}
Let $x\in \del M$ be a point on the boundary. 
\begin{itemize}
\item[i)] If $x$ is elliptic for $L_0$, then for $B$ sufficiently small there exists a boundary chart $U$ equipped with unitary frames as above, satisfying 
\begin{align}\label{14:44:23}
Q_\ve(\varphi,\varphi)=\sum_{I,\mu,i}  ||v_{i,\ve}\cdot \varphi_{I,\ve}^\mu||^2+ \mcO(E_\ve(\varphi)\cdot ||\varphi||)
\end{align}
for all $\varphi\in \text{Dom}(d^\ast_{L_\ve})\cap C^\infty_c(U,\Lambda^q L^\ast\otimes V)$ and $\ve\in B$.
Here $\mcO(E_\ve(\varphi)\cdot ||\varphi||)$ denotes a term that can be bounded from above by a uniform constant times $E_\ve(\varphi)\cdot ||\varphi||$. 
\item[ii)] If $x$ is non-elliptic for $L_0$, denote by $\lambda_{1},\ldots,\lambda_{l-1}$ the eigenvalues of the Levi form of $L_0$ at $x$, which we may arrange to be diagonal with respect to the frame $w_{1,0},\ldots,w_{l-1,0}$. Then for $B$ sufficiently small there exists a boundary chart $U$ such that\footnote{Since $\varphi_{I,\ve}^\mu|_{\del M}=0$ whenever $l\in I$, the sum over the eigenvalues is well-defined.} 
\begin{align*}
Q_\ve(\varphi,\varphi)=\sum_{I,\mu,i}  ||v_{i,\ve}\cdot \varphi_{I,\ve}^\mu||^2+\sum_{I,\mu}\sum_{i\in I} \lambda_{i} \int_{\del M} \gamma_\ve |\varphi^\mu_{I,\ve}|^2+R_\ve(\varphi)+\mathcal{O}(E_\ve(\varphi)\cdot ||\varphi||)
\end{align*}
for all $\varphi\in \text{Dom}(d^\ast_{L_\ve})\cap C^\infty_c(U,\Lambda^q L^\ast\otimes V)$ and $\ve\in B$. Here $\gamma_\ve$ is the constant of (\ref{15:09:20}), and $R_\ve(\varphi)$ denotes a term that can be bounded by $\delta E_\ve(\varphi)^2$ for $\delta$ as small as required by taking $U$ sufficiently small (independently of $\ve$). In addition we have, for each $i<l$, 
\begin{align}\label{14:16:51}
|| v_{i,\ve}\cdot \varphi_{I,\ve}^\mu||^2\geq -\lambda_{i} \int_{\del M} \gamma_\ve|\varphi_{I,\ve}^\mu|^2+ R_\ve(\varphi)+\mathcal{O}(E_\ve(\varphi)\cdot ||\varphi||).
\end{align}

\end{itemize}
\end{lem}
\begin{proof}
Using (\ref{09:13:01}) we obtain 
\begin{align*}
||d_{L_\ve}\varphi||^2=&\sum_{IJK,\mu,ij}  \delta^{iI}_K\delta^{jJ}_K (v_{i,\ve}\cdot \varphi_{I,\ve}^\mu,v_{j,\ve}\cdot \varphi_{J,\ve}^\mu) + \mcO(E_\ve(\varphi)\cdot ||\varphi||) \\
=&\sum_{I,\mu,i\notin I}  ||v_{i,\ve}\cdot \varphi_{I,\ve}^\mu||^2 + \sum_{IJK,\mu,i\neq j}  \delta^{iI}_K\delta^{jJ}_K (v_{i,\ve}\cdot \varphi_{I,\ve}^\mu,v_{j,\ve}\cdot \varphi_{J,\ve}^\mu) +\mcO(E_\ve(\varphi)\cdot ||\varphi||).
\end{align*}
For $i\neq j$, $\delta^{iI}_K\delta^{jJ}_K$ can only be nonzero if $I=\langle jM\rangle$ and $J=\langle iM\rangle$ for $M:=K\backslash\{i,j\}$, where for a set $Z$ we denote by $\langle Z\rangle$ the strictly ordered set having the same elements as $Z$. A quick inspection shows that in that case $\delta^{iI}_{K}\delta^{jJ}_{K}=-\delta^{I}_{jM }\delta^{J}_{ iM }$, hence
\begin{align*}
||d_{L_\ve}\varphi||^2=&\sum_{I,\mu,i}  ||v_{i,\ve}\cdot \varphi_{I,\ve}^\mu||^2-\sum_{IJM,\mu,ij}  \delta^{I}_{jM}\delta^{J}_{iM} (v_{i,\ve}\cdot \varphi_{I,\ve}^\mu,v_{j,\ve}\cdot \varphi_{J,\ve}^\mu)	 + \mcO(E_\ve(\varphi)\cdot ||\varphi||).
\end{align*}
Note that the terms with $i=j$ in the second summation cancel the terms with $i\in I$ in the first one. Using integration by parts, we can write
\begin{align}
(v_{i,\ve}\cdot \varphi_{I,\ve}^\mu,v_{j,\ve}\cdot \varphi_{J,\ve}^\mu)=& -(\overline{v_{j,\ve}}\cdot v_{i,\ve}\cdot  \varphi_{I,\ve}^\mu, \varphi_{J,\ve}^\mu)+\mcO(E_\ve(\varphi)\cdot ||\varphi||)\nonumber\\
=&-([\overline{v_{j,\ve}},v_{i,\ve}]\cdot\varphi_{I,\ve}^\mu, \varphi_{J,\ve}^\mu)-(v_{i,\ve}\cdot \overline{v_{j,\ve}}\cdot \varphi_{I,\ve}^\mu, \varphi_{J,\ve}^\mu)+\mcO(E_\ve(\varphi)\cdot ||\varphi||).\nonumber
\end{align}
The reason why there are no boundary terms above is the following. If $j<l$, then $v_{j,\ve}$ is tangent to $\del M$, while if $j=l$ and $i<l$, then $\varphi_{I,\ve}^\mu|_{\del M}=0$ because $j\in I$ and therefore so is $v_{i,\ve}\cdot \varphi_{I,\ve}^\mu$ as $v_{i,\ve}$ is tangent to $\del M$. Finally, if $i=j=l$ then $\varphi_{I,\ve}^\mu|_{\del M}=\varphi_{J,\ve}^\mu|_{\del M}=0$ by (\ref{14:24:20}).

We can apply another integration by parts to the term $(v_{i,\ve}\cdot \overline{v_{j,\ve}}\cdot \varphi_{I,\ve}^\mu, \varphi_{J,\ve}^\mu)$, using the same arguments to conclude that there are no boundary terms. One of the lower order terms that appears in this process is of the form $(\overline{v_{j,\ve}}\cdot \varphi_{I,\ve}^\mu, \varphi_{J,\ve}^\mu)$, which is not $\mcO(E_\ve(\varphi)\cdot ||\varphi||)$, but it will be after applying yet another integration by parts. We then end up with
\begin{align*}
||d_{L_\ve}\varphi||^2=&\sum_{I,\mu,i}  ||v_{i,\ve}\cdot \varphi_{I,\ve}^\mu||^2-\sum_{IJM,\mu,ij}  \delta^{I}_{jM}\delta^{J}_{iM} (\overline{v_{j,\ve}}\cdot \varphi_{I,\ve}^\mu,\overline{v_{i,\ve}}\cdot \varphi_{J,\ve}^\mu)	 \\ &+ \sum_{IJM,\mu,ij}  \delta^{I}_{jM}\delta^{J}_{iM} ([\overline{v_{j,\ve}},v_{i,\ve}]\cdot \varphi_{I,\ve}^\mu,\varphi_{J,\ve}^\mu)+ \mcO(E_\ve(\varphi)\cdot ||\varphi||).
\end{align*}
The second term coincides exactly with $-||d_{L_\ve}^\ast\varphi||^2+\mcO(E_\ve(\varphi)\cdot ||\varphi||)$, which implies that
\begin{align*}
Q_\ve(\varphi,\varphi)=\sum_{I,\mu,i}  ||v_{i,\ve}\cdot \varphi_{I,\ve}^\mu||^2 +  \sum_{IJM,\mu,ij}  \delta^{I}_{jM}\delta^{J}_{iM} ([\overline{v_{j,\ve}},v_{i,\ve}]\cdot \varphi_{I,\ve}^\mu,\varphi_{J,\ve}^\mu)+ \mcO(E_\ve(\varphi)\cdot ||\varphi||).
\end{align*}
Following Rem.\ref{12:51:01}, we expand 
\begin{align*}
[\overline{v_{j,\ve}},v_{i,\ve}]=\sum_{k=1}^la^k_{ij,\ve} v_{k,\ve}+\sum_{k=1}^lb^k_{ij,\ve} \overline{v_{k,\ve}}
\end{align*}
for some (not necessarily unique) functions $a^k_{ij,\ve},b^k_{ij,\ve}$. When inserted inside the inner product $([\overline{v_{j,\ve}},v_{i,\ve}]\cdot \varphi_{I,\ve}^\mu,\varphi_{J,\ve}^\mu)$, the terms involving $a^k_{ij,\ve}$ are $\mathcal{O}(E_\ve(\varphi)\cdot||\varphi||)$, while those involving $b^k_{ij,\ve}$ for either $k<l$, $i=l$ or $j=l$ can be integrated by parts (all boundary terms vanish), after which they become $\mathcal{O}(E_\ve(\varphi)\cdot||\varphi||)$. The only terms defying this procedure are those with $i,j<l$ and $k=l$. 

If $x$ is elliptic for $L_0$, then for $B$ sufficiently small there exists a boundary chart $U$ around $x$, all of whose points are elliptic for $L_\ve$ for every $\ve\in B$. As explained in Rem.\ref{12:51:01}, in this case we can arrange for $b^l_{ij,\ve}$ to vanish on $\del M$. In particular, the problematic terms involving $b^l_{ij,\ve}$ can also be integrated by parts without inducing boundary terms, proving i).

If $x$ is non-elliptic for $L_\ve$, Rem.\ref{12:51:01} implies that $b^l_{ij,\ve}=\lambda_{i}\delta_{ij}+c_{ij,\ve}$ for $i,j<l$, where $c_{ij,\ve}\in C^\infty(U)$ depends continuously on $\ve$ and vanishes at $x$ when $\ve=0$. We have
\begin{align*}
\big(\lambda_{i}\delta_{ij}\overline{v_{l,\ve}}\cdot \varphi^\mu_{I,\ve},\varphi_{J,\ve}^\mu\big)
=&-\lambda_{i}\delta_{ij}(\varphi^\mu_{I,\ve},v_{l,\ve}\cdot \varphi^\mu_{J,\ve})+\lambda_{i}\delta_{ij}\int_{\del M} |\varphi^\mu_{I,\ve}|^2 \iota_{\overline{v_{l,\ve}}} \text{vol}_g +\mathcal{O}(||\varphi ||^2)\\
=&\lambda_{i}\delta_{ij}\int_{\del M} \gamma_\ve |\varphi^\mu_{I,\ve}|^2+\mathcal{O}(E_\ve(\varphi)\cdot ||\varphi ||), 
\end{align*} 
using (\ref{15:09:20}). 
Similarly, we obtain
\begin{align*}
\big(c_{ij,\ve} \overline{v_{k,\ve}} \varphi^\mu_{I,\ve},\varphi^\mu_{J,\ve}\big)= \int_{\del M} \gamma_\ve c_{ij,\ve} \varphi^\mu_{I,\ve}  \overline{\varphi^\mu_{J,\ve}}+\mathcal{O}(E_\ve(\varphi)\cdot ||\varphi ||).
\end{align*} 
Now $\gamma_\ve$ depends smoothly on $\ve$ and can be bounded uniformly from above, so that 
\begin{align*}
|\int_{\del M}  \gamma_\ve c_{ij,\ve}\varphi^\mu_{I,\ve} \overline{\varphi^\mu_{J,\ve}}|\leq C (\sup_U  |c_{ij,\ve} |) E_\ve(\varphi)^2,
\end{align*} 
where $\text{sup}_U| c_{ij,\ve} |$ can be made arbitrarily small by taking $U$ and $B$ sufficiently small. This proves the first statement of ii). For the second statement we compute
\begin{align*}
||v_{i,\ve}\cdot \varphi^\mu_{I,\ve}||^2=&-(\overline{v_{i,\ve}} \cdot v_{i,\ve}\cdot \varphi^\mu_{I,\ve},\varphi^\mu_{I,\ve})+\mathcal{O}(E_\ve(\varphi)\cdot ||\varphi||)\\
=&-([\overline{v_{i,\ve}},v_{i,\ve}]\cdot \varphi^\mu_{I,\ve},\varphi^\mu_{I,\ve})-(v_{i,\ve}\cdot\overline{v_{i,\ve}}\cdot \varphi^\mu_{I,\ve},\varphi^\mu_{I,\ve})	+\mathcal{O}(E_\ve(\varphi)\cdot ||\varphi||)\\
=&-([\overline{v_{i,\ve}},v_{i,\ve}]\cdot \varphi^\mu_{I,\ve},\varphi^\mu_{I,\ve})+(\overline{v_{i,\ve}}\cdot \varphi^\mu_{I,\ve},\overline{v_{i,\ve}}\cdot \varphi^\mu_{I,\ve})+\mathcal{O}(E_\ve(\varphi)\cdot ||\varphi||)\\
\geq &-([\overline{v_{i,\ve}},v_{i,\ve}]\cdot \varphi^\mu_{I,\ve},\varphi^\mu_{I,\ve})+\mathcal{O}(E_\ve(\varphi)\cdot ||\varphi||)\\
=&-( (\lambda_{i}+c_{ii,\ve})\overline{v_{l,\ve}}\cdot \varphi^\mu_{I,\ve},\varphi^\mu_{I,\ve})	+\mathcal{O}(E_\ve(\varphi)\cdot ||\varphi||)\\
=&-\lambda_{i}\int_{\del M}\gamma_\ve|\varphi^\mu_{I,\ve}|^2-\int_{\del M}\gamma_\ve c_{ii,\ve}|\varphi^\mu_{I,\ve}|^2	+\mathcal{O}(E_\ve(\varphi)\cdot ||\varphi||). 
\end{align*} 
Using the same arguments as before we arrive at (\ref{14:16:51}). 
\end{proof}

\begin{prop}\label{14:03:13}
If $L_0$ is $q$-convex for some $q\geq 0$, then for $B$ sufficiently small we have 
 \begin{align}\label{16:43:18}
E_\ve(\varphi) \leq C Q_\ve(\varphi) 
 \end{align}  
for all $\varphi\in \text{Dom}(d^\ast_{L_\ve})\cap C^\infty(\Lambda^qL^\ast\otimes V)$ and all $\ve\in B$. Here $C$ is independent of $\ve$.
\end{prop}
\begin{proof} 
If we can prove that every point $x\in M$ has a neighbourhood $U$ such that (\ref{16:43:18}) holds for $\varphi\in \text{Dom}(d^\ast_{L_\ve})\cap C^\infty_c(U,\Lambda^qL^\ast\otimes V)$ and $B$ sufficiently small, then compactness of $M$ together with a partition of unity argument imply the desired result. If $x$ lies in the interior, we can simply use Thm.\ref{13:13:59} together with $E_\ve(\varphi)\leq C ||\varphi||_1$. Otherwise, if $x\in \del M$ we distinguish between the cases where $x$ is elliptic or non-elliptic for $L_0$.  

Suppose first that $x$ is elliptic for $L_0$. Using the notation of Lem.\ref{13:08:22}, we claim that 
 \begin{align}\label{14:41:39}
\int_{\del M}|\varphi|^2\leq C(\sum_{i,I,\mu}  ||v_{i,\ve}\cdot \varphi_{I,\ve}^\mu||^2+||\varphi||^2).
 \end{align} 
Indeed, in a boundary chart $(\R^m_-,(t^1,\ldots, t^{m-1},r))$ we use Stokes' theorem to obtain 
\begin{align*}
\int_{\del M}|\varphi|^2=\int_{\R^{m-1}} |\varphi|^2 \sqrt{g}dt=\int_{\R^{m-1}}\int_{-\infty}^0 \tfrac{d}{dr}(|\varphi|^2 \sqrt{g})drdt=\int_M 2\text{Re}\langle \tfrac{d}{dr}\varphi,\varphi\rangle+\mcO(||\varphi||^2).  
\end{align*}
Using that $x$ is elliptic, we expand (as before) $\tfrac{d}{dr}=\sum_i a^i_\ve v_{i,\ve}+b^i_\ve\overline{v_{i,\ve}}$ with $b^l|_{\del M}=0$. Applying integration by parts and Cauchy-Schwartz yields (\ref{14:41:39}). Combining this with (\ref{14:44:23}) gives 
\begin{align*}
E_{\ve}(\varphi)^2\leq &C \big(\sum_{I,\mu,i} || v_{i,\ve}\cdot \varphi_{I,\ve}^\mu||^2  + \int_{\del M} |\varphi|^2+ ||\varphi||^2\big)\leq C(\sum_{i,I,\mu}  ||v_{i,\ve}\cdot \varphi_{I,\ve}^\mu||^2+||\varphi||^2)\\
\leq & CQ_\ve(\varphi,\varphi)+\mcO(E_\ve(\varphi)\cdot ||\varphi||)\leq C Q_\ve(\varphi,\varphi)+C\delta E_\ve(\varphi)^2 + C\tfrac{1}\delta ||\varphi||^2,
\end{align*} 
where $\delta>0$ can be chosen arbitrarily. For $\delta<1/C$ we obtain the desired estimate.

Now suppose that $x$ is non-elliptic for $L_0$. Let $\zeta\in [0,1]$ be a real constant that will be determined later, and use Lem.\ref{13:08:22}ii) to estimate 
\begin{align*}
Q_\ve(\varphi,\varphi) \geq &\zeta \sum_{I,\mu,i} || v_{i,\ve}\cdot \varphi^\mu_{I,\ve}||^2+(1-\zeta) \sum_{I,\mu,i;\lambda_{i}<0} ||v_{i,\ve}\cdot \varphi^\mu_{I,\ve}||^2+\sum_{I,\mu}\sum_{i\in I} \lambda_{i} \int_{\del M}  \gamma_\ve|\varphi^\mu_{I,\ve}|^2\\ &-\delta E_\ve (\varphi)^2 -\mathcal{O}(E_\ve(\varphi)\cdot ||\varphi||)\\
\geq &\zeta \sum_{I,\mu,i} || v_{i,\ve}\cdot \varphi^\mu_{I,\ve}||^2 + \sum_{I,\mu} \Big(	(1-\zeta)\sum_{i;\lambda_{i}<0} (-\lambda_{i})	+\sum_{i\in I}\lambda_{i}	\Big)\int_{\del M} \gamma_\ve|\varphi^\mu_{I,\ve}|^2\\&-2\delta E_\ve (\varphi)^2 -\mathcal{O}(E_\ve (\varphi)\cdot||\varphi||).
\end{align*} 
We want to estimate the boundary term from below by $\zeta \int_{\del M} |\varphi^\mu_{I,\ve}|^2$. From the definition of $\gamma_\ve$ (see (\ref{15:09:20})) we see that there is a uniform bound $\gamma_\ve\geq \eta$ for some $\eta>0$, so it 
suffices to show that for each index set $I$ of size $q$ we can obtain an estimate of the form 
\begin{align}\label{17:52:22}
\frac{\zeta}{\eta}\leq (1-\zeta)\sum_{i;\lambda_{i}<0} (-\lambda_{i})	+\sum_{i\in I}\lambda_{i} = \zeta\sum_{i;\lambda_{i}<0} \lambda_{i}+\sum_{i\notin I;\lambda_{i}<0} (-\lambda_{i})+\sum_{i\in I;\lambda_{i}>0} \lambda_{i}.
\end{align}
If $l\in I$ then $\varphi^\mu_{I,\ve}=0$ on $\del M$ and there is nothing to prove, so assume that $l\notin I$. If $l-q$ of the $\lambda_{i}$ are positive, then there must be an $i_0\in I$ with $\lambda_{i_0}>0$. In this case we choose $\zeta$ small enough so that $\lambda_{i_0}\geq \zeta(\tfrac{1}\eta-\sum_{i;\lambda_{i}<0} \lambda_{i})$, which implies the above estimate. Alternatively, if $q+1$ of the $\lambda_{i}$ are negative then there is an $i_1\notin I$ with $\lambda_{i_1}<0$. In this case we take $\zeta$ small enough so that $-\lambda_{i_1}\geq \zeta(\tfrac{1}\eta-\sum_{i;\lambda_{i}<0} \lambda_{i})$, which also implies the above estimate. Hence, we can choose $\zeta$ small enough relative to the eigenvalues of the Levi-form of $L_0$ at $x$ so that (\ref{17:52:22}) holds for every index set $I$ of size $q$. In particular,
\begin{align*}
Q_\ve(\varphi,\varphi)\geq \zeta E(\varphi)^2-\zeta ||\varphi||^2-2\delta E_\ve(\varphi)^2-C\delta' E_\ve(\varphi)^2-\tfrac{C}{\delta'} ||\varphi||^2,
\end{align*}
where $\delta'>0$ can be chosen arbitrarily small. Taking both $\delta$ and $\delta'$ sufficiently small compared to $\zeta$ (this also requires shrinking the chart $U$) yields $E_\ve(\varphi)\leq CQ_\ve(\varphi)$. 
\end{proof} 

\noindent \textsl{Proof of Thm.\ref{08:54:49}.} Consider the family $P_\ve:=d_{L_\ve}$ of differential operators on the bundle $E=\oplus_s \Lambda^sL^\ast \otimes V$. We start by showing that conditions 1)-3) of Thm.\ref{13:30:27} are satisfied for $P_\ve$ in degree $q$ (c.f.\ Rem.\ref{14:08:00}). The first condition is exactly the ellipticity condition on $L_0$, which is an open condition and so will be satisfied for $B$ sufficiently small. For condition 2), observe that $\sigma(P_{\ve,f}^\ast,dr)=\rho_\ve^\ast(dr)\wedge$, where $\rho_\ve:L\rightarrow TM_\C$ denotes the anchor. Ellipticity of $L_\ve$ guarantees that $\rho_\ve^\ast(dr)$ is nowhere zero along $\del M$, which implies that the kernel has constant rank (for all $x\in \del M$ and $\ve\in B$). Finally, condition 3) follows immediately by combining Prop.\ref{09:06:25} and Prop.\ref{14:03:13}, and we conclude that $d_{L_\ve}$ satisfied elliptic regularity in degree $q$. The five statements of Thm.\ref{08:54:49} follow from the results of sections \ref{21:39:11} and \ref{10:54:51}. Specifically; 
\begin{itemize}
\item[1)] Follows from Lem.\ref{13:27:34} and the definition of elliptic regularity for families (c.f\ (\ref{16:21:04})). 
\item[2)] Follows from Cor.\ref{07:54:08}. 
\item[3)] Follows from Thm.\ref{08:47:18}.
\item[4)] Follows from Prop.\ref{08:17:09}.
\item[5)] Follows from Cor.\ref{14:27:15}, Prop.\ref{11:56:27} and Rem.\ref{14:27:58}. \qed
\end{itemize}

\appendix


\section{Appendix}

In this appendix we collect some \lq\lq Leibniz\rq\rq\ rules involving Sobolev norms. We will use the notation of Sect.\ref{12:54:47}, and start with some numerical estimates.  


\begin{lem} \label{09:41:11}
Let $k\in \R$ be a real number.
\begin{itemize}

\item[i)] For all $\xi,\eta\in\R^m$ we have $\Big(\frac{1+|\xi|^2}{1+|\eta|^2}\Big)^k  \leq 2^{|k|} (1+|\xi-\eta|^2)^{|k|}$. 

\item[ii)] Define $K_1(\xi,\eta):= (1+|\xi|^2)^{k/2}-(1+|\eta|^2)^{k/2}$, where $\xi,\eta\in\R^m$. Then
\begin{align*}
|K_1(\xi,\eta)| &\leq |k|\cdot |\xi-\eta| \cdot \big((1+|\xi|^2)^{\frac{k-1}{2}}+(1+|\eta|^2)^{\frac{k-1}{2}}\big).
\end{align*}

\item[iii)] Define $K_3(\xi,\eta_1,\eta_2):= (1+|\xi|^2)^{\frac{k}{2}}+(1+|\eta_1|^2)^{\frac{k}{2}}-(1+|\eta_2|^2)^{\frac{k}{2}}-(1+|\xi+\eta_1-\eta_2|^2)^{\frac{k}{2}}$, where $\xi,\eta_1,\eta_2\in\R^m$. Setting $C:=k(k-1)$, we have
\begin{align*}
|K_3(\xi,\eta_1,\eta_2)|\leq C |\xi-\eta_2| \cdot |\eta_1-\eta_2|\cdot  \int_0^1 \int_0^1 (1+\big|\xi+t(\eta_1-\eta_2)+t'(\eta_2-\xi)\big|^2)^{\frac{k-2}{2}} dtdt'.	
\end{align*}

\end{itemize}
\end{lem}
\begin{proof}
i): See \cite[Lem.(A.1.3)]{MR0461588}. 

ii): Consider the smooth function $f(x):=(1+x^2)^{k/2}$ on $\R_{\geq 0}$. For all $x,y\in \R_{\geq 0}$ we have
\begin{align*}
|f(x)-f(y)|\leq |x-y|\sup_{x\leq z \leq y} |f'(z)| 
\leq |k| |x-y| \big((1+x^2)^{\frac{k-1}{2}}+(1+y^2)^{\frac{k-1}{2}}\big). 
\end{align*}
Setting $x=|\xi|$, $y=|\eta|$ and using $| |\xi|- |\eta| |\leq |\xi-\eta|$, we obtain ii).

iii): Define $g(x):=(1+|x|^2)^{k/2}$ so that $K_3(\xi,\eta_1,\eta_2)=g(\xi)+g(\eta_1)-g(\eta_2)-g(\xi+\eta_1-\eta_2)$. A couple of applications of the fundamental theorem of calculus gives
\begin{align*}
K_3(\xi,\eta_1,\eta_2)=\sum_{i,j}  (\eta_2-\xi)^i(\eta_1-\eta_2)^j \int_0^1\int_0^1 \del_i\del_jg(\xi+t'(\eta_2-\xi)+t(\eta_1-\eta_2))dtdt', 
\end{align*}
from which the desired estimate readily follows.
\end{proof}
Below we will write $\alpha \lesssim \beta$ if there is a constant $C$ such that $\alpha \leq C \beta$. The only constraint on $C$ is that it is independent of the functions $f$, $g$ and $\varphi$ appearing in the inequalities below. 
   
\begin{prop} \label{17:24:46} Set $a:=1+\frac{m}{2}$, where $m$ is the dimension of Euclidean space.
\
\newline
\noindent i) For $s\geq 0$ and $f,\varphi\in \mathcal{S}$ we have 
\begin{align}\label{16:10:17}
|| f\varphi||_s    \lesssim||f||_{s+a} ||\varphi||+||f||_{a} ||\varphi||_{s}.
\end{align}

For $s\in \R$ and $f,\varphi\in \mathcal{S}$ we have
\begin{align}
|| f\varphi||_s    \lesssim ||f||_{|s|+a} ||\varphi||_s.
\end{align}

\noindent ii) For $s\in \R_{\geq 0}$, $k\in \R_{\geq 1}$ and $f,\varphi\in \mathcal{S}$ we have  
\begin{align}
|| [\Lambda^k,f]\varphi||_s    \lesssim  &||f||_{s+k+a} ||\varphi||+||f||_{k+a} ||\varphi||_{s}+||f||_{s+1+a} ||\varphi||_{k-1}+||f||_{1+a} ||\varphi||_{s+k-1}.\label{13:49:50}
\end{align}

For $s,k\in \R$ and $f,\varphi\in \mathcal{S}$ we have
\begin{align}\label{14:01:04}
|| [\Lambda^k,f]\varphi||_s   & \lesssim (||f||_{|s+k-1|+1+a} +||f||_{|s|+1+a}) ||\varphi||_{s+k-1}.
\end{align}
\noindent iii) For $s\in \R_{\geq 0}$, $k\in \R_{\geq 1}$ and $f,\varphi\in \mathcal{S}$ we have 
\begin{align}
|| [\Lambda^k,[\Lambda^k,f]] \varphi||_s   \lesssim & ||f||_{2k+s+a} ||\varphi||+||f||_{2k+a}||\varphi||_{s}\nonumber \\&	+||f||_{2+a}||\varphi||_{2k+s-2}	+||f||_{s+2+a}||\varphi||_{2k-2}.\label{14:14:05}
\end{align}

For $s,k\in \R$ and $f,\varphi\in \mathcal{S}$ we have 
\begin{align}\label{14:14:14}
|| [\Lambda^k,[\Lambda^k,f]] \varphi||_s   & \lesssim (||f||_{|s+2k-2|+2+a}+||f||_{|s|+2+a} )|| \varphi ||_{s+2k-2}.
\end{align}

\noindent iv) For $s\in \R_{\geq 0}$, $k\in \R_{\geq 2}$ and $f,g,\varphi\in \mathcal{S}$ we have 
\begin{align}\label{18:03:18}
|| [[\Lambda^k,f],&g]\varphi||_s  \lesssim  \big( ||f||_{k-1+s+a} ||g||_{1+a}+||f||_{k-1+a} ||g||_{s+1+a}+||f||_{s+1+a} ||g||_{k-1+a}\nonumber\\
&+||f||_{1+a} ||g||_{k-1+s+a}\big)||\varphi||_{}+(||f||_{k-1+a} ||g||_{1+a}+||f||_{1+a} ||g||_{k-1+a})||\varphi||_{s}			\nonumber\\ 
&+ (||f||_{s+1+a} ||g||_{1+a}+||f||_{1+a} ||g||_{s+1+a})||\varphi||_{k-2}+||f||_{1+a} ||g||_{1+a}||\varphi||_{k-2+s}.
\end{align}

For $s,k\in \R$ and $f,g,\varphi\in \mathcal{S}$ we have
\begin{align}\label{18:18:02}
|| [[\Lambda^k,f],g]\varphi||_s  \lesssim \big(&||f||_{1+|s|+|k-2|+a} ||g||_{1+a}+||f||_{1+|s|+a} ||g||_{1+|k-2|+a}  \\ &+||f||_{1+|k-2|+a} ||g||_{1+|s|+a} +||f||_{1+a} ||g||_{1+|s|+|k-2|+a}  \big) ||\varphi||_{s+k-2}.\nonumber
\end{align}
\end{prop}

\begin{proof}
$i)$: 
Using $\widehat{f\varphi}=\widehat{f}\star \widehat{\varphi}$, where $\star$ denotes convolution product, we obtain
\begin{align*}
|| f\varphi||^2_s & = \int (1+|\xi|^2)^s |\widehat{f\varphi}(\xi)|^2 d\xi \leq \int \Big( \int (1+|\xi|^2)^{s/2}|\widehat{f}(\xi-\eta)| |\widehat{\varphi}(\eta)|d\eta \Big)^2 d\xi.
\end{align*}
If $s\geq 0$ we can use the triangle inequality (raised to a non-negative power) to estimate 
\begin{align*}
(1+|\xi|^2)^{s/2}\lesssim (1+|\xi-\eta|^2)^{s/2}+ (1+|\eta|^2)^{s/2}.
\end{align*}
Using the Cauchy-Schwarz inequality we get
\begin{align*}
 \int (1+|\eta|^2)^{s/2}|\widehat{f}(\xi-\eta)| |\widehat{\varphi}(\eta)|d\eta \leq   \Big( \int |\widehat{f}(\xi-\eta)| d\eta \Big)^{\frac{1}{2}} \cdot  \Big( \int (1+|\eta|^2)^{s}|\widehat{f}(\xi-\eta)| |\widehat{\varphi}(\eta)|^2d\eta \Big)^{\frac{1}{2}},
\end{align*}
which implies that 
\begin{align*}
\int \Big( \int (1+|\eta|^2)^{s/2}|\widehat{f}(\xi-\eta)| |\widehat{\varphi}(\eta)|d\eta \Big)^2 d\xi \leq  \Big( \int |\widehat{f}(\eta)| d\eta\Big)^2 \cdot ||\varphi||_s^2\lesssim ||f||^2_a ||\varphi||_s^2.
\end{align*}
Here in the last step we applied the Cauchy-Schwarz inequality again to estimate
\begin{align*}
\Big(\int |\widehat{f}(\eta)| d\eta\Big)^2 \leq \int (1+|\eta|^2)^{-a} d\eta \int (1+|\eta|^2)^{a} |\widehat{f}(\eta)|^2 d\eta \leq C ||f||^2_a
\end{align*}
where $C$ is finite because $a>m/2$. In a similar fashion one proves that 
\begin{align*}
\int \Big( \int (1+|\xi-\eta|^2)^{s/2}|\widehat{f}(\xi-\eta)| |\widehat{\varphi}(\eta)|d\eta \Big)^2 d\xi \lesssim ||f||_{s+a} ||\varphi||^2,
\end{align*}
proving (\ref{16:10:17}). For arbitrary $s\in \R$ we have to replace the triangle inequality with the bound 
$$(1+|\xi|^2)^{s/2}\lesssim (1+|\xi-\eta|^2)^{|s|/2}(1+|\eta|^2)^{s/2}		$$
which follows from Lem.\ref{09:41:11}i). The rest of the steps are then similar to the ones above.

$ii)$: The strategy is the same as in $i)$. First we observe that 
\begin{align*}
\mathcal{F}([\Lambda^k,f]\varphi)(\xi)= \int K_1(\xi,\eta) \widehat{f}(\xi-\eta) \widehat{\varphi}(\eta) d\eta,
\end{align*}
where $K_1$ 
was defined in Lem.\ref{09:41:11}ii). 
When $k-1\geq 0$ and $s\geq 0$, the triangle inequality together with Lem.\ref{09:41:11}ii) imply that
\begin{align*}
(1+|\xi|^2)^{s/2}|K_1(\xi,\eta)|&\lesssim  (1+|\xi-\eta|^2)^{\frac{s+k}{2}} +(1+|\xi-\eta|^2)^{\frac{k}{2}}(1+|\eta|^2)^{\frac{s}{2}} \\
&+ (1+|\xi-\eta|^2)^{\frac{s+1}{2}}(1+|\eta|^2)^{\frac{k-1}{2}}+(1+|\xi-\eta|^2)^{\frac{1}{2}}(1+|\eta|^2)^{\frac{s+k-1}{2}}.
\end{align*}
The same steps as in $i)$ then yield (\ref{13:49:50}).
For arbitrary $s$ and $k$ we use Lem.\ref{09:41:11}ii) to estimate 
\begin{align*}
(1+|\xi|^2)^{s/2}|K_1(\xi,\eta)|&\lesssim \big( (1+|\xi-\eta|^2)^{\frac{1+|s+k-1|}{2}}+(1+|\xi-\eta|^2)^{\frac{1+|s|}{2}}\big)(1+|\eta|^2)^{\frac{k+s-1}{2}},
\end{align*}
giving (\ref{14:01:04}).

$iii)$: We have 
\begin{align*}
\mathcal{F}([\Lambda^k,[\Lambda^k,f]]\varphi)(\xi)= \int K_2(\xi,\eta) \widehat{f}(\xi-\eta) \widehat{\varphi}(\eta) d\eta,
\end{align*}
where $K_2(\xi,\eta):= \big( (1+|\xi|^2)^{k/2}-(1+|\eta|^2)^{k/2} 	\big)^2=K_1(\xi,\eta)^2$. Using Lem.\ref{09:41:11}ii) we obtain
\begin{align*}
(1+|\xi|^2)^{s/2}|K_2(\xi,\eta)|\lesssim & \ (1+|\xi-\eta|^2)^{\frac{s}{2}+k} +(1+|\xi-\eta|^2)^{\frac{s}{2}+1}(1+|\eta|^2)^{k-1} \\
&+ (1+|\xi-\eta|^2)^{k}(1+|\eta|^2)^{\frac{s}{2}}+(1+|\xi-\eta|^2)^{}(1+|\eta|^2)^{\frac{s}{2}+k-1},
\end{align*}
valid for $s\geq 0$ and $k\geq 1$. The same steps as in i) then give (\ref{14:14:05}). For arbitrary $s$ and $k$ we use Lem.\ref{09:41:11}i) to estimate 
\begin{align*}
(1+|\xi|^2)^{s/2}|K_2(\xi,\eta)|&\lesssim  \big( (1+|\xi-\eta|^2)^{|\frac{s}{2}+k-1|+1}+(1+|\xi-\eta|^2)^{|\frac{s}{2}|+1}\big)(1+|\eta|^2)^{\frac{s}{2}+k-1},
\end{align*}
which gives (\ref{14:14:14}).

$iv)$: We have
\begin{align*}
\mathcal{F}([[\Lambda^k,f],g]\varphi)(\xi)= \iint  K_3(\xi,\eta_1,\eta_2) \widehat{f}(\xi-\eta_2) \widehat{g}(\eta_2-\eta_1) \widehat{\varphi}(\eta_1) d\eta_1d\eta_2,
\end{align*}
where $K_3$ was defined in Lem.\ref{09:41:11}iv). Consequently, for $s\geq 0$ and $k\geq 2$ we have
\begin{align*}
(1+|\xi|^2)^{\frac{s}{2}}|K_3(\xi,\eta&)|\lesssim   (1+|\xi-\eta_2|^2)^{\frac{1}{2}}(1+|\eta_1-\eta_2|^2)^{\frac{1}{2}}
 \Big( (1+|\xi-\eta_2|^2)^{\frac{s}{2}} + (1+|\eta_1-\eta_2|^2)^{\frac{s}{2}}\\&+(1+|\eta_1|^2)^{\frac{s}{2}}\Big)\Big((1+|\xi-\eta_2|^2)^{\frac{k-2}{2}} + (1+|\eta_1-\eta_2|^2)^{\frac{k-2}{2}}+(1+|\eta_1|^2)^{\frac{k-2}{2}}\Big). 
\end{align*} 
Expanding this out gives nine terms, leading to (\ref{18:03:18}) by using the same steps as in $i)$. 

For arbitrary $s$ and $k$ we use Lem.\ref{09:41:11} ii) to bound
\begin{align*}
(1+\big|\xi+t(\eta_1-\eta_2)+t'(\eta_2-\xi)\big|^2)&^{\frac{k-2}{2}}=\Big(\frac{1+\big|\xi+t(\eta_1-\eta_2)+t'(\eta_2-\xi)\big|^2}{1+|\eta_1|^2}\Big)^{\frac{k-2}{2}}  (1+|\eta_1|^2)^{\frac{k-2}{2}}\\
\lesssim & (1+\big|\xi-\eta_1+t(\eta_1-\eta_2)+t'(\eta_2-\xi)\big|^2)^{|\frac{k-2}{2}|}(1+|\eta_1|^2)^{\frac{k-2}{2}}\\
\lesssim & \big((1+|\xi-\eta_2|^2)^{|\frac{k-2}{2}|}+(1+|\eta_1-\eta_2|^2)^{|\frac{k-2}{2}|}\big)(1+|\eta_1|^2)^{\frac{k-2}{2}},
\end{align*}
where we used $\xi-\eta_1+t(\eta_1-\eta_2)+t'(\eta_2-\xi)=(1-t')(\xi-\eta_2)+(1-t)(\eta_2-\eta_1)$. Proceeding as before yields (\ref{18:18:02}).
\end{proof}
\noindent There is also a boundary version of Prop.\ref{17:24:46}. 
We continue with the notation of Sect.\ref{12:54:47}, and denote by $\lceil s \rceil$ the smallest integer greater or equal than $s$. For $K\subset \R^m_-$ we denote by $C^\infty_K(\R^m_-)  \subset \mathcal{S}(\R^m_-)$ the functions with support in $K$. 

\begin{prop} \label{11:36:52} Let $K\subset \R^m_-$ be a compact subset and let $a=1+\frac{m}{2}$.
\
\newline
\noindent i) For $s\geq 0$ and $f\in C^\infty_K(\R^m_-)$, $\varphi\in \mathcal{S}(\R^m_-)$ we have
\begin{align}
|| f\varphi||_{\del,s}    \lesssim |f|_{\lceil s+a\rceil } ||\varphi||_\del+|f|_{\lceil a\rceil} ||\varphi||_{\del,s}.
\end{align}

For $s\in \R$ and $f\in C^\infty_K(\R^m_-)$, $\varphi\in \mathcal{S}(\R^m_-)$ we have
\begin{align}
|| f\varphi||_{\del,s}    \lesssim |f|_{\lceil |s|+a\rceil} ||\varphi||_{\del,s}.
\end{align}
\noindent ii) For $s\in \R_{\geq 0}, k\in \R_{\geq 1}$ and $f\in C^\infty_K(\R^m_-) ,\varphi\in \mathcal{S}(\R^m_-)$ we have
\begin{align}\label{11:44:35}
|| [\Lambda^k_\del,f]\varphi||_{\del,s}    \lesssim  & \ |f|_{\lceil s+k+a\rceil} ||\varphi||_\del+|f|_{\lceil k+a\rceil} ||\varphi||_{\del,s}\nonumber \\&+|f|_{\lceil s+1+a\rceil} ||\varphi||_{\del,k-1}+|f|_{\lceil 1+a\rceil} ||\varphi||_{\del,s+k-1}.
\end{align}

For $s,k\in \R$ and $f\in  C^\infty_K(\R^m_-),\varphi\in \mathcal{S}(\R^m_-)$ we have
\begin{align}\label{11:44:44}
|| [\Lambda^k_\del,f]\varphi||_{\del,s}   & \lesssim (|f|_{\lceil |s+k-1|+1+a\rceil} +|f|_{\lceil |s|+1+a\rceil}) ||\varphi||_{\del,s+k-1}.
\end{align}

\noindent iii) For $s\in \R_{\geq 0}$, $k\in \R_{\geq 1}$ and $f\in C^\infty_K(\R^m_-), \varphi\in \mathcal{S}(\R^m_-)$ we have
\begin{align}\label{11:44:51}
|| [\Lambda^k_\del,[\Lambda^k_\del,f]] \varphi||_{\del,s}   \lesssim & \ |f|^2_{\lceil 2k+s+a\rceil} ||\varphi||^2_\del	+|f|^2_{\lceil 2+a\rceil }||\varphi||^2_{\del,2k+s-2}	\nonumber \\&+|f|_{\lceil 2k+a\rceil}||\varphi||_{\del,s}+|f|^2_{\lceil s+2+a\rceil}||\varphi||^2_{\del,2k-2}.
\end{align}

For $s,k\in \R$ and $f\in C^\infty_K(\R^m_-), \varphi\in \mathcal{S}(\R^m_-)$ we have 
\begin{align}\label{11:44:59}
|| [\Lambda^k_\del,[\Lambda^k_\del,f]] \varphi||_{\del,s}   & \lesssim (|f|_{\lceil |s+2k-2|+2+a\rceil}+|f|_{\lceil |s|+2+a\rceil} )|| \varphi ||_{\del,s+2k-2}.
\end{align}

\noindent iv) For $s\in \R_{\geq 0},k\in \R_{\geq 2}$ and $f,g\in C^\infty_K(\R^m_-), \varphi\in \mathcal{S}(\R^m_-)$ we have
\begin{align}\label{11:45:07}
|| &[[\Lambda^k_\del,f],g]\varphi||_{\del,s}  \lesssim  \big( |f|_{\lceil k-1+s+a\rceil} |g|_{\lceil 1+a\rceil}+|f|_{\lceil k-1+a\rceil} |g|_{\lceil s+1+a\rceil}+|f|_{\lceil s+1+a\rceil} |g|_{\lceil k-1+a\rceil}\nonumber\\
&+|f|_{\lceil 1+a\rceil} |g|_{\lceil k-1+s+a\rceil}\big)||\varphi||_{\del}+(|f|_{\lceil k-1+a\rceil} |g|_{\lceil 1+a\rceil}+|f|_{\lceil 1+a\rceil} |g|_{\lceil k-1+a\rceil})||\varphi||_{\del,s}			\nonumber\\ 
&+ (|f|_{\lceil s+1+a\rceil} |g|_{\lceil 1+a\rceil}+|f|_{\lceil 1+a\rceil} |g|_{\lceil s+1+a\rceil})||\varphi||_{\del,k-2}+|f|_{\lceil 1+a\rceil} |g|_{\lceil 1+a\rceil}||\varphi||_{\del,k-2+s}.
\end{align}

For $s,k\in \R$ and $f,g\in C^\infty_K(\R^m_-), \varphi\in \mathcal{S}(\R^m_-)$ we have
\begin{align}\label{11:45:14}
|| [[\Lambda^k_\del,f],g]\varphi||_{\del,s}  \lesssim \big(&|f|_{\lceil 1+|s|+|k-2|+a\rceil} |g|_{1+a}+|f|_{\lceil 1+|s|+a\rceil} |g|_{\lceil 1+|k-2|+a\rceil}  \\ &+|f|_{\lceil 1+|k-2|+a\rceil} |g|_{\lceil 1+|s|+a\rceil} +|f|_{\lceil 1+a\rceil } |g|_{\lceil 1+|s|+|k-2|+a\rceil}  \big) ||\varphi||_{\del,s+k-2}.\nonumber
\end{align}
\end{prop}
\begin{proof}
For each value of $r$ we can consider the inequalities of Prop.\ref{17:24:46} for the functions $f(\cdot,r)$,  $g(\cdot,r)$ and $\varphi(\cdot,r)$ on $\R^{m-1}$. Integrating these over $r$ and using $||f(\cdot,r) ||_{s}\leq C |f(\cdot,r)|_{\lceil s \rceil}\leq C |f|_{\lceil s \rceil}$ for $s\in \R_{\geq 0}$ (where $C$ depends on $K\subset \R^m_-$), we obtain (\ref{11:44:35})-(\ref{11:45:14}). 
\end{proof}

Finally we mention a version of Rellich's lemma for the norm $||\text{D}(\cdot)||_{\del,-1/2}$, whose proof is outlined in the appendix of \cite{MR0461588} (see the discussion preceding Prop.A.3.1). 
\begin{prop}\label{18:03:02}
If a sequence $\varphi_j\in C^\infty_K(\R^m_-)$ satisfies a uniform bound $||\text{D}\varphi_j||_{\del,-1/2}\leq C$, then a subsequence converges in $L^2(\R^m)$. 
\end{prop}

\bibliographystyle{hyperamsplainnodash}
\addcontentsline{toc}{section}{References}
\bibliography{references}

\end{document}